\documentclass[a4paper]{amsart}

\usepackage{fix-cm}
\usepackage[final]{microtype}
\usepackage[dvipsnames,svgnames,x11names,hyperref]{xcolor}
\usepackage{dsfont,url,graphicx,verbatim,amssymb,enumerate,stmaryrd,booktabs,lmodern,mathtools,mathabx,nicefrac}
\SetSymbolFont{stmry}{bold}{U}{stmry}{m}{n}
\usepackage[pagebackref,colorlinks,citecolor=Mahogany,linkcolor=Mahogany,urlcolor=Mahogany,filecolor=Mahogany]{hyperref}
\usepackage[capitalize]{cleveref}
\usepackage[mathscr]{euscript}
\usepackage[margin=1.33in]{geometry}
\usepackage{tikz,tikz-cd}
\usetikzlibrary{matrix,calc,positioning,arrows,decorations.pathreplacing,patterns,arrows,patterns.meta}
\tikzset{
  snake left/.style={
    rounded corners,
    to path={
      let \p1 = (\tikztostart.east),
          \p2 = (\tikztotarget.west),
          \p3 = ($(\p1)!0.5!(\p2)$),
          \n1 = {8pt} 
      in
      (\p1)
      -- (\x1 + \n1, \y1)
      -- (\x1 + \n1, \y3)
      -- (\x2 - \n1, \y3) \tikztonodes
      -- (\x2 - \n1, \y2)
      -- (\p2)
    }
  }
}

\newtheorem{theorem}{Theorem}[section]
\newtheorem*{theorem*}{Theorem}
\newtheorem{lemma}[theorem]{Lemma}
\newtheorem{proposition}[theorem]{Proposition}
\newtheorem{corollary}[theorem]{Corollary}
\newtheorem*{corollary*}{Corollary}

\newtheorem{atheorem}{Theorem}

\newtheorem{innercustomgeneric}{\customgenericname}
\providecommand{\customgenericname}{}
\newcommand{\newcustomtheorem}[2]{%
  \newenvironment{#1}[1]
  {%
   \ifdefined\crefalias\crefalias{innercustomgeneric}{#2}\fi
   \renewcommand\customgenericname{#2}%
   \renewcommand\theinnercustomgeneric{##1}%
   \innercustomgeneric
  }
  {\endinnercustomgeneric}%
  \ifdefined\crefname\crefname{#2}{#2}{#2s}\fi
}

\newcustomtheorem{customthm}{Theorem}
\newtheorem{acorollary}[atheorem]{Corollary}

\newcustomtheorem{customconj}{Conjecture}
\theoremstyle{definition}
\newtheorem{definition}[theorem]{Definition}
\newtheorem*{definition*}{Definition}

\newtheorem{notation}[theorem]{Notation}

\theoremstyle{remark}
\newtheorem{example}[theorem]{Example}
\newtheorem*{example*}{Example}
\newtheorem*{remark*}{Remark}

\newtheorem{remark}[theorem]{Remark}

\renewcommand{\sf}[1]{{\mathsf{#1}}}
\newcommand{\scr}[1]{{\mathscr{#1}}}
\renewcommand{\rm}[1]{{\mathrm{#1}}}
\newcommand{\lra}{\longrightarrow}

\newcommand{\ul}[1]{{\underline{#1}}}
\newcommand{\ol}[1]{{\overline{#1}}}

\DeclareMathOperator*{\colim}{colim}
\newcommand{\ds}[1]{{\mathds{#1}}}
\newcommand{\fr}[1]{{\mathfrak{#1}}}
\newcommand{\bb}[1]{{\mathds{#1}}}
\newcommand{\cal}[1]{{\mathcal{#1}}}
\renewcommand{\bf}[1]{{\mathbf{#1}}}

\makeatletter
\newcommand{\superimpose}[2]{{%
  \ooalign{%
    \hfil$\m@th#1\@firstoftwo#2$\hfil\cr
    \hfil$\m@th#1\@secondoftwo#2$\hfil\cr
  }%
}}
\makeatother

\newcommand{\Fun}{\rm{Fun}}
\newcommand{\Map}{\rm{Map}}

\newcommand{\Sp}{\scr{S}\rm{p}}
\newcommand{\aug}{\rm{aug}}

\newcommand{\St}{\rm{St}}
\newcommand{\SSt}{{\scr{S}\rm{t}}}

\newcommand{\PGL}{\rm{PGL}}
\newcommand{\GL}{\rm{GL}}

\newcommand{\Spc}{{\scr{S}\rm{pc}}}

\newcommand{\alg}{\rm{alg}}
\newcommand{\Alg}{\rm{Alg}}

\newcommand{\coAlg}{\rm{coAlg}}
\newcommand{\Q}{\ds{Q}}
\newcommand{\N}{\ds{N}}
\newcommand{\Z}{\ds{Z}}
\newcommand{\C}{\ds{C}}
\newcommand{\R}{\ds{R}}
\newcommand{\heart}{\ensuremath\heartsuit}
\renewcommand{\cot}{\rm{cot}}
\newcommand{\triv}{\rm{triv}}
\newcommand{\prim}{\rm{prim}}
\newcommand{\fgt}{\rm{fgt}}
\newcommand{\free}{\rm{free}}
\newcommand{\cofree}{\rm{cofree}}
\newcommand{\fil}{\rm{fil}}
\newcommand{\gc}{\rm{gc}}

\newcommand{\indec}{\rm{indec}}

\newcommand{\gr}{\rm{gr}}

\newcommand{\coLie}{\rm{\rm{coLie}}}

\newcommand{\SL}{\rm{SL}}
\newcommand{\BGL}{\rm{BGL}}

\newcommand{\BGLb}{\bf{BGL}}

\newcommand{\para}{\boxbackslash}

\newcommand{\Vect}{\rm{Vect}}

\newcommand{\id}{\rm{id}}
\newcommand{\inc}{\rm{inc}}
\newcommand{\pr}{\rm{pr}}

\newcommand{\rkfil}{\rm{rkfil}}
\newcommand{\slopetwo}{{2{\bullet}{\shortminus}1}}

\newcommand{\StH}{\rm{St}^{2}}
\newcommand{\StL}{\rm{St}^{\infty}}

\newcommand{\FC}{\mathrm{FC}} 
\newcommand{\FI}{\mathrm{FI}} 
\newcommand{\FL}{\rm{FL}} 
\newcommand{\FIR}{\mathrm{FIR}} 
\newcommand{\FLR}{\rm{FLR}} 
\newcommand{\compactldots}{\mathinner{\ldotp\mkern-2mu\ldotp\mkern-2mu\ldotp}}
\newcommand{\compactcdots}{\mathinner{\cdotp\mkern-2mu\cdotp\mkern-2mu\cdotp}}

\newcommand{\PolyL}{\mathscr{G}} 
\newcommand{\PolyLF}{\scr{L}^\rm{f}} 
\newcommand{\MotCoLie}{\scr{L}^\rm{MTM}} 

\newcommand{\Gr}{\mathrm{Gr}}
\newcommand{\Conf}{\mathrm{Conf}}

\newcommand{\CorG}{\mathrm{Cor}^\PolyL} 
 
\newcommand{\CorH}{\mathrm{Cor}_\rm{Hod}} 
\newcommand{\ItG}{\mathrm{I}^\PolyL}

\newcommand{\LiG}{\mathrm{Li}^\PolyL}
\newcommand{\GLiG}{\mathrm{GLi}^\PolyL}

\newcommand{\dpw}{\rm{dpw}}
\newcommand{\nil}{\rm{nil}}
\newcommand{\red}{\rm{red}}
\newcommand{\DQ}{{\scr{D}_\bb{Q}}}

\DeclareMathSymbol{\shortminus}{\mathbin}{AMSa}{"39}

\DeclareFontFamily{U}{min}{}
\DeclareFontShape{U}{min}{m}{n}{<-> udmj30}{}

\AtBeginDocument{%
	\def\MR#1{}
}

\title{Mixed Tate motives over number fields}

\author{Alexander Kupers}
\address{Department of Computer and Mathematical Sciences, University of Toronto Scarborough, 1265 Military Trail, Toronto, ON M1C 1A4, Canada}
\email{a.kupers@utoronto.ca}

\author{Daniil Rudenko}
\address{Department of Mathematics,
University of Chicago,
5734 S. University Avenue, 
Chicago, IL 60637, USA}
\email{rudenkodaniil@uchicago.edu}

\author{Ismael Sierra}
\address{Department of Mathematics, University of Toronto. Bahen Centre 40 St. George Street, Room 6290. Toronto, Ontario Canada M5S 2E4}
\email{ismael.sierra@utoronto.ca}

\date{\today}

\begin{document}
	
\begin{abstract}This paper relates algebraic $K$-theory of fields to polylogarithms via general linear groups. We focus on the case of number fields and prove that the motivic realisation map from the Goncharov Lie coalgebra to the motivic Lie coalgebra is an isomorphism. This implies the Goncharov universality conjecture and a structural result for special values of Dedekind zeta functions. We also construct explicit polylogarithmic cocycles representing nonzero multiples of the Borel classes.
\end{abstract}

\maketitle

\vspace{-.75cm}

\tableofcontents

\vspace{-.75cm} 

\section{Introduction} In \cite{KRS1} we introduced the \emph{Goncharov Lie coalgebra} of a field $F$. It was constructed from the $E_\infty$-algebra $\BGLb(F)_\bb{Q}$ obtained from the chains $C_*(\BGL_n(F);\bb{Q})$ for $n \geq 1$ with multiplication induced by block sum, as
\[
\PolyL(F) \coloneq \bigoplus_{n \geq 1} \PolyL_n(F) \qquad \text{with} \qquad \PolyL_n(F) \coloneq H_{n,2n-1}^{E_{\infty}}(\BGLb(F)_{\bb{Q}}).
\]
We gave a presentation of $\PolyL(F)$ in terms of correlators and described its cobracket; alternatively one may give a presentation in terms of generators $\smash{\LiG_{n_1,\compactldots,n_k}(a_1,\compactldots,a_k)}$, which are analogues in the Goncharov Lie coalgebra of the multiple polylogarithms. We then combined this with the Rognes rank spectral sequence to find symbolic descriptions of certain rationalised algebraic $K$-theory groups of a field $F$. In this paper we continue studying this object, but this paper can be read mostly independently from \cite{KRS1}.

\medskip

We focus on number fields, though we also have comparable but easier results for fields of transcendence degree one over a finite field (see \cref{sec:transcendence-degree-1}). In \cite[Theorem D.a]{KRS1} we proved that for number fields there exists a motivic realisation functor $R^\rm{MTM}$ from category of a finite-dimensional graded comodules over $\PolyL(F)$ to the category of mixed Tate motives over $F$, which by the Tannakian formalism induces a map $r^\rm{MTM}$ of graded Lie coalgebras (see e.g.~\cite{DupSurvey}). The main result of this paper is that the latter is an isomorphism, and hence the former is an equivalence. This yields a complete and explicit description of the category of mixed Tate motives over a number field:

\begin{atheorem} \label{thm:polyl-iso-number-field} For a number field $F$  there is a unique isomorphism of graded Lie coalgebras
\[
r^{\rm{MTM}}\colon \PolyL(F) \overset{\cong}\lra \MotCoLie(F)\]
satisfying $\LiG_{n_1,\compactldots,n_k}(a_1,\compactldots,a_k) \mapsto \rm{Li}^\rm{MTM}_{n_1,\compactldots,n_k}(a_1,\compactldots,a_k)$. Consequently, the associated tensor functor is an equivalence of Tannakian categories \[R^\rm{MTM} \colon \rm{Comod}^\rm{fd}_{\PolyL(F)}(\rm{GrMod}_\bb{Q}) \overset{\simeq}\lra \rm{MTM}_\bb{Q}(F).\]
\end{atheorem}

\cref{thm:polyl-iso-number-field} has a concrete consequence for motives, proving the \emph{Goncharov universality conjecture} for number fields \cite[Conjecture 17(a)]{Gon95}:

\begin{acorollary}\label{cor:goncharov-universality} The class of every framed rational mixed Tate motive over a number field is a linear combination of motivic multiple polylogarithms.
\end{acorollary}

\cref{cor:goncharov-universality} implies that any period of a mixed Tate motive over $F$ can be expressed in terms of multiple polylogarithms evaluated at elements in $F$. A related application to Dedekind zeta functions is the following. For a number field $F$, the dimensions $\dim(K_{2n-1}(F)_\bb{Q}) = d_n$ for $n>1$ of the rationalised algebraic $K$-theory groups were computed by Borel \cite{BorelStable}: if $r_1$ is the number of real embeddings and $r_2$ is the number of conjugate pairs of complex embeddings, then $d_n = r_1+r_2$ if $n$ is odd and $d_n = r_2$ is $n$ is even. The next statement uses the real period map $p_\R\colon \PolyL_n(\C)\to \R$ defined in \cref{sec: real periods} (with a suitable universal choice of normalisation, see \cref{sec:normalisation}). This sends $\smash{\LiG_{n_1,\compactldots,n_k}(a_1,\compactldots,a_k)}$ to the evaluation of a single-valued real-valued variant of the corresponding multiple polylogarithm at the point $(a_1,\compactldots,a_k) \in \bb{C}^k$.

\begin{acorollary}\label{cor:zagier} For a number field $F$ and $n>1$, the kernel of the cobracket map $\delta\colon \PolyL_n(F)\to \Lambda^2\PolyL(F)$ is a $\Q$-vector space of dimension $d_n$. For any basis $\xi_1,\dots,\xi_{d_n}$ of the kernel, we have
\[\zeta_F(n) \sim_{\bb{Q}^\times} \frac{\pi^{n([F: \Q]-d_n)}}{\sqrt{|D_F|}}\det\Bigl(p_\R\bigl(\sigma_j(\xi_i)\bigr)\Bigr)_{1 \leq i,j \leq d_n}\]
where $\sim_{\bb{Q}^\times}$ denotes that two real numbers agree up to a nonzero rational multiple and $D_F$ denotes the discriminant.
\end{acorollary}

For totally real fields this can also be deduced by combining \cite[Corollary 1.4]{BrownTR} and \cite[Theorem 1.3]{Rud20}. The Zagier conjecture asks for a similar formula for $\zeta_F(n)$ but requires that only classical polylogarithms appear, rather than multiple polylogarithms \cite[Section 8]{Zag90} \cite[Conjecture 6]{Gon95}. This follows from \cref{cor:zagier} if the depth conjecture for $\scr{G}(F)$ is true. Previous results on the Zagier conjecture include \cite{ZagierHyperbolic,Gon95,GR18}, see \cite{Dup20}. 

\subsection{Proof of \cref{thm:polyl-iso-number-field}} The proof of \cref{thm:polyl-iso-number-field} combines
\begin{enumerate}[(a)]
    \item the $E_\infty$-algebra machinery from \cite{KRS1}, based on \cite{GKRW18,GKRW20},
    \item the computations of the cohomology of $S$-arithmetic groups  obtained using automorphic methods due to Borel and Yang \cite{borelyang}, and \item a description of the cocycles that induce the canonical map from homology to $E_\infty$-homology (see \cref{sec:intro-borel}).
\end{enumerate}
Recall that $\BGLb(F)_\bb{Q}$ is an $E^\rm{nu}_\infty$-algebra with homology $H_{n,d}(\BGLb_\Q(F)) \cong H_d(\BGL_n(F);\bb{Q})$ for $n>0$ and that the Goncharov Lie coalgebra $\scr{G}(F)$ is built from its $E_\infty$-homology. Since $\scr{L}^\rm{MTM}(F)$ is isomorphic to a cofree graded Lie coalgebra with cogenerators given by $K_*(F)_\Q$ for $*>0$ \cite{DG05}, our first step is to prove that $\scr{G}(F)$ is also isomorphic to this, that is, the Goncharov Lie coalgebra and the motivic Lie coalgebra are \emph{abstractly} isomorphic.

It would suffice if the quotient $\BGLb(F)_\bb{Q}/(\sigma)$ of the $E_\infty^\rm{nu}$-algebra $\BGLb(F)_\bb{Q}$ by the generator $\sigma \in H_0(\GL_1(F);\bb{Q})$ were a trivial algebra on the units of $F$ and duals of the Borel classes: these exactly correspond to the rational algebraic $K$-theory of $F$ by the work of Borel \cite{BorelStable}, and the Koszul dual of a trivial algebra is a cofree Lie coalgebra. While the previous statement is not true, the work of Borel and Yang implies that the homology of the quotient, $H_{*,*}(\BGLb(F)_\bb{Q}/(\sigma))$, is a trivial algebra on the units of $F$ and duals of the Borel classes \emph{below a line of slope 2} \cite{borelyang}. To use this, we show one may work below such a line (see \cref{thm:image-edge-is-cobracket}) and then exploit the existence of a map $\BGLb(F)_\bb{Q} \to \bf{C}_*(\fr{gl},\fr{u})$ to lift Borel--Yang's statement about homology of $\BGLb(F)_\bb{Q}/(\sigma)$ to a statement about $\BGLb(F)_\bb{Q}/(\sigma)$ itself. Its target constructed as the dual to the relative Lie algebra cochains of the pairs $\fr{u}_n \subset \fr{gl}_n$ (see \cref{thm:integration-assocative-algebra}) and is trivial as an algebra below a line of slope 2. The map is, under the van Est isomorphism, given by pairing homology against continuous cohomology and is known to detect the duals of the Borel classes.

Once we know $\scr{G}(F)$ and $\scr{L}^\rm{MTM}(F)$ are abstractly isomorphic, we prove that motivic realisation map is an isomorphism. For this second step we proceed by induction on $n$ and the initial case $n=1$ is well-known. For the induction step, we take $n \geq 2$ and consider for each embedding $F \subset \C$ the commutative diagram
\[\begin{tikzcd}[row sep=small] & K_{2n-1}(F)_\bb{Q} \arrow[dashed,bend right=18]{ld} \dar[swap]{\rm{edge}_n} \\ 
H_{2n-1}(\GL_n(F);\bb{Q}) \rar \dar & \scr{G}_n(F) \rar{r^\rm{MTM}} \dar & \scr{L}^\rm{MTM}_n(F) \dar & \\
H_{2n-1}(\GL_n(\C);\bb{Q}) \rar & \scr{G}_n(\C) \rar{r^\rm{Hod}} & \scr{L}^\rm{Hod}_n \rar & \R & \end{tikzcd}\]
where a dashed lift exists since the rank conjecture is true for number fields by the work of Borel and Yang. By unwinding the constructions, we produce a cocycle representing the bottom composition and observe it represents a continuous cohomology class that vanishes when restricted to $\GL_{n-1}(\C)$, so must be a multiple of the Borel class. This is used to show that the composition $K_{2n-1}(F)_\bb{Q} \to \scr{G}_n(F) \to  \scr{G}_n(\C) \to \scr{L}^\rm{Hod}_n \to \R$ is the same multiple of the Borel regulator associated with the embedding $F \hookrightarrow \C$. To see this multiple must be \emph{nonzero} we combine two facts. Firstly, the element $\LiG_n(i) \in \scr{G}_n(\bb{Q}(i))$ maps to a nonzero real number under the composition $\scr{G}_n(\bb{Q}(i)) \to \scr{G}_n(\bb{C}) \to \scr{L}^\rm{Hod} \to \bb{R}$ by a computation of Levin. Secondly, it is in the image of $\smash{K_{2n-1}(F)_\bb{Q}}$ since we prove that $\rm{im}(\rm{edge}_n) = \smash{\ker(\delta|_{\scr{G}_n(F)})}$ and since the cobracket involves only terms of lower degree we can use the induction hypothesis to compute that $\smash{\delta(\LiG_n(i))} = 0$ (this uses that $i$ is a root of unity). We conclude that the Borel regulators factor over the motivic realisation map and thus the composition
\[K_{2n-1}(F)_\bb{Q} \xrightarrow{\rm{edge}_n} \ker(\delta|_{\scr{G}_n(F)}) \xrightarrow{r^\rm{MTM}} \ker(\delta^\rm{MTM}|_{\scr{L}^\rm{MTM}_n(F)}) \lra \bb{R}^{d_n}\]
is injective by the work of Borel \cite{BorelReg,BorelRegErrata}. As all terms have the same dimension, the middle map must be an isomorphism, completing the induction step.

\subsection{Cocycles for Borel classes}\label{sec:intro-borel} In the proof of \cref{thm:polyl-iso-number-field}, we obtained a polylogarithmic cocycle representative for the Borel classes. We now describe this in detail, as such cocycles were sought after, see e.g.~\cite{GonExplicit} \cite{GoncharovArakelov} \cite{CGR}, and previously only known for the first three Borel classes. In \cite{KRS3} we will use these to prove a surjectivity conjecture of Monod in the context of continuous bounded cohomology of connected semisimple Lie groups with finite centre.

\medskip

We start with a general procedure: for an $n$-spherical complex $X$ with $G$-action, we can interpret the augmented cellular chains with rational coefficients as an iterated extension of $\bb{Q}$ by $\smash{\widetilde{H}_n(X;\bb{Q})}$ in $\bb{Q}[G]$-modules, and through Yoneda's work as $x \in \smash{H^{n+1}(G;\widetilde{H}_n(X;\bb{Q}))}$. Then the map induced on pointed $G$-orbits by the canonical pointed $G$-equivariant map from the circle into the double suspension of $X$ is given by the cap product
\[- \cap x \colon H_*(G;\ds{Q}) \lra H_{*-(n+1)}(G;\widetilde{H}_n(X;\bb{Q})).\]
Finding an explicit cocycle representative $\sf{x}$ of $x$, and hence a chain level description of this cap product map, uses the proof that $X$ is $n$-spherical.

\medskip

Using \cite[Section 4.4.1]{KRS1}, the canonical map 
\[H_d(\GL_n(F);\bb{Q}) = H_{n,d}(\BGLb(F)_\bb{Q}) \lra H^{E_1}_{n,d}(\BGLb(F)_\bb{Q}) \cong H_{d-(n-1)}(\GL_n(F);\St_n(F))\]
arises in such a manner for the $(n-2)$-spherical complex given by the Tits building $T(F^n)$. Its top rational homology is given by the (rationalised) Steinberg module $\St_n(F)$ and so the canonical map is given by the cap product with a \emph{Steinberg class} $\rm{st}_n \in H^{n-1}(\GL_n(F);\St_n(F))$, which admits an explicit representative $\sf{st}_n$ arising from the proof of the Solomon--Tits theorem. The proof of Rognes conjecture in \cite{CharltonRadchenkoRudenko} implies that the analogous canonical map 
\[H_d(\GL_n(F);\bb{Q}) = H_{n,d}(\BGLb(F)_\bb{Q}) \lra H^{E_\infty}_{n,d}(\BGLb(F)_\bb{Q}) \cong H_{d-2n+2}(\GL_n(F);\StL_n(F))\]
is given by the cap product with $\pi(\rm{st}_n \cup \rm{st}_n)$ with $\pi \colon \StH_n(F) \to \StL_n(F)$ the projection, and can be used to give an explicit representative $\pi(\sf{st}_n \cup \sf{st}_n)$ for this class.

Recall now that as a consequence of the proof of \cref{thm:polyl-iso-number-field} the Borel class is given, up to a nonzero multiple, by the map
\[H_{2n-1}(\GL_n(\C);\bb{Q}) \lra H_1(\GL_n(\C);\StL_n(\C)) \overset{\cong}\lra \scr{G}_n(\C) \overset{p_\bb{R}} \lra \bb{R}.\] 
We have an explicit cocycle for the canonical map on the left, and the right map merely applies the real period map, so it remains to incorporate the identification of $H_{1}(\GL_n(F);\StL_n(F))$ with $\scr{G}_n(F)$. This involves a cocycle built out of the differences $D_{x,y} = D_x-D_y$ of dual decomposition operators (see \cref{sec:dual-decomposition}). Combining these ingredients we obtain:

\begin{atheorem}\label{thm:borel-cocycles} The following inhomogeneous cocycle represents, up to a nonzero multiple, the Borel class $\rm{Bo}_{n,2n-1} \in H^{2n-1}(\GL_n(\bb{C});\bb{R})$:
   \begin{align*} \GL_n(\C)^{2n-1} &\lra \bb{R} \\
    [g_1|\compactcdots|g_{2n-1}] &\longmapsto(-1)^{n-1} p_\R \big(D_{\ell_{n-1},\ell_n}(\pi ([\ell_0,\compactldots,\ell_{n-1}] \otimes [\ell_n,\compactldots,\ell_{2n-1}])) \big),\end{align*}
    where $\ell_i \coloneq g_1 \compactcdots g_i \cdot \ell$ for a fixed line $\ell \in \bb{P}^1(\C^n)$.
\end{atheorem}

Let us refer to the expression
\[
\sf{g}_n(\ell_0,\compactldots,\ell_{2n-1}) \coloneq (-1)^{n-1}  D_{\ell_{n-1},\ell_n}(\pi ([\ell_0,\compactldots,\ell_{n-1}] \otimes [\ell_n,\compactldots,\ell_{2n-1}]))
\]
as the \emph{Goncharov cocycle}. For generic configurations, it admits a more explicit description:

\begin{atheorem}\label{athm: borel explicit} For lines $\ell_0,\dots,\ell_{2n-1}\in \mathbb{P}^{n-1}$ in general position, we have  
\[
\sf{g}_n(\ell_0,\compactldots,\ell_{2n-1})=(-1)^{\binom{n}{2}+1}\GLiG_n(\ell_0,\compactldots,\ell_{n-1},\ell_{2n-1},\compactldots,\ell_n),
\]
where $\GLiG_n$ is a $\PolyL$-version of the \emph{cluster Grassmannian polylogarithm} of \cite[Section 1.5]{MR22}. 
\end{atheorem}

\begin{remark}The alternation of the Goncharov cocycle $\sum_{\sigma \in \mathfrak{S}_{2n}}(-1)^{\sigma}\sf{g}_n(\ell_{\sigma(0)},\compactldots,\ell_{\sigma(2n-1)})$ is proportional to Goncharov's \emph{Grassmannian polylogarithm}, whose symbol was constructed in \cite[Theorem 1.3]{Gon13} and for which a polylogarithmic presentation was found in \cite{CGR}. \cref{thm:borel-cocycles} implies that the cocycle constructed by Goncharov is a nonzero multiple of the Borel regulator: to see this, use that the homogeneous cochain corresponding to \cref{thm:borel-cocycles} is given by mapping $(g_0,\compactldots,g_{2n-1})$ to the same expression but with $\ell_i \coloneq g_i \cdot \ell$ and that alternation is a nonzero multiple of the identity on homogeneous cochains \cite[Proposition 4.26]{Frigerio}. This was previously known only for $n\le 4$ \cite{Bloch,Gon95b,GR18}. It is surprising that the Goncharov cocycle has cluster-like properties.\end{remark}

\subsection{Edge homomorphisms} \label{sec:intro-edge} We finally highlight several ingredients in the proof of \cref{thm:polyl-iso-number-field}, which either may be of independent interest or were promised in \cite{KRS1}. 

\medskip

When $\bf{R}$ is an $E_\infty^\rm{u}$-algebra in spaces with $\pi_0\,\bf{R} \cong \bb{N}$, the group completion theorem provides an acyclic map
\[\colim_n \bf{R}(n) \lra \Omega^\infty_0(\bf{R}^\gc)\]
where the colimit is taken by multiplication with some fixed choice of $\sigma \in \bf{R}(1)$ and $(-)^\gc$ stands for group completion. For example, for $\bf{R}$ the unitalisation of $\BGLb(F)$ this yields the homology equivalence $\BGL_\infty(F)^+ \to \Omega^\infty_0 K(F)$.

In \cref{sec:stable-algebras-and-rank-spectral-sequences} of this paper we give analogous constructions for an $E^\rm{u}_\infty$-algebra $\bf{R}$ in a presentable stable symmetric monoidal $\scr{C}$ with an additional rank grading and a choice of element in rank 1, encoded as a map $\sigma \colon \free_{E_\infty^\rm{u}}(1_! 1_\scr{C}) \to \bf{R}$. Namely, we construct from this data an \emph{stable algebra} $\bf{R}/(\sigma-1)$ where $\sigma$ is set equal to the unit (so this loses its additional rank grading); if $\scr{C} = \Spc$ and $\pi_0\,\bf{R} \cong \bb{N}$ as above, then this is equivalent to $\Omega^\infty(\bf{R}^\gc/\bb{S})$ which is rationally the same as $\Omega^\infty_0 (\bf{R}^\gc)$. Taking into account an induced filtration by rank, we obtain a filtered object whose associated graded is given by $\bf{R}/(\sigma)$, which has rational homology $H_*(\bf{R}(n),\bf{R}(n-1);\bb{Q})$ in additional grading $n$, as the unit lies in rank $0$ but $\sigma$ in rank $1$. This yields a rank spectral sequence that is a reduced variant of the Rognes rank spectral sequence used in \cite{KRS1}.

This generality is necessary because in \cref{sec:goncharov-lie-coalgebra-general} it is applied to $E_\infty^\rm{nu}$-algebras in rational chain complexes (freely taking augmentations or unitalising when necessary) that do \emph{not} arise from spaces. We construct these as follows, given an $E_\infty^\rm{nu}$-algebra $\bf{R}$ in rational chain complexes: take $E_\infty^\rm{nu}$-indecomposables, perform a Postnikov truncation on the resulting shifted Lie coalgebra, and take $s\,\rm{coLie}$-primitives. The case of interest for this paper is when the $E_\infty$-homology of $\bf{R}$ vanishes in bidegrees $(n,d)$ satisfying $d < 2n-1$  and the Postnikov truncation is to bidegrees $(n,d)$ satisfying $d \leq 2n-1$. The resulting \emph{Koszul truncation} is analogue to Randal-Williams' construction of the ``stability Hopf algebra'' in \cite{RWchromatic}. This construction in particular applies to $\BGLb(F)_\bb{Q}/(\sigma)$ and then produces the Koszul dual $E_\infty^\rm{nu}$-algebra to the Goncharov Lie coalgebra. There is a canonical map from $\bf{R}$ to its Koszul truncation and applying the techniques from \cref{sec:stable-algebras-and-rank-spectral-sequences} to this we use the induced map on rank spectral sequences to define refined edge homomorphisms and study their properties. Rather than further discussing generalities, let us specialise to $\BGLb(F)_\bb{Q}$ and state the outcome: 

\begin{enumerate}[\noindent (1)]
\item We prove that the map $\rm{edge}_n \colon K_{2n-1}(F)_\bb{Q} \to \scr{G}_n(F)$ of \cite[Definition 9.2]{KRS1} has image in $\ker(\delta|_{\scr{G}_n(F)}) \subseteq \scr{G}_n(F)$.
    \item We produce \emph{refined edge homomorphisms}
\[\rm{e}_{n,2n-i} \colon \gr^\prim_n K_{2n-i}(F)_\bb{Q} \lra H^\rm{CE}_{n,2n-i}(\scr{G}(F))\]
from the associated graded $\gr^\prim_n K_{2n-i}(F)_\bb{Q}$ of the primitive rank filtration to the Chevalley--Eilenberg homology of the Goncharov Lie coalgebra. 
\item We describe how to compute $\rm{edge}_n \colon K_{2n-1}(F)_\bb{Q} \to \scr{G}_n(F)$ on elements of low ``indecomposable rank filtration'' in terms of the canonical map from homology to $E_\infty$-homology.
\end{enumerate}
Part (1) is a crucial step in the proof of \cref{thm:polyl-iso-number-field}. Part (3) allows us to relate edge homomorphisms to \cref{sec:intro-borel}: more precisely, we relate a \emph{Goncharov regulator map} (defined using edge homomorphisms and real periods) to Borel's regulator (see \cref{prop:borel-regulator-multiple}). This is also used in the proof of \cref{thm:polyl-iso-number-field}. 

Assuming the rank conjecture stating that there is an isomorphism $\smash{\gr^\prim_n K_{2n-i}(F)_\bb{Q}} \cong \smash{K_{2n-i}^{(n)}(F)_{\bb{Q}}}$, we propose the maps $\rm{e}_{n,2n-i}$ of part (2) as candidates for the isomorphisms in \cite[Conjecture A.b]{KRS1}, using that $H^\rm{CE}_{n,2n-i}(\scr{G}(F)) = H^i(\scr{G}(F))_n$ in terms of \cite[(2)]{KRS1}:

\begin{customconj}{A.b (refined)}\label{conjecture main gamma} Let $F$ be an arbitrary field. For $1\leq i \leq n$, then the refined edge homomorphisms are isomorphisms
\[
\gr^\prim_n K_{2n-i}(F)_\bb{Q} \overset{\cong}\lra H^i(\PolyL(F))_n.
\]
\end{customconj}

\cref{thm:polyl-iso-number-field} proves the above conjecture for $F$ a number field. Moreover, for arbitrary fields $F$ we prove the above conjecture for Milnor $K$-theory, i.e.~in the case $i=n$, see \cref{lem:milnor K theory}. 

\subsection*{Acknowledgments} AK acknowledges the support of the Natural Sciences and Engineering Research Council of Canada (NSERC) [funding reference number 512156 and 512250]. DR was supported by NSF grant DMS-2502729. IS acknowledges the University of Toronto. The authors would like to thank Clément Dupont, Søren Galatius, Alexander Goncharov, David Kazhdan, and Oscar Randal-Williams for helpful discussions and suggestions.

\section{Stable algebras and rank spectral sequences}\label{sec:stable-algebras-and-rank-spectral-sequences}
In this section we develop for $E_\infty^\rm{nu}$-algebras with a choice of stabilisation map in a presentable symmetric monoidal $\infty$-category $\scr{C}$ the following tools:
\begin{enumerate}[(i)]
    \item \emph{Algebraic group completion} and \emph{stable algebra} construction.
    \item Variants of this in the filtered and graded settings.
    \item Rank spectral sequences associated with these variants.
\end{enumerate}
Part (i) is summarised by the following table, which explains the analogy to spaces:
\begin{center}\medskip\begin{tabular}{cc} \toprule
     $\Spc$ & $\scr{C}$  \\ \midrule
    $E_\infty^\rm{u}$-algebra $\bf{R}$ with $\pi_0\,\bf{R} \cong \bb{N}$  &  $E_\infty^\rm{u}$-algebra in $\scr{C}$ with $\sigma \colon \bf{U} \to \bf{R}$ \\
    group completion $\Omega^\infty \bf{R}^\gc$ & algebraic group completion $\bf{R}[\sigma^{-1}]$ \\ 
    colimit $\rm{colim}_n\, \bf{R}(n)$ & stable algebra $\bf{R}/(\sigma-1)$
    \\\bottomrule
\end{tabular}\medskip\end{center}
The analogue of the map $\sigma \colon \bf{U} \coloneq \free_{E_\infty^\rm{u}}(1_\scr{C}) \to \bf{R}$ in the right column is in the left column given by a choice of representative of the component of $\bf{R}$ of $1 \in \bb{N}$. We suggest that the reader on a first reading skips ahead to \cref{sec:alg-k-theory-goncharov-maps}, which states the applications of these results to the $E_\infty^\rm{nu}$-algebra $\BGLb(F)_\bb{Q}$.

\subsection{Algebraic group completion and stable algebras} We start this subsection by a more detailed discussion about the setting of spaces, and then give the definitions and basic properties of the algebraic group completion and stable algebra constructions.

\subsubsection{Group completion in spaces} The category $\Spc$ of spaces is a symmetric monoidal category with cartesian symmetric monoidal structure. A unital $E_\infty$-algebra $\bf{R} \in \Alg_{E_\infty^\rm{u}}(\Spc)$ is \emph{group-like} if the map $(\pr_1,m) \colon \bf{R} \times \bf{R} \to \bf{R} \times \bf{R}$ is an equivalence (equivalently, the commutative monoid $\pi_0\bf{R}$ is a group), and we denote by $\Alg_{E_\infty^\rm{u}}^\rm{grp}(\Spc)$ the full subcategory of $E_\infty^\rm{u}$-algebras which are group-like. Its inclusion into $\Alg_{E_\infty^\rm{u}}(\Spc)$ admits a left adjoint \cite[Corollary 4.4]{GGN} and we will use the equivalences of categories \cite[5.1.3.7]{LurieHA}
\[\begin{tikzcd} \Alg_{E_\infty^\rm{u}}^\rm{grp}(\Spc) \arrow[shift left=2pt]{r}{B^\infty} &[10pt] \Sp_{\ge 0} \arrow[shift left=2pt]{l}{\Omega^\infty}.\end{tikzcd}\]
to view this as a \emph{group-completion} functor 
\[(-)^\gc \colon \Alg_{E_\infty^\rm{u}}(\Spc) \lra \Sp_{\ge 0}.\]
From this the space-level group completion can be recovered as $\Omega^\infty \bf{R}^\gc$ and the unit transformations give a canonical $E_\infty^\rm{u}$-algebra map $\bf{R} \to \Omega^\infty \bf{R}^\gc$. This is what our ``algebraic group completion'' construction aims to generalise.

We next explain our ``stable algebra'' construction. If we start with $\bf{R} \in \Alg_{E_\infty^\rm{u}}(\Spc)$ such that there is an isomorphism $\pi_0\,\bf{R} \cong \bb{N}$ of commutative monoids, then its group completion is a connective spectrum $\bf{R}^\gc$ with an isomorphism $\pi_0\, \bf{R}^\gc \cong \bb{Z}$. In this situation, we may consider the identity path-component of the group completion: $\Omega_0^\infty \bf{R}^\gc$. By the group-completion theorem \cite[Proposition 1.1]{McDuffSegal} \cite[Theorem 1.1]{RWgroupcompletion} there is an acyclic map of spaces
\[\colim_n\,\bf{R}(n) \lra \Omega_0^\infty \bf{R}^\gc,\]
where $\bf{R}(n)$ denotes the path component of $\bf{R}$ corresponding to $n \in \bb{N}$ and the colimit is taken over a stabilisation map given by multiplying by an element of $\bf{R}(1)$ (which is well-defined up to homotopy). The left side is what our ``stable algebra'' aims to generalise.

\subsubsection{Algebraic group completion} We next explain our notion of algebraic group completion. Recall that the category $\smash{\Alg_{E_\infty^\rm{u}}(\scr{C})}$ of $E^\rm{u}_\infty$-algebras in $\scr{C}$ admits a symmetric monoidal structure so that the forgetful functor is symmetric monoidal \cite[3.2.4.4]{LurieHA}, and by \cite[3.2.4.7]{LurieHA} this tensor product of $E_\infty^\rm{u}$-algebras agrees with the coproduct in the sense that the map
\[\bf{R} \sqcup^{E^\rm{u}_\infty} \bf{S} \overset{\simeq}\lra \bf{R} \otimes \bf{S}\]
induced by the units, is an equivalence. We will generally opt to use the right side as our notation, since it makes visible the underlying object.

Let us denote by $\bf{U} \coloneq \free_{E_\infty^\rm{u}}(1_\scr{C}) \in \Alg_{E_\infty^\rm{u}}(\scr{C})$ the free $E_\infty^\rm{u}$-algebra on the monoidal unit $1_\scr{C}$. It comes with the following additional structure:

\begin{enumerate}[\noindent (1)]
    \item The \emph{group completion augmentation}
\[\epsilon_\gc \colon \bf{U} \lra 1_\scr{C}\]
is defined as the adjoint to $1_\scr{C} \xrightarrow{\id} \fgt_{E_\infty^\rm{u}} 1_\scr{C}$.
    \item The \emph{codiagonal}
        \[\nabla \colon \bf{U} \otimes \bf{U} \lra \bf{U}\]
    is given by under the equivalence $\bf{U} \sqcup^{E_\infty^\rm{u}} \bf{U} \simeq \bf{U} \otimes \bf{U}$ by the map $\id_{\bf{U}}$ on each component. Its underlying map is the multiplication of $\bf{U}$.
    \item The \emph{diagonal}
        \[\Delta \colon \bf{U} \lra \bf{U} \otimes \bf{U}\]
    is given by using the oplax symmetric monoidality $\free_{E_\infty^\rm{u}}(X \otimes Y) \to \free_{E_\infty^\rm{u}}(X) \otimes\free_{E_\infty^\rm{u}}(Y)$
on the free algebra functor as the left adjoint to $\fgt_{E_\infty^\rm{u}}$, specialising to $X = 1_\scr{C} = Y$, and using the equivalence $1_\scr{C} \otimes 1_\scr{C} \simeq 1_\scr{C}$. It is the unique map of $E_\infty^\rm{u}$-algebras sending the generator $1_\scr{C}$ of the left to the tensor product of generators $1_\scr{C} \otimes 1_\scr{C}$ of the right.
\end{enumerate}

We will work in the category $\Alg_{E_\infty^\rm{u}}(\scr{C})^{\bf{U}/}$ of algebras under $\bf{U}$. One should think of its objects as $E_\infty^\rm{u}$-algebras $\bf{R}$ with a map $\sigma \colon \bf{U} \to \bf{R}$ specifying a way to ``stabilise''. The following is informally given by inverting $\sigma$, as the notation suggests:

\begin{definition}\label{def:algebraic-gc} The \emph{algebraic group completion} functor is
\begin{align*}(-)[\sigma^{-1}] \colon \Alg_{E_\infty^\rm{u}}(\scr{C})^{\bf{U}/} &\lra \Alg_{E_\infty^\rm{u}}(\scr{C}) \\
(\bf{U} \xrightarrow{\sigma} \bf{R}) &\longmapsto 1_\scr{C} \otimes_{\bf{U}} (\bf{U} \otimes \bf{R}),\end{align*}
equivalently given by the following pushout in $\Alg_{E_\infty^\rm{u}}(\scr{C})$
 \[\begin{tikzcd} \bf{U} \rar{(\id_\bf{U} \sqcup \sigma) \circ \Delta} \arrow{d}[swap]{\epsilon_\gc} &[40pt] \bf{U} \sqcup^{E_\infty^\rm{u}} \bf{R} \arrow{d}\\[-5pt]
   1_\scr{C} \rar & \bf{R}[\sigma^{-1}]. \end{tikzcd}\]
\end{definition}

\begin{remark}In the language of cellular $E_k$-algebras \cite{GKRW18}, this construction takes $\bf{R}$, adds a free $E_\infty$-cell on $D^0$, and declares it to be the multiplicative inverse of $\sigma$.\end{remark}

In the next lemma, we use the natural transformation given by the compositions
\[\bf{R} \xrightarrow{\inc} \bf{U} \sqcup^{E_\infty^\rm{u}} \bf{R} \lra \bf{R}[\sigma^{-1}].\]
The first part says $(-)[\sigma^{-1}] \simeq (-) \otimes_\bf{U} \bf{U}[\id_\bf{U}^{-1}]$ and the second part says this preserves colimits.

\begin{lemma}\label{lem:alg-group-completion-first-props}\,
    \begin{enumerate}[(i)]
        \item \label{enum:alg-group-completion-first-props-i} For $\bf{R} \in \Alg_{E_\infty^\rm{u}}(\scr{C})^{\bf{U}/}$ there is a natural pushout diagram in the category $\Alg_{E_\infty^\rm{u}}(\scr{C})$
    \[\begin{tikzcd}
        \bf{U} \arrow{r}{\sigma} \arrow{d} &[10pt] \bf{R} \arrow{d} \\[-5pt] \bf{U}[\id_\bf{U}^{-1}] \arrow{r}{\sigma[\sigma^{-1}]} & \bf{R}[\sigma^{-1}].
    \end{tikzcd}\]
    \item \label{enum:alg-group-completion-first-props-ii} The functor $(-)[\sigma^{-1}]$ preserves colimits.
    \end{enumerate}
\end{lemma}

\begin{proof}For \eqref{enum:alg-group-completion-first-props-i}, we use that there is a commutative diagram
\[\begin{tikzcd}
    & \bf{U} \rar{\sigma} \dar{\inc} &[10pt] \bf{R} \dar{\inc} \\[-5pt]
    \bf{U} \arrow{r}{\Delta} \arrow{d}{\epsilon_\gc} & \bf{U} \sqcup^{E_\infty^\rm{u}} \bf{U} \rar{\id_\bf{U} \sqcup \sigma} \arrow{d} & \bf{U} \sqcup^{E_\infty^\rm{u}} \bf{R} \arrow{d} \\[-5pt]
    1_\scr{C} \arrow{r} & \bf{U}[\id_\bf{U}^{-1}] \rar{\sigma[\sigma^{-1}]} & \bf{R}[\sigma^{-1}].
\end{tikzcd}\]
The left and bottom outer squares are pushouts, and thus so is the right square. But the top square is also a pushout, hence so is the right outer square. For part \eqref{enum:alg-group-completion-first-props-ii}, we use part \eqref{enum:alg-group-completion-first-props-i} and the well-known fact that any functor of the form $A \sqcup_B (-) \colon \scr{D}^{B/} \to \scr{D}$ preserves colimits.
\end{proof}

\begin{example}\label{exam:stable-algebra-of-gc-augmentation} For the group completion augmentation $(\bf{U} \overset{\epsilon_\gc}\lra 1_\scr{C})$, the unit map yields an equivalence $\smash{1_\scr{C} \overset{\simeq}\lra 1_\scr{C}/[\epsilon_\gc^{-1}]}$. To see this, note this map fits in a pushout square 
\[\begin{tikzcd} \bf{U} \arrow{rr}{(\id_\bf{U} \sqcup \epsilon_{\gc}) \circ \Delta} \arrow{d}[swap]{\epsilon_\gc} &[20pt]  & \bf{U} \sqcup^{E_\infty^\rm{u}} 1_\scr{C} \arrow{d}\\[-5pt]
   1_\scr{C} \arrow{rr} & & 1[\epsilon_\gc^{-1}] \end{tikzcd}\]
in $\Alg_{E_\infty^\rm{u}}(\scr{C})$, and the claim follows since the horizontal top map is under the identification of its target with $\bf{U}$ equivalent to $\id_\bf{U}$.
\end{example}

\subsubsection{Stable algebras} The stable algebra construction is an instance of a construction that declares two elements $\sigma$ and $\tau$ of $\bf{R}$ to be equal by taking a relative tensor product or coequaliser. Recall that in a category with coproducts the left is a coequaliser diagram if and only if the right is a pushout
\vspace{-.2cm} \[\begin{tikzcd} X \rar[shift left=.5ex]{f} \rar[shift left=-.5ex,swap]{g} & Y \rar{p} & Z \end{tikzcd} \qquad \begin{tikzcd} X \sqcup X \dar[swap]{\nabla} \rar{(f,g)} & Y \dar \\[-5pt]
X \rar & Z.\end{tikzcd}\]

\begin{definition}Given two maps $\sigma,\tau \colon \bf{U} \to \bf{R}$ in $\Alg_{E^\rm{u}_\infty}(\scr{C})$ we define $\bf{R}/(\sigma-\tau)$ as the relative tensor product $\bf{U} \otimes_{\bf{U} \otimes \bf{U}} \bf{R}$, or equivalently by the following pushout in $\Alg_{E_\infty^\rm{u}}(\scr{C})$
 \[\begin{tikzcd} \bf{U} \sqcup^{E_\infty^\rm{u}} \bf{U} \rar{(\sigma,\tau)} \dar[swap]{\nabla} & \bf{R} \dar \\[-5pt]
 \bf{U} \rar & \bf{R}/(\sigma-\tau).\end{tikzcd}\]
\end{definition}

\begin{remark}In the language of cellular $E_k$-algebras \cite{GKRW18}, this construction takes $\bf{R}$, and attaches a free $E_\infty$-cell on $D^1$ along its boundary mapped to $\sigma$ and $\tau$.\end{remark}

Algebraic group completion was given by inverting $\sigma$, and the stable algebra is giving by declaring $\sigma$ to be $1$. This uses that there is a map $1 \colon \bf{U} \to \bf{R}$ adjoint to the map $1_\scr{C} \to \fgt_{E_\infty^\rm{u}} \bf{R}$ that picks out the unit. A special case of the previous construction is hence:

\begin{definition}\label{def:stable-algebra} The \emph{stable algebra} functor is
\begin{align*}(-)/(\sigma-1) \colon \Alg_{E_\infty^\rm{u}}(\scr{C})^{\bf{U}/} &\lra \Alg_{E_\infty^\rm{u}}(\scr{C}) \\
(\bf{U} \xrightarrow{\sigma} \bf{R}) &\longmapsto \bf{U} \otimes_{\bf{U} \otimes \bf{U}} \bf{R}\end{align*}
where the relative tensor product is with respect to $(\sigma,1) \colon \bf{U} \otimes \bf{U} \to \bf{R}$ and $\nabla \colon \bf{U} \otimes \bf{U} \to \bf{U}$.
\end{definition}

By definition we have the following pushout in $\Alg_{E_\infty^\rm{u}}(\scr{C})$
\[\begin{tikzcd}
    \bf{U} \sqcup^{E_\infty^\rm{u}} \bf{U} \arrow{r}{(\sigma, 1)} \arrow{d}[swap]{\nabla} & \bf{R} \arrow{d} \\[-5pt] 
    \bf{U} \arrow{r} & \bf{R}/(\sigma-1).
\end{tikzcd}\]

\begin{example}\label{exam:stable-algebra-of-id-u} The group completion augmentation $\epsilon_\gc \colon \bf{U} \to 1_\scr{C}$ equalises the maps $\id_\bf{U}, 1 \colon \bf{U} \to \bf{U}$, so induces a map $\bf{U}/(\id_\bf{U}-1) \smash{\overset{\simeq}\lra} 1_\scr{C}$. This is an equivalence, by verifying the universal property.
\end{example}

This implies the following lemma, proven similarly to \cref{lem:alg-group-completion-first-props}. The first part says that 
\[(-)/(\sigma-1) \simeq (-) \otimes_{\bf{U}} \bf{U}/(\id_\bf{U}-1) \simeq (-) \otimes_\bf{U} 1_\scr{C},\] 
and the second part that these constructions preserve colimits.

\begin{lemma}\label{lem:stable-alg-first-props}\,
    \begin{enumerate}[(i)]
        \item \label{enum:stable-alg-first-props-i} 
            For $\bf{R} \in \Alg_{E_\infty^\rm{u}}(\scr{C})^{\bf{U}/}$ there are natural pushout squares in $\Alg_{E_\infty^\rm{u}}(\scr{C})$
            \[\begin{tikzcd} \bf{U} \rar{\sigma} \dar & \bf{R} \dar \\[-5pt]
            \bf{U}/(\id_\bf{U}-1) \rar & \bf{R}/(\sigma-1)\end{tikzcd} \quad\text{and} \quad \begin{tikzcd}
            \bf{U} \arrow{r}{ \sigma} \arrow{d}[swap]{\epsilon_\gc} & \bf{R} \arrow{d} \\[-5pt] 1_\scr{C} \arrow{r} & \bf{R}/(\sigma-1).
          \end{tikzcd}\]
        \item \label{enum:stable-alg-first-props-ii} The functor $(-)/(\sigma-1)$ preserves colimits.
    \end{enumerate}
\end{lemma}

Consider now the augmented setting, where $\bf{U}$ is augmented through $\epsilon_\gc$. By definition, an object in $\Alg^\aug_{E_\infty^\rm{u}}(\scr{C})^{\bf{U}/}$ yields a commutative square
\[\begin{tikzcd} \bf{U} \rar{\sigma} \dar[swap]{\epsilon_\gc} & \bf{R} \dar{\epsilon} \\[-5pt]
1_\scr{C} \rar{\id} & 1_\scr{C} \end{tikzcd}\]
so as a consequence of \cref{lem:stable-alg-first-props} \eqref{enum:stable-alg-first-props-i}, $\bf{R}/(\sigma-1)$ comes with a preferred augmentation, which we also refer to as the \emph{group completion augmentation}. Hence the stable algebra functor lifts to
\begin{equation}\label{eqn:stable-algebra-augmentation}(-)/(\sigma-1) \colon \Alg^\aug_{E_\infty^\rm{u}}(\scr{C})^{\bf{U}/} \lra \Alg^\rm{aug}_{E_\infty^\rm{u}}(\scr{C}).\end{equation}

\subsubsection{Cofibres} In this subsection we suppose $\scr{C}$ is pointed. Then the stable algebra construction has a variant which yields the cofibre instead. For any $\bf{R}$, there is a zero map $0 \colon \bf{U} \to \bf{R}$ adjoint to the zero map $1_\scr{C} \to \fgt_{E_\infty^\rm{u}} \bf{R}$ and we can use this to define $\bf{R}/(\sigma-0)$.

\begin{example}The zero map for $\bf{R} = 1_\scr{C}$ yields the \emph{canonical augmentation} $\epsilon_\rm{can} \colon \bf{U} \to 1_\scr{C}$. As the canonical augmentation equalises the maps $\id_\bf{U},0 \colon \bf{U} \to \bf{U}$, so induces a map $\bf{U}/(\id_\bf{U}-0) \smash{\overset{\simeq}\lra} 1_\scr{C}$ which is an equivalences, as may be verified by the universal property.
\end{example}

As in \cref{lem:stable-alg-first-props}, it follows that for $\bf{R} \in \Alg_{E_\infty^\rm{u}}(\scr{C})$ there are natural pushout squares
            \[\begin{tikzcd} \bf{U} \rar{\sigma} \dar & \bf{R} \dar \\[-5pt]
            \bf{U}/(\id_\bf{U}-0) \rar & \bf{R}/(\sigma-0)\end{tikzcd} \quad\text{and} \quad \begin{tikzcd}
            \bf{U} \arrow{r}{\sigma} \arrow{d}[swap]{\epsilon_\rm{can}} & \bf{R} \arrow{d} \\[-5pt] 1_\scr{C} \arrow{r} & \bf{R}/(\sigma-0),
          \end{tikzcd}\]
and by definition of the cofibre $\bf{R}/(\sigma)$ as a pushout, the right square says:

\begin{lemma}There is a natural equivalence between $\bf{R}/(\sigma-0)$ and $\bf{R}/(\sigma)$ in $\Alg_{E_\infty^\rm{u}}(\scr{C})$.
\end{lemma}

\subsubsection{Comparing algebraic group completion and stable algebras}
Given these two functors---the algebraic group completion and the stable algebra---we will now compare them: setting $\sigma$ to $1$ after inverting it should be the same as directly setting it to $1$. This is the content of the next lemma. We can lift $(-)[\sigma^{-1}]$ to an endofunctor 
\begin{align*}(-)[\sigma^{-1}] \colon \Alg_{E_\infty^\rm{u}}(\scr{C})^{\bf{U}/} &\lra \Alg_{E_\infty^\rm{u}}(\scr{C})^{\bf{U}/} \\
(\bf{U} \xrightarrow{\sigma} \bf{R}) &\longmapsto (\bf{U} \to \bf{R} \to \bf{R}[\sigma^{-1}]).\end{align*} 

\begin{lemma} \label{lem:stable-algebra-gc}
    Given $(\bf{U} \xrightarrow{\sigma} \bf{R}) \in \Alg_{E_\infty^\rm{u}}(\scr{C})^{\bf{U}/}$ the map $\bf{R} \to \bf{R}[\sigma^{-1}]$ induces a natural equivalence in $\Alg_{E_\infty^\rm{u}}(\scr{C})$
    \[\bf{R}/(\sigma-1) \xrightarrow{\simeq} \bf{R}[\sigma^{-1}]/(\sigma-1).\]
\end{lemma}

\begin{proof}
   Since the stable algebra functor preserves colimits by \cref{lem:stable-alg-first-props} \eqref{enum:stable-alg-first-props-ii}, we can apply it to the pushout square of \cref{lem:alg-group-completion-first-props} \eqref{enum:alg-group-completion-first-props-i} to reduce the statement to proving that
   \[\bf{U}/(\id_\bf{U}-1) \lra \bf{U}[\id_\bf{U}^{-1}]/(\id_\bf{U}-1)\]
   is an equivalence. As $1_\scr{C} \simeq  \bf{U}/(\id_\bf{U}-1)$ by \cref{exam:stable-algebra-of-id-u}, it suffices to show that $1_\scr{C} \simeq \bf{U}[\id_\bf{U}^{-1}]/(\id_\bf{U}-1)$; there is no need to identify the map because $1_\scr{C}$ is initial. By \cref{lem:alg-group-completion-first-props} \eqref{enum:alg-group-completion-first-props-i} and \cref{exam:stable-algebra-of-gc-augmentation}, $\bf{U}[\id_\bf{U}^{-1}]/(\id_\bf{U}-1) \simeq 1_\scr{C}[\epsilon_\gc^{-1}] \simeq 1_\scr{C}$.\end{proof}

\subsubsection{Interpretation in terms of module categories} There is a category $\rm{Mod}_\bf{A}(\scr{C})$ of modules over an $E_\infty^\rm{u}$-algebra $\bf{A}$ in $\scr{C}$ \cite[4.5.1.1]{LurieHA} (we suppress $\scr{O} = E_\infty^\rm{u}$ from the notation for legibility). Its underlying category agrees with that of left (or right) modules by \cite[4.5.1.6]{LurieHA} but it comes with a symmetric monoidal structure whose underlying tensor product is given by the relative tensor product $(-)\otimes_\bf{A} (-)$ \cite[Theorem 4.5.2.1]{LurieHA} and whose monoidal unit is given by $\bf{A}$. There is an equivalence \cite[3.4.1.7]{LurieHA}
\[\Alg_{E_\infty^\rm{u}}(\rm{Mod}_\bf{A}) \lra \Alg_{E_\infty^\rm{u}}(\scr{C})^{\bf{A}/}\]
given by remembering the underlying $E_\infty^\rm{u}$-algebra in $\scr{C}$ and its unit map from $\bf{A}$. Thus we can interpret the domains of the algebraic group completion and stable algebra constructions as $\Alg_{E_\infty^\rm{u}}(\rm{Mod}_\bf{U})$. From this perspective, the stable algebra construction is induced by the functor $\rm{Mod}_\bf{U} \to \rm{Mod}_{1_\scr{C}} \simeq \scr{C}$ obtained by base-change along $\epsilon_\gc \colon \bf{U} \to 1_\scr{C}$.

\subsubsection{Algebraic group completion and stable algebras in spaces}
We now return to the setting of spaces and restrict our attention to $E^\rm{u}_\infty$-algebras $\bf{R}$ with an isomorphism of commutative monoids $\pi_0\,\bf{R} \cong \bb{N}$. Since $\bb{N}$ has no non-trivial monoid automorphisms, there is in fact a \emph{unique} identification $\pi_0\,\bf{R} \cong \bb{N}$. In this case we have the following, opting to write $* \coloneq 1_\Spc$:

\begin{proposition} \label{prop:stable-algebra-spc}
    Let $\bf{R} \in \Alg_{E_\infty^\rm{u}}(\Spc)$ such that $\pi_0\, \bf{R} \cong \N$. Then we have:
    \begin{enumerate}[(i)]
        \item \label{enum:stable-algebra-spc-i} There is a unique map $\sigma \colon \bf{U} \to \bf{R}$ that induces an isomorphism on $\pi_0$.
        \item \label{enum:stable-algebra-spc-ii} There is a natural equivalence $\bf{R}[\sigma^{-1}] \simeq \Omega^\infty \bf{R}^\gc$ in $\Alg_{E_\infty^\rm{u}}(\Spc)$.
        \item \label{enum:stable-algebra-spc-iii} There is a unique \emph{group-completion augmentation} $\epsilon_\gc \colon \bf{R} \to \ast$ and it makes the following commutes
        \[\begin{tikzcd} \bf{U} \rar{\sigma} \arrow{rd}[swap]{\epsilon_\gc} & \bf{R} \dar{\epsilon_\gc} \\[-5pt]
        & \ast.\end{tikzcd}\]
        This induces an augmentation $\epsilon\colon \bf{R}/(\sigma-1) \to \ast$ of the stable algebra of $\bf{R}$ of \eqref{eqn:stable-algebra-augmentation}.
        \item \label{enum:stable-algebra-spc-iv} There is a natural equivalence of augmented $E_\infty^\rm{u}$-algebras in spaces $\bf{R}/(\sigma-1) \simeq \Omega^\infty(\bf{R}^\gc/\bb{S})$, where the map of spectra $\bb{S} \to \bf{R}^\gc$ is induced by $\sigma$ using $\bf{U}^\gc \simeq \bb{S}$ and the augmentation of $\Omega^\infty(\bf{R}^\gc/\bb{S})$ is given by applying $\Omega^\infty$ to the $0$ map $\bf{R}^\gc/\bb{S} \to 0$.
    \end{enumerate}
\end{proposition}

\begin{proof}For \eqref{enum:stable-algebra-spc-i}, it is given by the map adjoint to $\ast \to \fgt_{E_\infty^\rm{u}}\bf{R}$ picking out the component corresponding to $1 \in \bb{N}$. For \eqref{enum:stable-algebra-spc-ii} we use the group completion map $\bf{R} \to \Omega^\infty \bf{R}^{\gc}$ induces a map
\[\bf{R}[\sigma^{-1}] \lra (\Omega^\infty \bf{R}^{\gc})[\sigma^{-1}] \simeq \Omega^\infty \bf{R}^{\gc}\]
that is an equivalence by \cite[Corollary 13.41]{GKRW18}. For \eqref{enum:stable-algebra-spc-iii}, we observe that the monoidal unit $\ast = 1_\Spc$ is also the terminal object, so it is an $E_\infty^\rm{u}$-algebra and every algebra $\bf{R} \in \Alg_{E_\infty^\rm{u}}(\Spc)$ is uniquely augmented. 

For \eqref{enum:stable-algebra-spc-iv}, we recall that by \cref{lem:stable-algebra-gc} the map $\bf{R} \to \bf{R}[\sigma^{-1}]$ induces an equivalence
\[\bf{R}/(\sigma-1) \overset{\simeq}\lra \bf{R}[\sigma^{-1}](\sigma-1) \simeq (\Omega^\infty \bf{R}^\gc)/(\sigma-1),\]
where the right equivalence uses \eqref{enum:stable-algebra-spc-ii}. Its target is the pushout in $\Alg_{E_\infty^\rm{u}}(\Spc)$
    \[\begin{tikzcd}
    \bf{U} \arrow{r}{ \sigma'} \arrow{d}[swap]{\epsilon_\gc} & \Omega^\infty \bf{R}^\gc \arrow{d} \\[-5pt] \ast \simeq \Omega^\infty 0 \arrow{r} & (\Omega^\infty \bf{R}^\gc)/(\sigma-1),
\end{tikzcd}\]
where $\sigma'$ is the composition $\bf{U} \xrightarrow{\sigma} \bf{R} \to (\Omega^\infty \bf{R}^\gc)$. We claim that $\pi_0\,(\Omega^\infty \bf{R}^\gc)/(\sigma-1)$ is trivial. The inclusion $\scr{S}\rm{et} \hookrightarrow \Spc$ has a left adjoint given by $\pi_0 \colon \Spc \to \scr{S}\rm{et}$, and both functors are strong symmetric monoidal, so induce an adjunction on categories of $E_\infty^\rm{u}$-algebras, where $\Alg_{E_\infty^\rm{u}}(\scr{S}\rm{et}) \simeq \rm{CMon}$ is the $1$-category of symmetric monoids. In particular, $\pi_0 \colon \Alg_{E_\infty^\rm{u}}(\Spc) \to \rm{CMon}$ preserves all colimits, so $\pi_0((\Omega^\infty \bf{R}^\gc)/(\sigma-1))=0$ using the above diagram.

Thus $(\Omega^\infty \bf{R}^\gc)/(\sigma-1)$ is equivalent to its group completion as an algebra in $\Alg_{E_\infty^\rm{u}}(\Spc)$, and since group completion is a left adjoint then we can apply it to the previous diagram to show that $(\Omega^\infty \bf{R}^\gc)/(\sigma-1)$ is the pushout in $\Alg_{E_\infty^\rm{u}}^\rm{gp}(\Spc)$ of 
\[\begin{tikzcd}
    \Omega^\infty \bf{U}^\gc \arrow{r}{ \sigma'} \arrow{d}[swap]{\epsilon_\gc} & \Omega^\infty \bf{R}^\gc \arrow{d} \\[-5pt] \ast \simeq \Omega^\infty 0 \arrow{r} & (\Omega^\infty \bf{R}^\gc)/(\sigma-1).
\end{tikzcd}\]
By the Barratt--Priddy--Quillen--Segal theorem, $\bf{U}^\gc \simeq \bb{S}$, so using the adjoint equivalence $B^\infty \dashv \Omega^\infty$, we obtain an equivalence 
\[\bf{R}/(\sigma-1) \simeq (\Omega^\infty \bf{R}^\gc)/(\sigma-1) \simeq \Omega^\infty(\bf{R}^\gc/\bb{S}).\]
Finally, observe the augmentations are unique and $* \simeq \Omega^\infty 0$.  
\end{proof}

\begin{remark} Note there is a canonical algebra map $\Omega^\infty_0 \bf{R}^\rm{gc} \to \Omega^\infty(\bf{R}^\rm{gc}/\mathbb{S})$ and this is rationally is an equivalence. This explains why we consider $\bf{R}/(\sigma-1)$ as an analogue of $\colim_n \bf{R}(n)$. This seems to be the best one can do, see \cref{rem:pos-char-counterexample}.\end{remark}

\subsubsection{The pointed setting}\label{sec:stable-algebra-and-basepoints} It can be convenient to work in a pointed setting, and we explain how the above constructions interact with basepoints, e.g.~for $\scr{C} = \Spc$.

Let $\scr{C}_*$ denote the category of pointed objects in a category $\scr{C}$ with terminal object $t$. The forgetful map $\fgt_\ast \colon \scr{C}_* \to \scr{C}$ has a left adjoint $(-)_+ \colon \scr{C} \to \scr{C}_*$ given by $X_+ \coloneq (t \to X \sqcup t) \in \scr{C}_*$. 
If $\scr{C}$ is symmetric monoidal then there is a symmetric monoidal structure on $\scr{C}_*$ such that $(-)_+$ is strong symmetric monoidal and $\fgt_\ast$ is lax symmetric monoidal, so one gets a corresponding adjunction on categories of algebras. The functor $(-)_+$ commutes with the free algebra functor $\free_{E_\infty^\rm{u}}(-)$ since the right adjoints commute. In particular, $(\bf{U}_\scr{C})_+ \simeq \bf{U}_{\scr{C}_*}$ (since $1_{\scr{C}_*} \simeq (1_\scr{C})_+$) and hence if $\bf{R} \in \Alg_{E_\infty^\rm{u}}(\scr{C})$, then there are equivalences in $\Alg_{E_\infty^\rm{u}}(\scr{C}_*)$ and $\Alg_{E_\infty^\rm{u}}(\scr{C}_*)$
    \[(\bf{R}[\sigma^{-1}])_+ \simeq (\bf{R}_+)[\sigma_+^{-1}] \qquad \text{and} \qquad (\bf{R}/(\sigma-1_\scr{C}))_+ \simeq (\bf{R}_+)/(\sigma_+-1_{\scr{C}_*}).\]

\subsection{Stable algebras in the filtered and graded settings} To take advantage of natural gradings by ``rank'' and construct spectral sequences related to this, we lift the stable algebra construction to the filtered and graded settings.

\subsubsection{Filtered and graded objects}\label{sec:filtered-and-graded} We start by summarising some results from \cite[Section A.3]{KRS1}. Let $\bb{N}$ be the category whose objects are the nonnegative integers and whose morphisms are all identities. Similarly, let $\bb{N}_{\leq}$ be the category whose nonnegative objects are the integers and whose morphisms are given by a unique morphism $n \to m$ whenever $n \leq m$.

\begin{definition}Let $\scr{C}$ be a category.
	\begin{itemize}
		\item The category of \emph{(nonnegatively) graded objects} in $\scr{C}$ is $\Fun(\bb{N},\scr{C})$.
		\item The category of \emph{(nonnegatively) filtered objects} in $\scr{C}$ is $\Fun(\bb{N}_{\leq},\scr{C})$.
	\end{itemize}
\end{definition}

Addition makes $\bb{N}$ and $\bb{N}_{\leq}$ into symmetric monoidal categories, inducing Day convolution symmetric monoidal structures on the categories of graded and filtered objects \cite[Section A.2]{KRS1}.

For each integer $n \in \bb{N}$ we have functors $n^* \colon \Fun(\bb{N},\scr{C}) \to \scr{C}$ and $n^* \colon \Fun(\bb{N}_{\leq},\scr{C}) \to \scr{C}$ which admit left and right adjoints, denoted $n_!$ and $n_*$ respectively, as long as $\scr{C}$ has an initial object $\rm{i}$ and terminal object $\rm{t}$. We will mostly have use to $n_!$, which informally places an object in grading or filtration $n$. In \cite[Lemma A.12]{KRS1} we explained when these are (op)lax monoidal. In particular, the functor $0^* \colon \Fun(\bb{N}_{\leq},\scr{C}) \to \scr{C}$ is symmetric monoidal and thus induces a lax symmetric monoidality on the right adjoint $0_* \colon \scr{C} \to \Fun(\bb{N}_{\leq},\scr{C})$. The monoidal unit $1_\scr{C}$ thus gives a commutative algebra object $S(0) \coloneq 0_* 1_\scr{C}$ in $\Fun(\bb{N}_{\leq},\scr{C})$ and hence in $\Fun(\bb{N}_{\leq},\scr{C})$. If $\scr{C}$ is pointed we can identify $S(0)$-modules in $\Fun(\bb{N}_{\leq},\scr{C})$ with those filtered objects $X$ such that the maps $X(i) \to X(i+1)$ factor over the terminal object, and define $\rm{gr}$ as the left adjoint to the forgetful functor $\rm{Mod}_{S(0)}(\Fun(\bb{N}_{\leq},\scr{C})) \to \Fun(\bb{N}_{\leq},\scr{C})$. Pulling back along the unique map $p \colon \bb{N}_{\leq} \to \ast$ induces a functor $\rm{const} \coloneq p^* \colon \scr{C} \to \Fun(\bb{N}_{\leq},\scr{C})$ which admits a left adjoint $\colim \coloneq p_!$ and right adjoint $\lim \coloneq p_*$. 

\begin{lemma}\label{lem:colim-gr-sym-mon-left-adjoints} \quad
	\begin{enumerate}[(i)]
		\item \label{enum:colim-gr-sym-mon-left-adjoints-i} $\colim \colon \Fun(\bb{N}_{\leq},\scr{C}) \to \scr{C}$ is a symmetric monoidal left adjoint.
		\item \label{enum:colim-gr-sym-mon-left-adjoints-ii} $\gr \colon \Fun(\bb{N}_{\leq},\scr{C}) \to \Fun(\bb{N},\scr{C})$ is a symmetric monoidal left adjoint if $\scr{C}$ is pointed.
	\end{enumerate}
\end{lemma}

\begin{proof}For \eqref{enum:colim-gr-sym-mon-left-adjoints-i} we apply \cite[Lemma A.6 (ii)]{KRS1} to $\colim = p_!$. For \eqref{enum:colim-gr-sym-mon-left-adjoints-ii} we note that the left adjoint to the forgetful functor $\rm{Mod}_{S(0)}(\Fun(\bb{N}_{\leq},\scr{C})) \to \Fun(\bb{N}_{\leq},\scr{C})$ is given by tensoring with the idempotent $S(0)$. \end{proof}

The unique functor $\iota \colon \bb{N} \to \bb{N}_\le$ that is the identity on objects induces a restriction functor $\iota^* \colon \Fun(\bb{N}_{\leq},\scr{C}) \to \Fun(\bb{N},\scr{C})$, which has a left adjoint given by left Kan extension
\[\iota_! \colon \Fun(\bb{N},\scr{C}) \lra \Fun(\bb{N}_{\leq},\scr{C}).\]
Explicitly, we have $\iota_!(-) \simeq \bigsqcup_{n \in \Z}{n_! n^*(-)}$. Since $\iota$ is strong symmetric monoidal, by \cite[Lemma A.6]{KRS1} the functor $\iota_!$ is strong symmetric monoidal and $\iota^*$ is lax symmetric monoidal.

\begin{lemma} \label{lem:basic-props-rank-filtration} \,
    \begin{enumerate}[(i)]
        \item \label{enum:basic-props-rank-filtration-i}  There is a natural equivalence $\colim \circ \iota_! \simeq \colim$ of functors $\Fun(\bb{N},\scr{C}) \to \scr{C}$.
        \item \label{enum:basic-props-rank-filtration-ii} There is a natural equivalence $\gr \circ \iota_! \simeq \id_{\Fun(\bb{N},\scr{C})}$ of functors $\Fun(\bb{N},\scr{C}) \to \Fun(\bb{N},\scr{C})$ if $\scr{C}$ is pointed.
    \end{enumerate}
\end{lemma}

\begin{proof}
    Part \eqref{enum:basic-props-rank-filtration-i} follows similarly from $\colim$ being the left adjoint to $\rm{const}$ and the natural equivalence $\iota^* \circ \rm{const} \simeq \rm{const}$ of right adjoints. The functor $\gr$ is the left adjoint to the forgetful functor 
    \[i \colon \Fun(\bb{N},\scr{C}) \simeq \rm{Mod}_{S(0)}(\Fun(\bb{N}_{\leq},\scr{C})) \lra \Fun(\bb{N}_{\leq},\scr{C}),\] 
    exhibiting graded objects as filtered objects whose structure maps factor through the zero object. There is a natural equivalence $\iota^* \circ i \simeq \id_{\Fun(\bb{N},\scr{C})}$ of right adjoints, and passing to left adjoints gives the first part of \eqref{enum:basic-props-rank-filtration-ii}.
\end{proof}

\subsubsection{Stable algebras in the filtered setting} We work in the category $\Fun(\bb{N}_\le,\scr{C})$ of filtered objects, where $\scr{C}$ is a presentable symmetric monoidal $\infty$-category. Our goal is to make sense of a ``filtered stable algebra'', whose colimit will agree with the stable algebra of \cref{def:stable-algebra}; it differs from applying the constructions in that section only by shifting the monoidal unit in grading. As the stable algebra can serve as a replacement for group completion, we will not need a ``filtered group completion construction''. 

\medskip

We consider the filtered algebra $\bf{U}^\fil \coloneq \free_{E_\infty^\rm{u}}(1_! 1_\scr{C}) \in \Alg_{E_\infty^\rm{u}}(\Fun(\bb{N}_{\leq},\scr{C}))$, where $1_! \colon \scr{C} \to \Fun(\bb{N}_{\leq},\scr{C})$ is the functor that places an object in filtration $1$. This is \emph{not} an instance of the previous theory applied to $\Fun(\bb{N}_{\leq},\scr{C})$. It comes with the following additional structure:
\begin{enumerate}[\noindent (1)] \item 
The \emph{filtered group completion augmentation} 
\[\epsilon^\fil_\gc \colon \bf{U}^\fil \lra 1_{\Fun(\bb{N}_{\leq},\scr{C})},\]
is defined as the adjoint to $1_! 1_\scr{C} \to \fgt_{E_\infty^\rm{u}} 1_{\Fun(\bb{N}_{\leq},\scr{C})} \simeq 0_! 1_\scr{C}$ which is in turn adjoint to $\id \colon 1_\scr{C} \to 1^* 0_! 1_\scr{C} \simeq 1_\scr{C}$.
\item The \emph{filtered codiagonal}
\[\nabla^\fil \colon \bf{U}^\fil \otimes \bf{U}^\fil \lra \bf{U}^\fil\]
is given under $\bf{U}^\fil \sqcup^{E_\infty^\rm{u}} \bf{U}^\fil \simeq \bf{U}^\fil \otimes \bf{U}^\fil$ by the map that is $\id_{\bf{U}^\fil}$ on each component.
\end{enumerate}

We can then make the following definition, using the map $1 \colon \bf{U}^\fil \to \bf{R}$ which is adjoint to the map $1_\scr{C} \to 1^* \fgt_{E_\infty^u}(\bf{R})$ that picks out the unit.

\begin{definition}\label{def:filtered-stable-algebra} The \emph{filtered stable algebra} functor is
\begin{align*}(-)/(\sigma-1) \colon \Alg_{E_\infty^\rm{u}}(\Fun(\bb{N}_{\leq},\scr{C}))^{\bf{U}^\fil/} &\lra \Alg_{E_\infty^\rm{u}}(\Fun(\bb{N}_{\leq},\scr{C})) \\
(\bf{U}^\fil \xrightarrow{\sigma} \bf{R}) &\longmapsto \bf{U}^\fil \otimes_{\bf{U}^\fil \otimes \bf{U}^\fil} \bf{R}\end{align*}
where the relative tensor product is with respect to the maps $(\sigma,1) \colon \bf{U}^\fil \otimes \bf{U}^\fil \to \bf{U}^\fil$ and $\nabla^\fil \colon \bf{U}^\fil \otimes \bf{U}^\fil \to \bf{U}^\fil$.\end{definition}

Thus by definition, we have the following pushout in $\Alg_{E_\infty^\rm{u}}(\Fun(\bb{N}_{\leq},\scr{C}))$
\[\begin{tikzcd}
    \bf{U}^\fil \sqcup^{E_\infty^\rm{u}} \bf{U}^\fil \arrow{r}{(\sigma, 1)} \arrow{d}[swap]{\nabla^\fil} & \bf{R} \arrow{d} \\[-5pt] \bf{U}^\fil \arrow{r} & \bf{R}/(\sigma-1).
\end{tikzcd}\]
The proof of \cref{lem:stable-alg-first-props} and its preceding computation go through. In particular, the filtered algebra stable construction preserves colimits and may be constructed in terms of the filtered group completion augmentation: there is a natural pushout in $\Alg_{E_\infty^\rm{u}}(\Fun(\bb{N}_{\leq},\scr{C}))$ 
    \begin{equation}\label{eqn:filered-r-mod-sigma-1} \begin{tikzcd}
        \bf{U}^\fil \arrow{r}{\sigma} \arrow{d}[swap]{\epsilon^\fil_\gc} & \bf{R} \arrow{d} \\[-5pt] 1_{\Fun(\bb{N}_{\leq},\scr{C})} \arrow{r} & \bf{R}/(\sigma-1).
    \end{tikzcd}\end{equation}

As $\colim$ and $\gr$ are symmetric monoidal left adjoints by \cref{lem:colim-gr-sym-mon-left-adjoints}, they induce functors 
\begin{align*}\colim \colon \Alg_{E_\infty^\rm{u}}(\Fun(\bb{N}_{\leq},\scr{C}))^{\bf{U}^\fil/} &\lra \Alg_{E_\infty^\rm{u}}(\scr{C})^{\colim \bf{U}^\fil/} \\
\gr \colon \Alg_{E_\infty^\rm{u}}(\Fun(\bb{N}_{\leq},\scr{C}))^{\bf{U}^\fil/} &\lra \Alg_{E_\infty^\rm{u}}(\Fun(\bb{N},\scr{C}_*))^{\gr(\bf{U}^\fil)/}\end{align*}
interacting with the group completion augmentation and stable algebra construction as follows:

\begin{lemma} \label{lem:filtered-stable-algebra}\,
    \begin{enumerate}[(i)]
        \item \label{enum:filtered-stable-algebra-i} Taking colimits induces an equivalence
        \[\colim\big(\epsilon^\fil_\gc \colon \bf{U}^\fil \to 1_{\Fun(\bb{N}_{\leq},\scr{C})}\big) \simeq \big(\epsilon_\gc \colon \bf{U} \to 1_\scr{C}\big)\]
        \item \label{enum:filtered-stable-algebra-ii}There is natural equivalence in $\Alg_{E_\infty^\rm{u}}(\scr{C})$
            \[\colim(\bf{R}/(\sigma-1)) \simeq \colim(\bf{R})/(\colim(\sigma)-1).\]
        \item \label{enum:filtered-stable-algebra-iii} Taking associated graded induces an equivalence
        \[\gr(\epsilon^\fil_\gc \colon \bf{U}^\fil \to 1_{\Fun(\bb{N},\scr{C})}) \simeq (\epsilon_\rm{can} \colon \gr(\bf{U}^\fil) \to 1_{\Fun(\bb{N},\scr{C}_*)}).\]
        \item \label{enum:filtered-stable-algebra-iv} There is a natural equivalence in $\Alg_{E_\infty^\rm{u}}(\scr{C})$
            \[\gr(\bf{R}/(\sigma-1)) \simeq \gr(\bf{R})/(\gr(\sigma)).\]
    \end{enumerate}
\end{lemma}

Assuming $\scr{C} = \Spc_*$, $\DQ$, or any other category with appropriate homology theory $H_*$, we can extract spectral sequences, e.g.~using \cite[Theorem 10.10]{GKRW18}. We opt to reindex them so all contributions to a homology group of the abutment lie on the same row:

\begin{notation}\label{not:rank-ss-indexing} The \emph{rank spectral sequence indexing convention} for spectral sequence is 
\[E^1_{n,d} = H_{n,d} \Longrightarrow H_d \qquad \text{with differential of bidegree $(-r,-1)$},\]
and can be obtained from the convention in \cite[Theorem 10.10]{GKRW18} by replacing $E^1_{p,q} = H_{p+q,p} \Rightarrow H_{p+q}$ with differentials of bidegree $(-r,r-1)$ with  by setting $n \coloneq p$ and $d \coloneq p+q$.\end{notation}

In the case of interest we get the following, using that half-plane spectral sequence with exiting differentials converges strongly and that the filtered object admits a multiplication:

\begin{lemma}\label{lem:filtered-rank-ss} Given $(\bf{U}^\fil \xrightarrow{\sigma} \bf{R}) \in \Alg_{E_\infty^\rm{u}}(\Fun(\bb{N}_{\leq},\scr{C}))^{\bf{U}^\fil/}$, there is a strongly convergent multiplicative \emph{rank spectral sequence}
    \[E^1_{p,q}= H_{n,d}(\gr(\bf{R})/(\gr(\sigma))) \Longrightarrow H_{d}(\colim(\bf{R})/(\colim(\sigma)-1))\]
with differentials $d^r$ of bidegree $(-r,-1)$.
\end{lemma}

We can consider $\bf{U}^\fil$ as augmented through the filtered group completion augmentation $\epsilon^\fil_\gc$ and thus make sense of the category $\smash{\Alg^\aug_{E_\infty^\rm{u}}(\Fun(\bb{N}_{\leq},\scr{C}))^{\bf{U}^\fil/}}$. An object in this category yield a commutative square
\[\begin{tikzcd} \bf{U}^\fil \rar{\sigma} \dar[swap]{\epsilon^\fil_\gc} & \bf{R} \dar{\epsilon} \\[-5pt]
1_{\Fun(\bb{N}_{\leq},\scr{C})} \rar{\id} & 1_{\Fun(\bb{N}_{\leq},\scr{C})}\end{tikzcd}\]
and by the pushout square \eqref{eqn:filered-r-mod-sigma-1} this induces an augmentation on $\bf{R}/(\sigma-1)$. Thus we get a lift of the filtered stable algebra functor to
\[(-)/(\sigma-1) \colon \Alg^\aug_{E_\infty^\rm{u}}(\Fun(\bb{N}_{\leq},\scr{C}))^{\bf{U}^\fil/} \lra \Alg^\aug_{E_\infty^\rm{u}}(\Fun(\bb{N}_{\leq},\scr{C})).\]
There are then analogous results in the augmented setting to \cref{lem:filtered-stable-algebra}. In particular, we can apply $\cot_{E_\infty^\rm{nu}}$ after passing to augmentation ideal to get a spectral sequence, to get an analogue of the rank spectral sequence for indecomposables: 

\begin{lemma}\label{lem:filtered-rank-ss-indec} Given $\bf{R} \in \Alg^\aug_{E_\infty^\rm{u}}(\Fun(\bb{N}_{\leq},\scr{C}))^{\bf{U}^\fil/}$, there is a strongly convergent \emph{indecomposable rank spectral sequence}
    \[{}^\infty E^1_{p,q}= H^{E_\infty}_{n,d}(\gr(\bf{R})/(\gr(\sigma))) \Longrightarrow H^{E_\infty}_{d}(\colim(\bf{R})/(\colim(\sigma)-1))\]
with differentials $d^r$ of bidegree $(-r,-1)$. 
\end{lemma}

The unit natural transformation $\id \to \triv_{E_\infty^\rm{nu}} \cot_{E_\infty^\rm{nu}}(-)$ induces the canonical map from homology to $E_\infty$-homology:

\begin{lemma}\label{lem:filtered-rank-ss-indec-map} Given $\bf{R} \in \Alg^\aug_{E_\infty^\rm{u}}(\Fun(\bb{N}_{\leq},\scr{C}))^{\bf{U}^\fil/}$, there is a map of multiplicative spectral sequence from \cref{lem:filtered-rank-ss} to \cref{lem:filtered-rank-ss-indec}, if the target is given the trivial multiplication.
\end{lemma}

\subsubsection{Stable algebras in the graded setting}\label{sec:stable-algebras-graded}
In the graded case we consider $\bf{U}^\gr \coloneq \rm{free}_{E_\infty^\rm{u}}(1_!1_\scr{C}) \in  \Alg_{E_\infty^\rm{u}}(\Fun(\bb{N},\scr{C}))$, but unlike in the previous cases, there need not exist an analogue of the group-completion augmentation. Rather, we will use the ``rank filtration'' to lift $\bb{N}$-graded objects to $\bb{N}$-filtered objects, and apply the techniques for the filtered case from the previous subsection. This is given by left Kan extension along the unique functor $\iota \colon \bb{N} \to \bb{N}_{\leq}$ that is the identity on objects, which is a symmetric monoidal left adjoint by \cref{sec:filtered-and-graded} and so induces
\[\rm{rkfil} \coloneq \iota_!^\rm{alg} \colon \Alg_{E_\infty^\rm{u}}(\Fun(\bb{N},\scr{C})) \lra \Alg_{E_\infty^\rm{u}}(\Fun(\bb{N}_{\leq},\scr{C})).\]

\begin{lemma}\label{lem:ugr-vs-ualg} $\rm{rkfil}(\bf{U}^\gr) \simeq \bf{U}^\fil$ and $\gr(\bf{U}^\fil) \simeq \bf{U}^\gr$.
\end{lemma}

\begin{proof}The left equivalence is given by
\[\rkfil(\bf{U}^\gr) \simeq \free_{E^\rm{u}_\infty} (\iota_! 1_! 1_\scr{C}) \simeq \free_{E^\rm{u}_\infty}(1_! 1_\scr{C}) \simeq \bf{U}^\fil.\]
using firstly that there are equivalences 
\[\fgt_{E^\rm{u}_\infty} \circ \iota^{*,\alg} \simeq \iota^* \circ \fgt_{E^\rm{u}_\infty}\qquad \text{and} \qquad \free_{E^\rm{u}_\infty} \circ \iota_! \simeq \iota_!^\alg \circ \free_{E^\rm{u}_\infty},\]
the left by construction and the right by passing to left adjoints. It uses secondly that there are similarly equivalences  $1^* \iota^* \simeq 1^*$ and $\iota_! \circ 1_! \simeq 1_!$. Moreover, applying $\gr$ to this and using \cref{lem:basic-props-rank-filtration} \eqref{enum:basic-props-rank-filtration-ii} we get the right equivalence.
\end{proof}

Thus $\iota^\rm{alg}_!$ induces a functor on slice categories under $\bf{U}^\fil$ or $\bf{U}^\gr$, and applying the filtered stable algebra construction we get:

\begin{definition}The \emph{rank-filtered stable algebra} functor is
\[\rkfil(-)/(\rkfil(\sigma)-1) \colon \Alg_{E_\infty^\rm{u}}(\Fun(\bb{N},\scr{C}))^{\bf{U}^\gr/} \lra \Alg_{E_\infty^\rm{u}}(\Fun(\bb{N}_{\leq},\scr{C})).\]
\end{definition}

Applying \cref{lem:basic-props-rank-filtration} and \cref{lem:filtered-stable-algebra}, we obtain:
            
\begin{lemma} \label{lem:graded-stable-algebra}\,
\begin{enumerate}[(i)]
\item \label{enum:graded-stable-algebra-i} There is a natural equivalence in $\Alg_{E_\infty^\rm{u}}(\scr{C})$
\[\colim(\rkfil(\bf{R})/(\rkfil(\sigma)-1)) \simeq (\colim \bf{R})/(\colim(\sigma)-1).\]
\item \label{enum:graded-stable-algebra-ii} There is a natural equivalence in $\Alg_{E_\infty^\rm{u}}(\Fun(\bb{N},\scr{C}))$
\[\gr(\rkfil(\bf{R})/(\rkfil(\sigma)-1)) \simeq \bf{R}/(\sigma).\]
\end{enumerate}
\end{lemma}

\begin{notation}We will opt to abbreviate $(\colim \bf{R})/(\colim(\sigma)-1)$ to $\bf{R}/(\sigma-1)$, as we believe there is no risk for confusion: the only way to make sense of the latter is to apply the functor $\colim \colon \Fun(\bb{N},\scr{C}) \to \scr{C}$ first.
\end{notation}

To take indecomposables of rank-filtered stable algebras we require augmentations. To construct these, let $\overline{\epsilon} \colon \bf{U}^\gr \to \bf{V} \coloneq \iota^{*,\rm{alg}} 1_{\Fun(\bb{N}_{\leq},\scr{C})}$ be adjoint to $\epsilon^\fil_\gc$ and consider the category
\[\left(\Alg_{E_\infty^\rm{u}}(\Fun(\bb{N},\scr{C}))_{/\bf{V}}\right)^{\overline{\epsilon}/}\]
of factorisations $\bf{U}^\gr \xrightarrow{\sigma} \bf{R} \xrightarrow{\epsilon} \bf{V}$ of $\overline{\epsilon}$. Given an object $\bf{R}$ of this, we get a commutative square
\[\begin{tikzcd}
    \bf{U}^\fil \simeq \rkfil(\bf{U}^\gr) \rar{\rkfil(\sigma)} \dar[swap]{\epsilon^\fil_\gc} & \rkfil(\bf{R}) \rar{\rkfil(\epsilon)} & \rkfil(\bf{V}) \dar \\[-5pt]
    1_{\Fun(\bb{N}_{\leq},\scr{C})} \arrow{rr}{\id} & & 1_{\Fun(\bb{N}_{\leq},\scr{C})}
\end{tikzcd}\]
with right map the counit of $\iota^\alg_! \dashv \iota^{*,\alg}$, and through its description as a pushout in \eqref{eqn:filered-r-mod-sigma-1}, this provides an augmentation on $\rkfil(\bf{R})/(\rkfil(\sigma) -1)$. We obtain from this spectral sequences, again in the indexing convention of \cref{not:rank-ss-indexing}, using \cref{lem:filtered-rank-ss,lem:filtered-rank-ss-indec,lem:filtered-rank-ss-indec-map}.

\begin{lemma}\label{lem:rank-spectral-sequences} Let $(\bf{U}^\gr \xrightarrow{\sigma} \bf{R} \xrightarrow{\epsilon} \bf{V}) \in \left(\Alg_{E_\infty^\rm{u}}(\Fun(\bb{N},\scr{C}))_{/\bf{V}}\right)^{\overline{\epsilon}/}$.
\begin{enumerate}[(i)]
    \item \label{enum:rank-spectral-sequences-i}There is a strongly convergent \emph{rank spectral sequence}
    \[E^1_{p,q}= H_{n,d}(\bf{R}/(\sigma)) \Longrightarrow H_{d}(\bf{R}/(\sigma-1))\]
with differentials $d^r$ of bidegree $(-r,-1)$. 
    \item \label{enum:rank-spectral-sequences-ii}There is a strongly convergent \emph{rank spectral sequence for indecomposables}
    \[{}^\infty E^1_{n,d}= H_{n,d}^{E_\infty} (\bf{R}/(\sigma)) \Longrightarrow H_d^{E_\infty}(\bf{R}/(\sigma-1))\]
    with $d^r$ of bidegree $(-r,-1)$.
    \item \label{enum:rank-spectral-sequences-iii}There is a map of multiplicative spectral sequences $E^r_{*,*} \to {}^\infty E^r_{*,*}$, where the multiplicative structure on the target is trivial, which agrees with the canonical map from homology to $E_\infty$-homology on the $E^1$-page and abutment.
\end{enumerate}
\end{lemma}

\subsection{Stable algebras in the graded setting over $\DQ$} \label{sec:graded-over-dq} In this paper we need the case $\scr{C}= \DQ$ and its relationship to $\Spc$. This has special features that allow for easier description of the rank-filtered stable algebra: the underlying object of $\rkfil(\bf{R})/(\rkfil(\sigma)-1)$ is
\begin{equation}\label{eqn:underlying-graded-rank-filtered} \bf{R}(0) \overset{\sigma}\lra \bf{R}(1) \overset{\sigma}\lra \bf{R}(2) \overset{\sigma}\lra\cdots.\end{equation}

\begin{remark}\label{rem:pos-char-counterexample} This is not true without working in characteristic $0$ by the existence of Dyer--Lashof operations: one cannot give \eqref{eqn:underlying-graded-rank-filtered} the structure of an $E_2^\rm{u}$-algebra so the canonical map 
\[\fgt_{E_\infty^\rm{u}}(\bf{R}) \lra  \gr(\bf{R}(0) \overset{\sigma}\lra \bf{R}(1) \overset{\sigma}\lra \cdots) \simeq (\fgt_{E_\infty^\rm{u}} \bf{R})/\sigma\]
lifts to an $E_2^\rm{u}$-algebra map. (We warn the reader that, unless one works rationally, $(\fgt_{E_\infty^\rm{u}} \bf{R})/\sigma$ need not agree with the underlying object of $\bf{R}/(\sigma)$.)

Let us take $\bf{R}$ to be the $E_\infty^\rm{u}$-algebra in $\scr{D}_{\bb{F}_2}$ given by $\free_{E_\infty^\rm{u}}(1_! 1)$. Suppose there existed an $\bf{X} \in \Alg_{E_\infty^\rm{u}}(\Fun(\bb{N}_{\leq},\scr{D}_{\bb{F}_2}))$ and an $E_\infty^\rm{u}$-algebra map $f \colon \bf{R} \to \gr(\bf{X})$ so that
\[\fgt_{E_\infty^\rm{u}}(\bf{X}) \simeq (\bf{R}(0) \overset{\sigma}\lra \bf{R}(1) \overset{\sigma}\lra \cdots)\]
and under the associated graded of this equivalence, $\fgt_{E_\infty^\rm{u}}(f)$ were the quotient map $\fgt_{E_\infty^\rm{u}}(\bf{R}) \to (\fgt_{E_\infty^\rm{u}} \bf{R})/\sigma$. Using knowledge of the homology of free $E_\infty^\rm{u}$-algebra \cite[Section 16.3]{GKRW18}, for $\sigma$ the canonical generator of $H_{1,0}(\bf{R})$ we have both 
\[f_*(\sigma)=0 \in H_{1,0}(\bf{X}) \qquad \text{and} \qquad H_{2,1}(\gr(\bf{X}))\cong\bb{F}_2 \{f_* Q^1(\sigma) \} \cong \bb{F}_2,\]
yielding a contradiction. In fact, \eqref{eqn:underlying-graded-rank-filtered} can not even have an $E_2^\rm{nu}$-algebra structure.
\end{remark}

\subsubsection{The underlying object of the rank filtration}
Several objects that appear in the graded case are equivalent: there is a map $\free_{E_1^\rm{u}}(1_! 1_\DQ) \to \smash{\fgt_{E_\infty^\rm{u}}^{E_1^\rm{u}} \bf{U}^\gr}$ of $E_1^\rm{u}$-algebras adjoint to the map $1_\DQ \to 1^* \fgt_{E_\infty^\rm{u}} \bf{U}^\gr$ selecting the generator of the target. A direct computation yields:

\begin{lemma} \label{lem:equivalences-of-algebras-in-DQ} The following are are equivalences of $E_1^\rm{u}$-, respectively $E_\infty^\rm{u}$-algebras,
\[\free_{E_1^\rm{u}}(1_! 1_\DQ) \lra \fgt_{E_\infty^\rm{u}}^{E_1^\rm{u}} \bf{U}^\gr \qquad \text{and} \qquad \bf{U}^\rm{gr} \lra \bf{V}.\]
\end{lemma}

We write $\Q[t] \coloneq \smash{\free_{E_1^\rm{u}}(1_! 1_\DQ)} \in \smash{\Alg_{E_1^\rm{u}}(\Fun(\bb{N},\DQ))}$, inspired by the notation in \cite[Remark 3.1.5]{lurie2015rotation}, and observe the above lemma endows it with a preferred lift to an $E_\infty^\rm{u}$-algebra. There are then equivalence of categories
   \begin{align*}\left(\Alg_{E_\infty^\rm{u}}(\Fun(\bb{N},\DQ))_{/\bf{V}}\right)^{\ol{\epsilon}/} &\simeq \left(\Alg_{E_\infty^\rm{u}}(\Fun(\bb{N},\DQ))_{/\bf{U}^\gr}\right)^{\id_{\bf{U}}/} \\
    &\simeq \left(\Alg_{E_\infty^\rm{u}}(\Fun(\bb{N},\DQ))_{/\bb{Q}[t]}\right)^{\id_{\bb{Q}[t]}/}.\end{align*}
Working in the latter, we may change perspective to that of $\bb{Q}[t]$-modules, using the equivalence of symmetric monoidal categories
 \begin{equation}\label{eqn:slice-as-qt-mod}
    \left(\Alg_{E_\infty^\rm{u}}^\rm{aug}(\Fun(\bb{N},\DQ))_{/\bb{Q}[t]}\right)^{\id_{\bb{Q}[t]}/} \simeq \Alg_{E_\infty^\rm{u}}^\rm{aug}(\rm{Mod}_{\Q[t]}(\Fun(\bb{N},\DQ))).\end{equation}
The following is the rational analogue of \cite[Proposition 3.1.6]{lurie2015rotation} and is proven by repeating its proof in the category $\DQ$ of rational spectra, rather than spectra (note Lurie uses the opposition direction for filtrations).

\begin{proposition} \label{prop:filtered-objects-as-modules}
There is an equivalence of symmetric monoidal categories
\[\theta \colon \Fun(\bb{N}_{\leq},\DQ) \simeq \rm{Mod}_{1_{\Fun(\bb{N}_{\leq},\DQ)}}(\Fun(\bb{N}_{\leq},\DQ)) \xrightarrow[\iota^*]{\simeq} \rm{Mod}_{\bb{Q}[t]}(\Fun(\bb{N},\DQ)).\]
given by sending the filtered object $X(0) \to X(1) \to \cdots $ to the graded object $X(0), X(1), \ldots$ viewed as a $\bb{Q}[t]$-module such that the generator acts via the maps $X(n) \to X(n+1)$. 
\end{proposition}

\begin{example}There is a natural equivalence $\theta^{-1}(\free_{\bb{Q}[t]}(-)) \simeq \iota_!(-)$ of functors $\Fun(\bb{N},\DQ) \to \Fun(\bb{N}_{\leq},\DQ)$, as they have the same universal property:
\begin{align*} \Map_{\Fun(\bb{N}_{\leq},\DQ)}(\iota_! X,Y) &\simeq \Map_{\Fun(\bb{N},\DQ)}(X,\iota^*Y) \\
&\simeq \Map_{\rm{Mod}_{\bb{Q}[t]}(\Fun(\bb{N},\DQ))}(\free_{\bb{Q}[t]} X,\theta(Y)).\end{align*}
\end{example}

By symmetric monoidality of $\theta$, it induces an equivalence
\[\theta^\alg \colon \Alg_{E_\infty^\rm{u}}^\rm{aug}(\Fun(\bb{N}_{\leq},\DQ)) \xrightarrow{\simeq} \Alg_{E_\infty^\rm{u}}^\rm{aug}\big(\rm{Mod}_{\Q[t]}(\Fun(\bb{N},\DQ))\big)\]
Now we can prove the main result of this subsection:

\begin{theorem} \label{thm:stable-algebras-DQ} The filtered (unital) stable algebra construction 
    \[\rkfil(-)/(\rkfil(\sigma)-1) \colon \left(\Alg_{E_\infty^\rm{u}}^\rm{aug}(\Fun(\bb{N},\DQ))_{/\bb{Q}[t]}\right)^{\id_{\bb{Q}[t]}/} \lra \Alg_{E_\infty^\rm{u}}^\rm{aug}(\Fun(\bb{N}_{\leq},\DQ))\]
    is given by applying the equivalence \eqref{eqn:slice-as-qt-mod} followed by an inverse of the equivalence $\theta^\alg$. In particular, the rank-filtered stable algebra functor is an equivalence of categories.
\end{theorem}

\begin{proof} Under \eqref{eqn:slice-as-qt-mod}, the augmented rank-filtered stable algebra is the functor
\begin{align*}\rkfil(-)/(\rkfil(\sigma)-1) \colon \Alg_{E_\infty^\rm{u}}^\rm{aug}(\rm{Mod}_{\bb{Q}[t]}(\Fun(\bb{N},\DQ))) &\lra \Alg_{E_\infty^\rm{u}}^\rm{aug}(\Fun(\bb{N}_{\leq},\DQ)) \\
\bf{R} &\longmapsto \rkfil(\bf{R}) \otimes_{\bb{Q}[t]^\fil} 1_{\Fun(\bb{N}_{\leq},\scr{C})} \end{align*}
where we write $\bb{Q}[t]^\fil \coloneq \rkfil(\bb{Q}[t]) \simeq \free_{E_\infty^\rm{u}}(1_! 1_{\DQ})$, $\bb{Q}[t]^\fil \to \rkfil(\bf{R})$ is obtained by applying $\rkfil$ to the unit $\bb{Q}[t] \to \bf{R}$, and $\bb{Q}[t]^\fil \to 1_{\Fun(\bb{N}_{\leq},\scr{C})}$ is the filtered group completion augmentation $\epsilon_\gc^\fil$.

We must prove its composition with $\theta^\alg$ is equivalent to the identity. As $\theta^\alg$ is given by applying $i^{*,\alg}$ and the equivalence $\smash{i^{*,\alg} 1_{\Fun(\bb{N}_{\leq},\DQ)}} \simeq \bb{Q}[t]$, and $\rkfil(-)=\smash{\iota_!^\alg}$ (we will use the latter for the remainder of this proof), the counit of $\smash{\iota_!^\alg \dashv \iota^{*,\alg}}$ yields natural transformations
\[\id \lra \theta^\alg \circ \iota_!^\alg(-)/(\iota_!^\alg(\sigma)-1) \qquad \text{and} \qquad (\theta^\alg)^{-1} \lra \iota_!^\alg(-)/(\iota_!^\alg(\sigma)-1).\]
Since both sides of the right side preserve small colimits it suffices to verify it is an equivalence on $\bf{R}$ of the form $\free_{E_\infty^\rm{u}}(\free_{\bb{Q}[t]}(X))$ for $X \in \Fun(\bb{N},\DQ)_{/\bb{Q}[t]}$. Using $\theta^{-1}(\free_{\bb{Q}[t]}(X)) \simeq \iota_! X$, on the one hand we have
\[(\theta^\alg)^{-1}(\free_{E_\infty^\rm{u}}(\free_{\rm{Mod}^{E_\infty}_{\bb{Q}[t]}}(X))) \simeq \free_{E_\infty^\rm{u}}(\iota_! X).\]
On the other hand we have 
\begin{align*}\iota_!^\alg(\free_{E_\infty^\rm{u}}(\free_{\bb{Q}[t]}(X)))/(\iota_!^\alg(\sigma)-1) &\simeq \free_{E_\infty^\rm{u}}(\free_{\bb{Q}^\fil[t]}(\iota_! X)) \otimes_{\bb{Q}[t]^\fil} 1_{\Fun(\bb{N}_{\leq},\scr{C})} \\
&\simeq \free_{E_\infty^\rm{u}}(\iota_!X).\end{align*}
The map between these is an equivalence because upon applying $\fgt_{E_\infty^\rm{u}}$ and restricting to $i_! X$, it is given by the inclusion $i_! X \to \fgt_{E_\infty^\rm{u}}(\free_{E_\infty^\rm{u}}(i_!X))$.\end{proof}

\begin{corollary}The underlying object of the stable algebra $\rkfil(\bf{R})/(\rkfil(\sigma)-1)$ is given by 
    \[(\bf{R}(0) \xrightarrow{\sigma} \bf{R}(1) \xrightarrow{\sigma} \cdots ) \in \Fun(\bb{N}_{\leq},\DQ).\]
\end{corollary}

\begin{remark}For $\bf{R} \in \Alg_{E_\infty^\rm{u}}(\Fun(\bb{N},\scr{C}))$, the multiplication of $\bf{R}$ gives maps $\bf{R}(n) \otimes \bf{R}(m) \to \bf{R}(n+m)$ that are up to homotopy compatible with $\sigma$. An interpretation of this corollary is that for $\scr{C} = \DQ$, this lifts to an $E_\infty^\rm{u}$-structure.
\end{remark}

Taking colimits of underlying objects, we get an equivalence and maps $\bf{R}(0) \to \bf{R}(1) \to \cdots \to  \colim_n\,\bf{R}(n) \simeq \bf{R}/(\sigma-1)$ and thus there are ``stabilisation'' maps in homology 
\[H_{0,d}(\bf{R}) \lra H_{1,d}(\bf{R}) \lra \cdots \lra H_d(\bf{R}/(\sigma-1))\]
and thus the filtration in the rank spectral sequence of \cref{lem:rank-spectral-sequences} \eqref{enum:rank-spectral-sequences-i} agrees with filtering by rank the underlying object. 

\subsubsection{Examples from spaces} We give a general method to produce examples:

\begin{lemma} \label{lem:rational-algebras-from-spaces}
    If $\bf{R} \in \Alg_{E_\infty^\rm{u}}(\Fun(\bb{N},\Spc))$ so that $\bf{R}(n)$ is connected all $n \in \bb{N}$, then its rational chains $\bf{R}_\bb{Q} \in \Alg_{E_\infty^\rm{u}}(\Fun(\bb{N},\DQ))$ has a canonical lift to 
    \[\bf{R}_\bb{Q} \in \left(\Alg_{E_\infty^\rm{u}}(\Fun(\bb{N},\DQ))_{/\bb{Q}[t]}\right)^{\id_{\bb{Q}[t]}/}.\] 
\end{lemma}

\begin{proof}
    Since $\bf{R}$ is path-connected in each degree, \cref{prop:stable-algebra-spc} \eqref{enum:stable-algebra-spc-i} gives a unique map up to homotopy $\sigma \colon \bf{U} \to \bf{R}$. To lift this to the graded setting, we recall that the unique functor $t \colon \bb{N} \to \ast$ induces a equivalence of categories 
    \[\{\bf{R} \in \Alg_{E_\infty^\rm{u}}(\Fun(\bb{N},\Spc)) \mid \, \bf{R}(n) \, \text{ is connected } \, \forall n \in \N \} \overset{t_!}\lra \{\bf{R} \in \Alg_{E_\infty^\rm{u}}(\Spc) \mid \, \pi_0\,\bf{R} \cong \bb{N}\},\]
    allowing us to lift this to a map $\sigma \colon \bf{U}^\gr \to \bf{R}$ in $\Alg_{E_\infty^\rm{u}}(\Fun(\bb{N},\rm{Spc}))$.
    
    In $\Spc$, the point $*$ is a terminal object, and hence the constant functor $\ul{*} \in \Fun(\bb{N},\Spc)$ is also a terminal object, so it admits a unique $E_\infty^\rm{u}$-structure and there is a unique $E_\infty^\rm{u}$-algebra map $\bf{R} \to \ul{*}$. Applying the strong symmetric monoidal functor given by rational chains we get 
    \[\bf{U}^\gr \xrightarrow{\sigma} \bf{R}_\bb{Q} \xrightarrow{\epsilon} \Q[t] \simeq \iota^* 1_{\Fun(\bb{N}_{\leq},\DQ)}\]
    and under the equivalences $\bf{U}^\gr \simeq \bb{Q}[t] \simeq \iota^* 1_{\Fun(\bb{N}_{\leq},\DQ)}$ this map is the identity.
\end{proof}

These examples have a special feature, particular to the category of spaces:

\begin{lemma}\label{lem:rational-algebras-from-spaces-free} Let $\bf{R}$ be as in \cref{lem:rational-algebras-from-spaces}. Then the following is free as an $E^\rm{u}_\infty$-algebra
\[\bf{R}_\bb{Q}/(\sigma-1) \in \Alg^\aug_{E_\infty^\rm{u}}(\DQ).\]
\end{lemma}

\begin{proof}By \cref{prop:stable-algebra-spc} we can identify $\bf{R}/(\sigma-1)$ with the infinite loop space $\Omega_0^\infty(\bf{R}^\gc/\bb{S})$ as $E_\infty^\rm{u}$-algebras and by the Milnor--Moore theorem \cite[Appendix]{MilnorMoore}, the rational homology of the latter is free as an algebra. Now we argue as \cite[Proposition 4.10]{KRS1}.\end{proof}

\section{The Goncharov Lie coalgebra of an $E_\infty$-algebra} \label{sec:goncharov-lie-coalgebra-general}

In the special case $\scr{C} = \DQ$, we add to the results of the previous section the following tools: 
\begin{enumerate}[\noindent (i)] 
\setcounter{enumi}{3} 
\item a \emph{Koszul truncation} that produces a new $E_\infty^\rm{nu}$-algebra from an old one by truncating its indecomposables to vanish above a line, and
\item for a reduced $E_\infty^\rm{nu}$-algebra whose indecomposables vanishes strictly below the line $* = \slopetwo$, a \emph{Goncharov Lie coalgebra} obtained by Koszul truncation of its quotient by $\sigma$.
\end{enumerate}
We start this discussion with \cref{sec:alg-k-theory-goncharov-maps},  which states the applications to the $E_\infty^\rm{nu}$-algebra $\BGLb(F)$. This is followed by the general constructions and results, which may be skipped on a first reading.

\medskip

In this section we work in the categories $\Fun(\bb{N},\DQ)$ or $\Fun(\bb{N}_{>0},\DQ)$, with the advantages explained in \cref{sec:graded-over-dq}. In $\Fun(\bb{N},\DQ)$, we will implicitly move between unital and nonunital algebras using the adjoint equivalence provided by augmentation and unitalisation. We take $t$-structures on these categories using (nonunital) abstract connectivities, constructed from the canonical one on $\DQ$, cf.~\cite[Section 1.2]{LurieHA}. We will drop the decorations for ``divided power'' ($\dpw$) and ``conilpotency'' ($\nil$) conditions for Lie coalgebras; the former is irrelevant in $\DQ$ and we will guarantee, through connectivity hypotheses, that the latter is irrelevant. 

\subsection{Mapping algebraic $K$-theory to the Goncharov Lie coalgebra} \label{sec:alg-k-theory-goncharov-maps} In this subsection we fix a field $F$ and state the results obtained by specialising the rest of this section to the $E_\infty^\rm{nu}$-algebra $\BGLb(F)$. In particular, we construct maps
\[\rm{e}_{n,d} \colon \gr^\rm{prim}_n K_d(F)_{\bb{Q}} \lra H^\rm{CE}_{n,d}(\scr{G}(F))\]
refining the edge homomorphisms $\rm{edge}_n \colon K_{2n-1}(F)_\bb{Q} \to \scr{G}_n(F)$, as promised in \cite{KRS1}, and establish some of their properties.

\medskip

Recall $\BGLb(F) \in \Alg_{E_\infty^\rm{nu}}(\Fun(\bb{N},\Spc))$, as in \cite[Section 4.1]{KRS1}, is given by the spaces $\BGLb(F)(n) \simeq \BGL_n(F)$ together with coherently commutative and associative multiplication maps $\BGL_n(F) \times \BGL_m(F) \to \BGL_{n+m}(F)$ induced by block sum. Using \cref{lem:rational-algebras-from-spaces} and taking augmentation ideals it yields an object
\[\BGLb(F)_\bb{Q} \in \left(\Alg_{E_\infty^\rm{nu}}(\Fun(\bb{N},\DQ))_{/\bb{Q^\rm{nu}}[t]}\right)^{\id_{\bb{Q}^\rm{nu}[t]}/}.\] 
This has a standard vanishing line of slope 2 as a consequence of the work of Galatius--Kupers--Randal-Williams \cite{GKRW20}, see \cite[Theorem 4.3]{KRS1}. Letting $\BGLb(F)_\bb{Q}/(\sigma)$ denote the cofibre of the map $\sigma \colon \bb{Q}^\rm{nu}[t] \to \BGLb(F)_\bb{Q}$ in the category of $E^\rm{nu}_\infty$-algebras, the maps
\[H^{E_\infty}_{n,2n-1}(\BGLb(F)_\bb{Q}) \lra H^{E_\infty}_{n,2n-1}(\BGLb(F)_\bb{Q}/(\sigma))\]
are isomorphisms and assemble to an isomorphism $\scr{G}(F) \overset{\cong}\lra \scr{G}(\BGLb(F)_\bb{Q})$ of Lie coalgebras, with left side as in \cite[p.~1]{KRS1} and right side as in \cref{defn Goncharov lie coalgebra}. 

Recalling that the group completion of $\BGLb_\bb{Q}(F)$ is by definition the algebraic $K$-theory spectrum of $F$, \cref{lem:rank-spectral-sequences} \eqref{enum:rank-spectral-sequences-ii} gives (implicitly unitalising) a rank spectral sequence
\[{}^\infty E^1_{n,d}=H^{E_\infty}_{n,d}(\BGLb(F)_\bb{Q}/(\sigma)) \Longrightarrow K_d(F)_\bb{Q}.\]
This spectral sequence receives a map from the Rognes rank spectral sequence considered in \cite[Section 9]{KRS1}, which is an isomorphism on the $E^1$-page in all entries except $E^1_{1,0}$ where it projects the permanent cycle $\bb{Q}\{\sigma\}$ to zero. Then \cref{cor:image-of-edge-in-ker-cobracket} says the edge homomorphism $\rm{edge}_n \colon K_{2n-1}(F)_\bb{Q} \to \scr{G}_n(F)$ considered in \cite[Definition 9.2]{KRS1} has image contained in the kernel of the cobracket.

We improve on this by taking into account rank filtrations. Identifying $K_d(F)_\bb{Q}$ with $H_d^{E_\infty}(\BGLb(F)_\bb{Q}/(\sigma-1))$ for $d>0$,  \cref{lem:filtered-map-to-gocharov-complex} provides maps 
\begin{equation} \label{eqn:goncharov-map}
  \rm{e}_{n,d} \colon \gr^\rm{prim}_n\,K_d(F)_\bb{Q} \lra H^\rm{CE}_{n,d}(\scr{G}(F))  
\end{equation}
which fit into a commutative diagram
    \[\begin{tikzcd} \fil_n^\rm{prim} K_{2n-1}(F)_\bb{Q} \rar{\rm{pr}} \dar[swap]{\rm{inc}}& \rm{gr}^\rm{prim}_n K_{2n-1}(F)_\bb{Q} \dar{\rm{e}_{n,2n-1}} \\[-5pt]
    K_{2n-1}(F)_\bb{Q} \rar{\rm{edge}_n} & \ker(\delta|_{\scr{G}_n(F)}) \subset \scr{G}_n(F).\end{tikzcd}\]

By construction, the primitive rank filtration here agrees with the classical one appearing e.g.~in \cite[\S 2.3]{Cathelineau}, \cite{borelyang}, \cite{deJeu}. Thus by \cite[Corollary 5]{deJeu} we have that $\fil^\rm{prim}_n\,K_d(F)_\bb{Q}$ lies in $\bigoplus_{r=1}^n \smash{K_d^{(r)}(F)_\bb{Q}}$, where $\smash{K_d^{(r)}(F)_\bb{Q}}$ is given by the $p^r$-eigenspace of the Adams operation $\psi_p$. There several other rank filtrations and these can behave quite differently; the next lemma is a reason to prefer the primitive rank filtration.

\begin{lemma} \label{lem:milnor K theory}
    The following maps are isomorphisms for all $n \geq 1$:
    \[\rm{e}_{n,n} \colon \rm{gr}^\rm{prim}_n K_n(F)_\bb{Q} \lra H_{n,n}^\rm{CE}(\scr{G}(F)).\]
    In fact, both sides are isomorphic to the Milnor $K$-theory groups $K^M_n(F)_\bb{Q}$.
\end{lemma}

\begin{proof}In \cite{NesterenkoSuslin}, Nesterenko and Suslin proved that projection $K^M_n(F) \to \gr^\rm{prim}_n K_n(F)_\bb{Q}$ is an isomorphism using the map
\[\gr^\rm{prim}_n K_n(F)_\bb{Q} \lra H_{n,n}(\BGLb(F)_\bb{Q}/(\sigma)),\]
(well-defined because of their homological stability result), identifying its target with $K^M_n(F)$, and proving that with respect to this identification the composition is given by multiplication with a nonzero constant. It remains to understand why the map induced by Koszul truncation
    \[K_n^M(F)_\bb{Q} \cong H_{n,n}(\BGLb(F)_\bb{Q}/(\sigma)) \to H_{n,n}(\rho_{\le \slopetwo}^\alg(\BGLb(F)_\bb{Q})/(\sigma)) \cong H^\rm{CE}_{n,n}(\scr{G}(F)) \cong K_n^M(F)_\bb{Q}\]
is an isomorphism. Taking together all $n$, we obtain a map of algebras, because the identification of Nesterenko--Suslin is compatible with the product arising from $E_\infty^\rm{nu}$-algebra structure on $\BGLb(F)_\bb{Q}$ \cite[Lemma 3.23]{NesterenkoSuslin}. Since Milnor $K$-theory is generated in degree $1$, verifying it is an isomorphism reduces to the case $n=1$, where it is an isomorphism by construction.\end{proof}

\subsection{Koszul truncations} Having discussed the applications used elsewhere in this paper, we now start with this section proper, beginning with the notion of Koszul truncation. 

\subsubsection{Abstract connectivity and $t$-structures} \label{subsection connectivity and t structures} Following \cite{GKRW18}, an \emph{abstract connectivity} on a category $\scr{C}$ is a functor $c \colon \scr{C} \to [-\infty,\infty]_{\geq}$, with target the category whose objects are the integers along with $\infty,-\infty$ and where there is a morphism $n \to m$ when $n \geq m$. That is, it is given by $\infty \to \cdots \to 2 \to 1 \to 0 \to -1 \to -2 \to \cdots \to -\infty$. This admits a symmetric monoidal structure given by $\min$, so if $\scr{C}$ is a symmetric monoidal there is a Day convolution tensor product $\ast$ of abstract connectivities. We will interested in $c$ that are lax, oplax, or symmetric monoidal with respect to this:
\begin{itemize}
	\item $c$ is \emph{subadditive} if $c \ast c \geq c$,
	\item $c$ is \emph{superadditive} if $c \ast c \leq c$,
	\item $c$ is \emph{additive} if $c \ast c = c$.
\end{itemize}

\begin{example}If $\scr{C} = \bb{N}$ or $\bb{N}_{>0}$ with monoidal structure given by addition, we have
\[\min\{c(i)+c(j) \mid i+j=k\} \begin{cases} \geq c(k) & \text{if $c$ is subadditive} \\
\leq c(k) & \text{if $c$ is superadditive} \\
= c(k) & \text{if $c$ is additive}.\end{cases}\]
\end{example}

We use these to produce $t$-structures of functor categories. Recall that a $t$-structure on a stable category $\scr{D}$ is a pair of full subcategories $\scr{D}_{\geq 0},\scr{D}_{\leq 0}$ satisfying the axioms of \cite[1.2.1.1]{LurieHA}. Setting $\scr{D}_{\geq n} \coloneq \scr{D}_{\geq 0}[n]$ and $\scr{D}_{\leq n} \coloneq \scr{D}_{\leq 0}[n]$, the inclusions $\scr{D}_{\geq n} \to \scr{D}$ admit right adjoints $\tau_{\geq n}$, the inclusions $\scr{D}_{\leq n} \to \scr{D}$ admit left adjoints $\tau_{\leq n}$ \cite[1.2.1.7]{LurieHA}, and there are natural fibre sequences $\tau_{\geq n} \to \rm{id} \to \tau_{\leq n-1}$. We can extend this to allow $n=-\infty,\infty$, by setting $\scr{D}_{\geq -\infty} = \scr{D} = \scr{D}_{\leq \infty}$ to be the entire category and $\scr{D}_{\geq \infty} = \{0\} = \scr{C}_{\leq -\infty}$ to be the subcategory of zero objects. The fibre sequence $\tau_{\geq 0} \to \rm{id} \to \tau_{\leq -1}$ determines the $t$-structure, and one may verify that postcomposition with it yields a $t$-structure on $\Fun(\scr{C},\scr{D})$: there is a unique $t$-structure on $\Fun(\scr{C},\scr{D})$ where $G \in \Fun(\scr{C},\scr{D})_{\geq 0}$ if and only if $G(x) \in \scr{D}_{\geq 0}$, and $G \in \Fun(\scr{C},\scr{D})_{\leq 0}$ if and only if $G(x) \in \scr{D}_{\leq 0}$. More generally we can incorporate an abstract connectivity:

\begin{lemma} \label{lem: unique t structure functor category with a connectivity functor}
Let $\scr{C}$ be a small category, $\scr{D}$ be a stable category with a $t$-structure, and $c \colon \scr{C} \to [-\infty,\infty]_{\geq}$ be any functor. Then there is a unique $t$-structure on $\Fun(\scr{C},\scr{D})$ so that $G \in \Fun(\scr{C},\scr{D})_{\geq 0}$ if and only if $G(x) \in \scr{D}_{\geq c(x)}$ for all $x \in \scr{C}$, and $G \in \Fun(\scr{C},\scr{D})_{\le 0}$ if and only if $G(x) \in \scr{D}_{\leq c(x)}$ for all $x \in \scr{C}$.
\end{lemma}

These interact well with Day convolution under the following conditions:

\begin{definition}Let $\scr{C}$ be a symmetric monoidal stable category. 
\begin{itemize}
    \item A $t$-structure on $\scr{C}$ is \emph{left-compatible} if $\otimes$ takes $\scr{C}_{\geq 0} \times \scr{C}_{\geq 0}$ into $\scr{C}_{\geq 0}$, and $1_\scr{C} \in \scr{C}_{\geq 0}$.
    \item A $t$-structure on $\scr{C}$ is  \emph{right-compatible} if $\otimes$ takes $\scr{C}_{\leq 0} \times \scr{C}_{\leq 0}$ into $\scr{C}_{\leq 0}$, and $1_\scr{C} \in \scr{C}_{\leq 0}$.
\end{itemize}
There are nonunital variants where we drop the condition on $1_\scr{C}$.
\end{definition}

(Note that a left-compatible $t$-structure is \emph{compatible} in the sense of \cite[2.2.1.3]{LurieHA}.) If a $t$-structure is left-compatible then $\scr{D}_{\geq n}$ for $n \geq 0$ inherits a symmetric monoidal structure so that the inclusion $\scr{D}_{\geq n} \to \scr{D}$ is symmetric monoidal and hence induces a lax monoidality on $\tau_{\geq n}$ by the mate correspondence. Moreover, on $\scr{D}_{\geq 0}$ we have that $\tau_{\leq n}$ is lax monoidal \cite[2.2.1.8, 2.2.1.9, 2.2.1.10]{LurieHA}. If a $t$-structure is right-compatible then we have the dual properties: $\tau_{\leq n}$ is oplax monoidal on $\scr{D}$ and $\tau_{\geq n}$ is oplax monoidal on $\scr{D}_{\leq 0}$. If we only have the nonunital variants of left- or right-compatibility, the symmetric monoidal structures and (op)lax monoidalities will be only nonunital. The formula for Day convolution implies:

\begin{lemma}Suppose $\scr{C}$ is a symmetric monoidal category with abstract connectivity $c$ with $c(0) \leq 0$ and $\scr{D}$ is a symmetric monoidal stable category with a $t$-structure. If the $t$-structure is left-compatible and the abstract connectivity $c$ is subadditive, then the induced $t$-structure on $\Fun(\scr{C},\scr{D})$ is left-compatible.\end{lemma}

It is less common that right-compatibility is preserved by the Day convolution construction but here is one example, proven using the analogous argument:

\begin{lemma} \label{lem: t structure right compatible}
Suppose $\scr{G}$ is a symmetric monoidal groupoid whose automorphism groups are finite, $c \colon \scr{G} \to [-\infty,\infty]_{\geq}$ any functor, and $\scr{D}$ is a symmetric monoidal stable category with a $t$-structure. If the $t$-structure is right-compatible, $\scr{D}_{\leq 0}$ is closed under coproducts and finite group orbits, and the abstract connectivity $c$ is superadditive, then the induced $t$-structure on $\Fun(\scr{G},\scr{D})$ is right-compatible.\end{lemma}

\subsubsection{Koszul truncation} A feature of working in $\DQ$ is that its standard $t$-structure is both left- and right-compatible with $\otimes$, and hence $\scr{C}_{\le 0}$ is closed under coproducts, finite group orbits and fixed points. More generally, we fix a (nonunital) abstract connectivity $c \colon \bb{N}_{>0} \to [-\infty,\infty)_\ge$. Then \cref{lem: unique t structure functor category with a connectivity functor} provides a $t$-structure on $\Fun(\bb{N}_{>0},\DQ)$ satisfying
\begin{align*}X \in \Fun(\bb{N}_{>0},\DQ)_{\ge c} \quad &\Longleftrightarrow \quad X(n) \in (\DQ)_{\ge c(n)} \text{for all $n \in \bb{N}_{>0}$},\\
X \in \Fun(\bb{N}_{>0},\DQ)_{\le c} \quad &\Longleftrightarrow \quad X(n) \in (\DQ)_{\le c(n)} \text{ for all $n \in \bb{N}_{>0}$.}\end{align*}
We will denote its heart by $\Fun(\bb{N}_{>0},\DQ)_{=c}^\heart$, which is equivalent to the category of bigraded vector spaces concentrated in bidegrees $(n,c(n))$ for $n \in \bb{N}_{>0}$. 

Assume now that $c$ is superadditive, i.e.~$\min\{c(i)+c(j) \mid i+j=k\} \leq c(k)$. By \cref{lem: t structure right compatible}, the resulting $t$-structure is right-compatible: the tensor product takes $\Fun(\bb{N}_{>0},\DQ)_{\le c} \times \Fun(\bb{N}_{>0},\DQ)_{\le c}$ into $\Fun(\bb{N}_{>0},\DQ)_{\le c}$, essentially due to the vanishing of $\rm{Tor}$-terms in the K\"unneth theorem. Thus the nonunital symmetric monoidal structure on $\Fun(\bb{N}_{>0},\DQ)$ restricts to a nonunital symmetric monoidal structure on subcategory, making the inclusion
\[i_{\le c} \colon \Fun(\bb{N}_{>0},\DQ)_{\le c} \lra \Fun(\bb{N}_{>0},\DQ)\]
into a strong symmetric monoidal functor, and via the mate correspondence its left adjoint truncation functor $\tau_{\le c}$ acquires an oplax symmetric monoidality.

Proceeding to Koszul duality, we restrict our attention to reduced objects:

\begin{definition}An object $X \in \Fun(\bb{N},\DQ)$ is \emph{reduced} if $X(0) \simeq 0$. An algebra over an operad, or coalgebra over a cooperad, in $\Fun(\bb{N},\DQ)$ is reduced if its underlying object is.\end{definition}

Restriction induces an equivalence $\Fun^\red(\bb{N},\DQ) \to \Fun(\bb{N}_{>0},\DQ)$, and using this we can identify reduced (co)algebras over nonunitary (co)operads in $\Fun(\bb{N},\DQ)$ with (co)algebras in $\Fun(\bb{N}_{>0},\DQ)$. In particular, this induces an identification of $\Alg^\rm{red}_{E_\infty^\rm{nu}}(\Fun(\bb{N},\DQ))$ with $\Alg_{E_\infty^\rm{nu}}(\Fun(\bb{N}_{>0},\DQ))$ and similarly for reduced Lie coalgebras. The upshot of this is that it is sufficient to consider only nonunital abstract connectivities $c \colon \bb{N}_{>0} \to [-\infty,\infty)_{\geq}$.

We will from now on freely use notions from \cite[Appendix B]{KRS1}. For example, \cite[Section B.1.6]{KRS1} provides a left adjoint $\smash{\tau_{\le c}^\rm{alg}}$ to the functor $\smash{i_{\le c}^\rm{alg}}$ induced by $i_{\le c}$ on categories of algebras. The composition $\rho_{\le c} \coloneqq i_{\le c} \circ \tau_{\le c} \colon \Fun(\bb{N}_{>0},\DQ) \to \Fun(\bb{N}_{>0},\DQ)$ is also oplax symmetric monoidal and admits a natural transformation $\id \to \rho_{\le c}$ of oplax symmetric monoidal functors by the mate correspondence. We similarly have functors
\[\rho^\rm{alg}_{\le c} \coloneq i^\rm{alg}_{\le c} \circ \tau^\alg_{\le c} \qquad \text{and} \qquad \rho^\rm{coalg}_{\le c} \simeq i^\rm{coalg}_{\le c} \circ \tau^\rm{coalg}_{\le c}.\]

\begin{definition}
    The \textit{Koszul truncation} is the functor
    \[\rho_{\le c}^{\rm{alg}} \coloneq i_{\le c}^\rm{alg} \circ \tau_{\le c}^\rm{alg} \colon \Alg^\red_{E_\infty^\rm{nu}}(\Fun(\bb{N},\DQ)) \lra \Alg^\red_{E_\infty^\rm{nu}}(\Fun(\bb{N},\DQ)),\]
    which comes with a unit natural transformation $\id \to \rho_{\le c}^\rm{alg}$.
\end{definition}

To consider the effect on indecomposables of the algebras of interest we make the further assumption that $c+1$ is superadditive on $\bb{N}_{>0}$. This amounts to $\min\{c(i)+c(j) \mid i+j = k\} \leq c(k)-1$ when $i,j>0$, so in particular $c$ is superadditive. Under this assumption is the comonad $\rm{Sym}_{s\,\rm{coLie}}$ preserves $\Fun(\bb{N},\DQ)_{\leq c}$:

\begin{lemma}Suppose $c+1$ is superadditive when restricted to $\bb{N}_{>0}$. If a symmetric sequence $X \in \rm{SSeq}(\DQ)$ has $X(0) \simeq 0$ and $X(r) \in \scr{D}_\bb{Q}^\heart$, then $sX \circ (-)$ preserves $\Fun(\bb{N},\DQ)_{\leq c}$.
\end{lemma}

\begin{proof}The suspension $sX(r)$ is concentrated in degree $r-1$, so if $Y \in \Fun(\bb{N},\DQ)_{\leq c}$ we have that $(sX(r) \otimes_{S_r} Y^{\otimes r})(n)$ has homology concentrated in degrees $\leq r-1+c^{*r}(n) \leq r-1+c(n)-r+1=c(n)$.\end{proof}

Fix $c \colon \bb{N}_{>0} \to [-\infty,\infty)_{\geq}$ so that $c+1$ is superadditive. Since we are working with reduced objects, Koszul duality is an equivalence by \cite[Theorem B.17]{KRS1}. Moreover, by \cite[Section B.1.5]{KRS1} the relationship between indecomposables and left adjoints on the level of algebras simplifies to a commutative diagram 
\[\begin{tikzcd} \Alg^\rm{red}_{E_\infty^\rm{nu}}(\Fun(\bb{N},\DQ)) \rar{\indec_{E_\infty^\rm{nu}}}[swap]{\simeq} \dar[swap]{\tau^\rm{alg}_{\le c}} &[20pt] \coAlg^\rm{red}_{s\,\rm{coLie}}(\Fun(\bb{N},\DQ)) \dar{\tau_{\le c}^\rm{coalg}} \\[-5pt]
\Alg^\rm{red}_{E_\infty^\rm{nu}}(\Fun(\bb{N},\DQ)_{\le c}) \rar{\indec_{E_\infty^\rm{nu}}}[swap]{\simeq}  & \coAlg^\rm{red}_{s\,\rm{coLie}}(\Fun(\bb{N},\DQ)_{\le c})
\end{tikzcd}\]
with horizontal maps equivalences. We recall that $\tau_{\le c}^\rm{coalg}$ agrees with the map induced on shifted Lie coalgebras by the oplax symmetric monoidal functor $\tau_{\leq c}$. We prefer to think of objects of $\Fun(\bb{N},\DQ)_{\leq c}$ as objects of $\Fun(\bb{N},\DQ)$ through $i_{\leq c}$: this inclusion $i_{\leq c}$ is symmetric monoidal and preserves coproducts and orbits for finite group actions, so induces a functor on Lie coalgebras that fits in a commutative square
\[\begin{tikzcd} \Alg^\rm{red}_{E_\infty^\rm{nu}}(\Fun(\bb{N},\DQ)_{\le c}) \rar{\indec_{E_\infty^\rm{nu}}}[swap]{\simeq} \dar[swap]{i_{\le c}^\rm{alg}} &[20pt] \coAlg^\rm{red}_{s\,\rm{coLie}}(\Fun(\bb{N},\DQ)_{\le c}) \dar{i_{\le c}^\rm{coalg}} \\[-5pt]
\Alg^\rm{red}_{E_\infty^\rm{nu}}(\Fun(\bb{N},\DQ)) \rar{\indec_{E_\infty^\rm{nu}}}[swap]{\simeq} & \coAlg^\rm{red}_{s\,\rm{coLie}}(\Fun(\bb{N},\DQ)).\end{tikzcd}\]
The above discussion then yields the following, incorporating a shift so that its statement involves ordinary Lie coalgebras rather than shifted ones:

\begin{lemma} \label{lem:truncation-via-indec}
Suppose $c \colon \bb{N}_{>0} \to [-\infty,\infty)_{\geq}$ is so that $c+1$ is superadditive. Then there is an equivalence of natural transformations
 \[\indec_{E_\infty}(-)[1] \circ (\id \to \rho^\rm{alg}_{\le c}) \simeq (\id \to \rho_{\le c+1}^\rm{coalg}) \circ (\indec_{E_\infty}(-)[1]) .\]
\end{lemma}

\begin{notation} \label{notation rho}
    From now on we will abbreviate $\rho_{\le c+1}^\rm{coalg}$ to $\rho_{\le c+1}$ when there is no risk of confusion, as it agrees with this on underlying objects.
\end{notation}

\subsection{Koszul truncations} We now work with the category of algebras with stabilisation map as in \cref{sec:stable-algebras-graded}, though we apply the adjoint equivalence given by taking augmentation ideal and unitalisation to move into the nonunital setting: that is, we work in $\Alg_{E_\infty^\rm{nu}}(\Fun(\bb{N},\DQ))^{\bb{Q}^\rm{nu}[t]/}$ where we write
\[\bb{Q}^\rm{nu}[t] \coloneq \free_{E^\rm{nu}_\infty}(1_! \Q) \in \Alg_{E_\infty^\rm{nu}}(\Fun(\bb{N},\DQ)).\]
We will be specifically interested in the following vanishing lines of slope 2:

\begin{notation}In this subsection, we let $2{\bullet}{\shortminus}1$ denote the nonunital abstract connectivity given by $n \mapsto 2n-1$ for $n \geq 1$. Note that $c+1$ is super-additive.\end{notation}

\begin{definition}\label{def:std-vanishing-line} We say that $(\bf{Q}^\rm{nu}[t] \xrightarrow{\sigma} \bf{R}) \in \Alg_{E^\rm{nu}_\infty}(\Fun(\bb{N},\DQ))^{\bf{Q}^\rm{nu}[t]/}$ has a \emph{standard vanishing line of slope 2} if $\rm{cot}_{{E^\rm{nu}_\infty}}(\bf{R}/(\sigma)) \in \Fun(\bb{N},\DQ)_{\ge \slopetwo}$.
\end{definition}

In other words, the $E_\infty$-homology groups of $\bf{R}/(\sigma)$ satisfy $H^{E_\infty}_{n,d}(\bf{R}/(\sigma)) = 0$ for $d < 2n-1$.

\begin{remark}We take $\bf{R}/(\sigma)$ instead of $\bf{R}$ because we are interested in the \emph{difference} between $\bb{Q}^\rm{nu}[t]$ and $\bf{R}$. A reader familiar with \cite{GKRW20} may recognise that this removes the generator in bidegree $(n,d) = (1,0)$ denoted $\sigma$.\end{remark}

We refer to the Koszul truncation $\rho_{\le \slopetwo}^\alg(\bf{R})$ simply as the \emph{Koszul truncation}. Its indecomposables are classical in the following sense: recall that a Lie coalgebra in $\Fun(\bb{N},\DQ)$ is \emph{coformal} if it is equivalent to to its homology with induced cobracket. 

\begin{lemma} \label{lem:coformality-critical}
If $(\bf{Q}^\rm{nu}[t] \xrightarrow{\sigma}\bf{R})$ has a standard vanishing line of slope 2 then the shifted Lie coalgebra $\indec_{E^\rm{nu}_\infty}(\rho_{\le \slopetwo}^\rm{alg}(\bf{R}/(\sigma)))$ is coformal.
\end{lemma}

\begin{proof} Apply \cref{lem:truncation-via-indec} to $\bf{R}/(\sigma)$ to get an equivalence
\[\indec_{E^\rm{nu}_\infty}(\rho_{\le \slopetwo}^\rm{alg}(\bf{R}/(\sigma)))[1] \simeq \rho_{\le 2\bullet} (\indec_{E^\rm{nu}_\infty}(\bf{R}/(\sigma))[1]).\] 
The slope 2 vanishing line yields that $\rm{cot}_{E^\rm{nu}_\infty}(\bf{R}/(\sigma))$ lies in $\Fun(\bb{N},\DQ)_{\ge \slopetwo}$. Thus the truncation of its shift, the underlying object of the right term, lies not only in $\Fun(\bb{N},\DQ)_{\le 2\bullet}$ but also $\Fun(\bb{N},\DQ)_{\ge 2\bullet}$ and hence in the heart. Coformality then follows from the fact that the heart is equivalent, up to shifts, to graded vector spaces via the homology functor. 
\end{proof}

By the previous lemma, $\rm{indec}_{E_\infty}(\rho_{\le 2 \bullet -1}^\rm{alg} (\bf{R}/(\sigma)))[1]$ is equivalent to its homology and hence to the Lie coalgebra $\scr{G}(\bf{R})$ of the following definition:

\begin{definition} \label{defn Goncharov lie coalgebra}For $(\bf{Q}^\rm{nu}[t] \xrightarrow{\sigma} \bf{R})$ with a standard vanishing line of slope 2, its \emph{Goncharov Lie coalgebra} is 
\vspace{-.2cm} \[\qquad \scr{G}(\bf{R}) \coloneq \bigoplus_{n \geq 1} \scr{G}_n(\bf{R}) \qquad \text{with} \qquad \scr{G}_n(\bf{R}) \coloneq H^{E_\infty}_{n,2n-1} (\bf{R}/(\sigma)),\]
considered as a Lie coalgebra in bigraded vector spaces, concentrated in bidegrees $(n,2n)$ with Lie cobracket of degree $0$.
\end{definition}

Thus, $\rm{indec}_{E_\infty}(\rho_{\le 2 \bullet -1}^\rm{alg} (\bf{R}/(\sigma))) \simeq \scr{G}(\bf{R})[-1] \in \coAlg_{\rm{s\,coLie}}(\DQ^\N)$ for $(\bf{Q}^\rm{nu}[t] \xrightarrow{\sigma}\bf{R})$ as in \cref{lem:coformality-critical}. 

\begin{remark}
    One might be tempted view $\scr{G}(\bf{R})$ as a Lie coalgebra with a grading, by letting $\scr{G}_n(\bf{R})$ be the degree $n$ part. From this perspective, however, there are \emph{no} extra Koszul signs since the original bigraded version is concentrated in even homological degrees. 
\end{remark}

By Koszul duality, the Koszul truncation of $\bf{R}/(\sigma)$ may be recovered from $\scr{G}(\bf{R})$ by applying the inverse $\prim_{\coLie}$ of $\indec_{E_\infty^\rm{nu}}$. One can also recover the Koszul truncation $\bf{R}$ itself, as well as the stable algebra, as long as the map $\sigma$ is split:

\begin{theorem}\label{thm:r-from-critical}
        Suppose $(\bf{Q}^\rm{nu}[t] \xrightarrow{\sigma} \bf{R} \xrightarrow{\epsilon} \bf{Q}^\rm{nu}[t]) \in (\Alg^\rm{red}_{E_\infty^\rm{nu}}(\Fun(\bb{N},\DQ))_{/\bf{Q}^\rm{nu}[t]})^{\id_{\bf{Q}^\rm{nu}[t]}/}$ has a standard vanishing line of slope 2. Then there is an equivalence
        \[\rho_{\leq \slopetwo}^\alg(\bf{R}) \simeq \bb{Q}^\rm{nu}[t] \sqcup^{E_\infty^\rm{nu}} \prim_{s\,\coLie}(\scr{G}(\bf{R})[-1]).\] 
\end{theorem}

This is a direct consequence of the first part of the next lemma by Koszul duality. The second part tells us that the Goncharov Lie coalgebra is a classical one.

\begin{lemma}  \label{lem:structure-goncharov-truncation}Suppose $(\bf{Q}^\rm{nu}[t] \xrightarrow{\sigma} \bf{R} \xrightarrow{\epsilon} \bf{Q}^\rm{nu}[t]) \in \left(\Alg_{E_\infty^\rm{nu}}(\Fun(\bb{N},\DQ))_{/\bf{Q}^\rm{nu}[t]}\right)^{\id_{\bf{Q}^\rm{nu}[t]}/}$ has a standard vanishing line of slope 2. Then
\begin{enumerate}[(i)]
    \item \label{enum:structure-goncharov-truncation-i} $\rho_{\le 2\bullet} (\indec_{E^\rm{nu}_\infty}(\bf{R})[1]) \simeq \rm{cotriv}_{\rm{coLie}}(1_! \bb{Q}\{\sigma\}[1]) \sqcup^{\rm{coLie}} \scr{G}(\bf{R})$,
    \item \label{enum:structure-goncharov-truncation-ii} $\rho_{\le 2\bullet} (\indec_{E^\rm{nu}_\infty}(\bf{R})[1])$ is coformal as a Lie coalgebra.
\end{enumerate}
\end{lemma}

\begin{proof} It suffices to prove part \eqref{enum:structure-goncharov-truncation-i} as part \eqref{enum:structure-goncharov-truncation-ii} is an immediate consequence. Applying the left adjoint $\cot_{E^\rm{nu}_\infty}$ to the cofibre sequence defining $\bf{R}/(\sigma)$ yields a cofibre sequence
\[ \cot_{E^\rm{nu}_\infty}(\bf{Q}^\rm{nu}[t]) \lra \cot_{E^\rm{nu}_\infty}(\bf{R}) \lra \cot_{E^\rm{nu}_\infty}(\bf{R}/(\sigma))\]
 in $\Fun(\N,\DQ)$. 
 The map $\cot_{E_\infty^\rm{nu}}(\epsilon)$ provides a splitting, and since the functor $\rho_{\le 2 \bullet}((-)[1])$ preserves finite coproducts we get an equivalence
 \begin{equation} \label{eq: split}
  \rho_{\le 2 \bullet}( \cot_{E^\rm{nu}_\infty}(\bf{R})[1]) \xrightarrow{\simeq} \rho_{\le 2 \bullet}(\cot_{E_\infty^\rm{nu}}(\bf{Q}^\rm{nu}[t])[1]) \sqcup \rho_{\le 2 \bullet}( \cot_{E^\rm{nu}_\infty}(\bf{R}/(\sigma))[1]).  
 \end{equation}
This in turn gives a splitting in homology, which by inspection is by bidegrees: the first summand only has homology $\Q$ in bidegree $(1,1)$, and the homology of the second summand is concentrated on the line $(n,2n)$.

\medskip

\noindent \emph{Claim} There is a Lie coalgebra map 
\[j \colon \rho_{\leq 2\bullet}(\indec_{E_\infty^\rm{nu}}(\bf{R}/(\sigma)[1]) \lra \rho_{\leq 2\bullet}(\indec_{E_\infty^\rm{nu}}(\bf{R})[1])\]
such that the corresponding map $j_0 \colon \rho_{\leq 2\bullet}(\cot_{E_\infty^\rm{nu}}(\bf{R}/(\sigma))[1]) \to \rho_{\leq 2\bullet}(\cot_{E_\infty^\rm{nu}}(\bf{R})[1])$ on underlying objects is given in homology by including all classes except the $\Q$ in bidegree $(1,1)$. 

\begin{proof}[Proof of Claim.]
The inclusion $i_{=2\bullet } \colon \Fun(\bb{N},\DQ)_{=2\bullet}^\heart \to {\Fun(\bb{N},\DQ)}_{\le 2\bullet}$ has a right adjoint given by restricting $\tau_{\ge 2\bullet}$ to the subcategory ${\Fun(\bb{N},\DQ)}_{\le 2\bullet}$. Since $2\bullet$ is additive, one verifies that symmetric monoidal structure on $\Fun(\bb{N},\DQ)$ restricts to the heart, making both the inclusion $i_{=2\bullet}$ and its right adjoint $\tau_{\ge 2\bullet}$ strong symmetric monoidal. Consequently, $\tau_{\ge 2\bullet} \tau_{\le 2\bullet} (\indec_{E_\infty}(\bf{R})[1])$ has a natural Lie coalgebra structure and the counit provides a map of Lie coalgebras  
\[i_{=2\bullet}\tau_{\ge 2\bullet} \tau_{\le 2\bullet} (\indec_{E^\rm{nu}_\infty}(\bf{R})[1]) \lra \tau_{\le 2\bullet} (\indec_{E^\rm{nu}_\infty}(\bf{R})[1]).\]
Postcomposing with the Lie coalgebra map $\tau_{\le 2\bullet} (\indec_{E^\rm{nu}_\infty}(\bf{R})[1]) \to \tau_{\le 2\bullet} (\indec_{E^\rm{nu}_\infty}(\bf{R}/(\sigma))[1])$ and applying the functor
$i_{\le 2 \bullet} \colon \Fun(\N,\DQ)_{\le 2 \bullet} \to \Fun(\N,\DQ)$ we get a map of Lie coalgebras 
\[i_{\le 2 \bullet} i_{=2\bullet}\tau_{\ge 2\bullet} \tau_{\le 2\bullet} (\indec_{E^\rm{nu}_\infty}(\bf{R})[1]) \xrightarrow{\simeq} \rho_{\le 2\bullet} (\indec_{E^\rm{nu}_\infty}(\bf{R}/(\sigma))[1])\]
that is an equivalence. To see this, we observe it is an equivalence on underlying objects because both lie in the heart and the map is a homology isomorphism (using the splitting \eqref{eq: split} in homology and that $\tau_{\ge 2 \bullet}$ kills the first summand).
We then define $j$ as the zig-zag
\[\rho_{\le 2\bullet} (\indec_{E^\rm{nu}_\infty}(\bf{R}/(\sigma))[1]) \xleftarrow{\simeq} i_{\le 2 \bullet} i_{=2\bullet}\tau_{\ge 2\bullet} \tau_{\le 2\bullet} (\indec_{E^\rm{nu}_\infty}(\bf{R})[1]) \to \rho_{\le 2\bullet} (\indec_{E^\rm{nu}_\infty}(\bf{R})[1]), \]
so that by definition post-composing it with $\rho_{\le 2\bullet} (\indec_{E^\rm{nu}_\infty}(\bf{R})[1]) \to \rho_{\le 2\bullet} (\indec_{E^\rm{nu}_\infty}(\bf{R}/(\sigma))[1])$ yields a self-map on 
$\rho_{\le 2\bullet} (\indec_{E^\rm{nu}_\infty}(\bf{R}/(\sigma))[1])$ acting as the identity on homology, which implies the result by the splitting of \eqref{eq: split}. 
\end{proof}

To finish the proof it suffices to show that the map of Lie coalgebras 
 \[\rm{cotriv}_\rm{coLie} (1_! \bb{Q} [1]) \sqcup \rho_{\le 2\bullet}( \rm{indec}_{E_\infty^\rm{nu}}(\bf{R}/(\sigma))[1])\ \xrightarrow{ \rho_{\le 2 \bullet} (\indec_{E_\infty^\rm{nu}}(\sigma)[1])\sqcup j}\rho^\rm{coalg}_{\le 2\bullet}(\rm{indec}_{E_\infty^\rm{nu}}(\bf{R})[1]) \] 
where $\sqcup$ indicates the coproduct of Lie coalgebras, is an equivalence. 

First observe that $\rm{indec}_{E_\infty^\rm{nu}}(\bf{Q}^\rm{nu}[t])[1] \simeq \rm{cotriv}_\rm{coLie} (1_! \bb{Q} [1])$ by Koszul duality and hence that $\rho_{\le 2 \bullet}(\rm{indec}_{E_\infty^\rm{nu}}(\bf{Q}^\rm{nu}[t])[1]) \simeq \rm{cotriv}_\rm{coLie} (1_! \bb{Q} [1])$. Thus, it suffices to check that the map is an equivalence on underlying objects, that is,
\[1_! \bb{Q} [1] \oplus \rho_{\le 2\bullet}( \rm{cot}_{E_\cot^\rm{nu}}(\bf{R}/(\sigma))[1])\ \xrightarrow{ \rho_{\le 2 \bullet} (\cot_{E_\infty^\rm{nu}}(\sigma)[1])\sqcup j_0}\rho_{\le 2\bullet}(\rm{cot}_{E_\infty^\rm{nu}}(\bf{R})[1])\] 
 is an equivalence. However, this follows from the homological description of $j_0$ and the fact that $\rho_{\le 2 \bullet} (\cot_{E_\infty^\rm{nu}}(\sigma)[1])$ is given in homology by inclusion of the $\Q$ summand in bidegree $(1,1)$ by the discussion at the beginning of the proof. 
\end{proof}

\begin{corollary}Let $(\bf{Q}^\rm{nu}[t] \xrightarrow{\sigma} \bf{R} \xrightarrow{\epsilon} \bf{Q}^\rm{nu}[t])$ be as in \cref{thm:r-from-critical}. Then we have:
    \begin{enumerate}[(i)]
        \item $\rho_{\leq \slopetwo}^\alg(\bf{R})/(\sigma) \simeq \rho_{\leq \slopetwo}^\alg(\bf{R}/(\sigma)) \simeq \prim_{s\,\coLie}(\scr{G}(\bf{R})[-1])$ in $\Alg^\rm{red}_{E_\infty^\rm{nu}}(\Fun(\bb{N},\DQ))$,
        \item $\rho_{\leq \slopetwo}^\alg(\bf{R})/(\sigma-1) \simeq \rho_{\leq \slopetwo}^\alg(\bf{R}/(\sigma-1)) \simeq \prim_{s\,\coLie}(\colim \scr{G}(\bf{R})[-1])$ in $\Alg_{E_\infty^\rm{nu}}(\DQ)$.
        \end{enumerate}
\end{corollary} 

\subsection{Edge homomorphisms} Recall that for 
\[(\bf{Q}^\rm{nu}[t] \xrightarrow{\sigma} \bf{R} \xrightarrow{\epsilon} \bf{Q}^\rm{nu}[t]) \in (\Alg^\rm{red}_{E_\infty^\rm{u}}(\Fun(\bb{N},\DQ))_{/\bf{Q}^\rm{nu}[t]})^{\id_{\bf{Q}^\rm{nu}[t]}/},\]
Implicitly unitalising, \cref{lem:rank-spectral-sequences} provides a pair of spectral sequences
\begin{align*}E^1_{n,d}=H_{n,d}(\bf{R}/(\sigma)) &\Longrightarrow H_d(\bf{R}/(\sigma-1)) \\
 {}^\infty E^1_{n,d}=H^{E_\infty}_{n,d}(\bf{R}/(\sigma)) &\Longrightarrow H_d^{E_\infty}(\bf{R}/(\sigma-1))\end{align*}
together with a morphism of spectral sequences $E^r_{*,*} \to {}^\infty E^r_{*,*} $ which is multiplicative if we give the domain the induced product on homology and the target the trivial product.

\subsubsection{Edge homomorphisms for algebras with a slope 2 vanishing line}
 
If $\bf{R}$ has a standard vanishing line of slope 2, as we suppose from now on, ${}^\infty E^1_{n,d}$ will be concentrated in bidegrees $(n,d)$ satisfying $d \geq 2n-1$. Thus in odd degrees its edge homomorphisms take the form
\[\rm{edge}_n \colon H_{2n-1}^{E_\infty}(\bf{R}/(\sigma-1)) \lra \scr{G}_n(\bf{R}) \cong H_{n,2n-1}^{E_\infty}(\bf{R}/(\sigma)).\] 
The goal of this subsection is to understand these better and give a more refined construction.

\begin{example}\label{exam:edge-ce-identification}  Let us start with computational consequences of \cref{thm:r-from-critical}: concrete descriptions of $\rho_{\leq \slopetwo}^\alg(\bf{R})$, its stable algebra, and the cofibre of stabilisation. Firstly, by \cite[Section C.4]{KRS1} $\prim_{s\,\rm{coLie}}$ may be computed using the Chevalley--Eilenberg complex, yielding a commutative square
\[\begin{tikzcd}H_{n,d}(\rho_{\leq \slopetwo}^\alg(\bf{R})) \rar{\cong} \dar &  \bigoplus_{i = 1}^n H_{i,d}^\rm{CE}(\scr{G}(\bf{R})) \dar{\pr} \\[-5pt]
H_{n,d}(\rho_{\leq \slopetwo}^\alg(\bf{R})/(\sigma)) \rar{\cong} & H^\rm{CE}_{n,d}(\scr{G}(\bf{R}))\end{tikzcd}\]
with left vertical map induced by the map $\bf{R} \to \bf{R}/(\sigma)$ of $E_\infty^\rm{nu}$-algebras. Secondly, there is a commutative square
\[\begin{tikzcd}H_{n,d}(\rho_{\leq \slopetwo}^\alg(\bf{R})) \rar{\cong} \dar &  \bigoplus_{i=1}^n H_{i,d}^\rm{CE}(\scr{G}(\bf{R})) \dar{\inc} \\[-5pt]
H_{d}(\rho_{\leq \slopetwo}^\alg(\bf{R})/(\sigma-1)) \rar{\cong} & \bigoplus_{i=1}^\infty H^\rm{CE}_{i,d}(\scr{G}(\bf{R}))\end{tikzcd}\]
with left vertical map induced by the stabilisation map $\rho_{\leq \slopetwo}^\alg(\bf{R})(n) \to \rho_{\leq \slopetwo}^\alg(\bf{R})/(\sigma-1)$.
In all these cases, the canonical map to $E_\infty$-homology is given by the projection of the Chevalley--Eilenberg complex onto its generators.\end{example}

Since the natural map $\bb{Q}^\rm{nu}[t] \to \rho_{\le \slopetwo}^\alg(\bb{Q}^\rm{nu}[t])$ is an equivalence, we get maps
\[(\bf{Q}^\rm{nu}[t] \xrightarrow{\sigma} \bf{R} \xrightarrow{\epsilon} \bf{Q}^\rm{nu}[t]) \lra (\bf{Q}^\rm{nu}[t] \xrightarrow{\sigma} \rho_{\le \slopetwo}^\alg(\bf{R}) \xrightarrow{\epsilon} \bf{Q}^\rm{nu}[t])\]
in the category $(\Alg^\rm{red}_{E_\infty^\rm{u}}(\Fun(\bb{N},\DQ))_{/\bf{Q}^\rm{nu}[t]})^{/\id_{\bf{Q}^\rm{nu}[t]}}$. Thus, by the naturality of the rank spectral sequence construction we get a map of spectral sequences 
\[\begin{tikzcd} {}^\infty E^1_{n,d}=H^{E_\infty}_{n,d}(\bf{R}/(\sigma)) \Longrightarrow H_d^{E_\infty}(\bf{R}/(\sigma-1)) \dar \\[-10pt]
{}^\infty \ol{E}^1_{n,d}=H^{E_\infty}_{n,d}(\rho_{\le \slopetwo}^\alg(\bf{R})/(\sigma)) \Longrightarrow H_d^{E_\infty}(\rho_{\le \slopetwo}^\alg(\bf{R})/(\sigma-1)).\end{tikzcd}\]
The first page ${}^\infty \ol{E}^1_{n,d}$ of the second spectral sequence is given by $\scr{G}(\bf{R})[-1]$, concentrated in bidegrees $(n,2n-1)$, and collapses for degree reasons---by \cref{not:rank-ss-indexing}, $d^r$ has bidegree $(-r,r-1)$. The map on the first pages of the spectral sequences is an isomorphism in bidegrees $(n,2n-1)$ and thus the following commutes
\[\begin{tikzcd} H^{E_\infty}_{n,2n-1}(\bf{R}/(\sigma)) \dar[swap]{\cong} &[10pt] H^{E_\infty}_{2n-1}(\bf{R}/(\sigma-1)) \dar \lar[swap]{\rm{edge}_n} \\[-5pt]
    H^{E_\infty}_{n,2n-1}(\rho_{\le \slopetwo}^\alg(\bf{R})/(\sigma)) & H^{E_\infty}_{2n-1}(\rho_{\le \slopetwo}^\alg(\bf{R})/(\sigma-1)) \lar{\cong} \end{tikzcd}\]
where the vertical maps are induced by naturality. We can compute the edge homomorphism in some cases:

\begin{lemma} \label{lem:edge-homomorphism-in-rank-n} 
    Let $(\bf{Q}^\rm{nu}[t] \xrightarrow{\sigma} \bf{R} \xrightarrow{\epsilon} \bf{Q}^\rm{nu}[t])$ be as above. For $n \geq 1$ the following diagram commutes
    \[\begin{tikzcd}  H_{n,2n-1}(\bf{R}) \rar{\rm{stab}} \dar &[5pt] H_{2n-1}(\bf{R}/(\sigma-1)) \dar{\rm{can}} \\[-5pt]
    H_{n,2n-1}^{E_\infty}(\bf{R}/(\sigma)) & H_{2n-1}^{E_\infty}(\bf{R}/(\sigma-1)) \lar[swap]{\rm{edge}_n}\end{tikzcd}\]
    where the vertical maps are induced by $\bf{R} \to \bf{R}/(\sigma)$ followed by the canonical map from homology to $E_\infty$-homology, and top horizontal map is the stabilisation map induced by $\bf{R} \to \bf{R}/(\sigma-1)$. 
\end{lemma}

\begin{proof}Naturality in $\bf{R}$ yields the cube
    \[\begin{tikzcd}[column sep=small, row sep=small] H_{n,2n-1}(\bf{R}) \arrow{rr} \arrow{dd} \arrow{rd}&[-15pt] &[-15pt] H_{2n-1}(\bf{R}/(\sigma-1)) \arrow{dd} \arrow{rd} & \\[-3pt]
    & H_{n,2n-1}(\rho_{\le \slopetwo}^\alg(\bf{R})) \arrow{rr} \arrow{dd} & & H_{2n-1}(\rho_{\le \slopetwo}^\alg(\bf{R})/(\sigma-1)) \arrow{dd} \\[-3pt]
    H^{E_\infty}_{n,2n-1}(\bf{R}/(\sigma))  \arrow{rd}{\cong} & & H^{E_\infty}_{2n-1}(\bf{R}/(\sigma-1)) \arrow{rd} \arrow{ll} & \\[-3pt]
    & H^{E_\infty}_{n,2n-1}(\rho_{\le \slopetwo}^\alg(\bf{R})/(\sigma)) & & H^{E_\infty}_{2n-1}(\rho_{\le \slopetwo}^\alg(\bf{R})/(\sigma-1)) \arrow{ll}\end{tikzcd}\]
    where all faces except possibly the front and back commute; the statement of the lemma is that the back commutes, and it suffices to prove that the front does. Using \cref{exam:edge-ce-identification}, we can identify the front as
    \[\begin{tikzcd} \bigoplus_{i=1}^n H^\rm{CE}_{i,2n-1}(\scr{G}(\bf{R})) \rar{\rm{inc}} \dar[swap]{\rm{inc} \circ \rm{pr}} & \bigoplus_{i=1}^\infty H^\rm{CE}_{i,2n-1}(\scr{G}(\bf{R}))\dar{\inc \circ \pr} \\[-5pt]
    \scr{G}_n(\bf{R}) & \scr{G}_n(\bf{R}) \lar[swap]{\id_n},\end{tikzcd}\]
    with projections onto the $n$th term, and it visibly commutes.
\end{proof}

\subsubsection{Rank filtrations on $E_\infty$-homology} To improve on \cref{lem:edge-homomorphism-in-rank-n}, we introduce two different filtrations on the homology groups $H_*^{E_\infty}(\bf{R}/(\sigma-1))$. They both arise by regarding $\bf{R}/(\sigma-1)$ as the colimit of the filtered object $\rkfil(\bf{R})/(\rkfil(\sigma)-1)$, but differ in which order we take image in homology or pass to indecomposables. Recall that for an filtered object $\bf{X}$ we use the notation $H_{n,d}(\bf{X}) \coloneq H_d(n^* \bf{X})$.

\begin{definition} \label{defn rank filtrations} Let $(\bb{Q}^\rm{nu}[t] \xrightarrow{\sigma} \bf{R} \xrightarrow{\epsilon} \bb{Q}^\rm{nu}[t]) \in (\Alg^\rm{red}_{E_\infty^\rm{u}}(\Fun(\bb{N},\DQ))_{/\bf{Q}^\rm{nu}[t]})^{\id_{\bf{Q}^\rm{nu}[t]}/}$.
    \begin{enumerate}[\noindent (i)]
        \item The \emph{Rognes rank filtration} on $H_*^{E_\infty}(\bf{R}/(\sigma-1))$ is the increasing filtration with filtration step $\fil^\rm{rog}_n H^{E_\infty}_d(\bf{R}/(\sigma-1))$ given by
        \[\rm{im}\big[H^{E_\infty}_{n,d}(\rkfil(\bf{R})/(\rkfil(\sigma)-1)) \to H^{E_\infty}_d(\bf{R}/(\sigma-1))\big].\]
        \item The \emph{indecomposable rank filtration} on $H_*^{E_\infty}(\bf{R}/(\sigma-1))$ is the increasing filtration with filtration step $\fil_n^\rm{ind} H_d^{E_\infty}(\bf{R}/(\sigma-1))$ given by 
        \[\rm{im}\big[H_{n,d}(\bf{R}) \to H_d(\bf{R}/(\sigma-1)) \to H_d^{E_\infty}(\bf{R}/(\sigma-1))\big].\]
    \end{enumerate}
\end{definition}

The Rognes rank filtration is that associated to the rank spectral sequence ${}^\infty E^1_{n,d}=\smash{H^{E_\infty}_{n,d}(\bf{R}/(\sigma))} \Rightarrow \smash{H_d^{E_\infty}(\bf{R}/(\sigma-1))}$ of \cref{lem:rank-spectral-sequences} \eqref{enum:rank-spectral-sequences-ii}. If $\bf{R}$ has a standard vanishing line of slope 2, then the Rognes rank filtration is exhaustive and satisfies $\rm{fil}_n^\rm{rog} H_{2n-1}^{E_\infty}(\bf{R}/(\sigma-1)) = H_{2n-1}^{E_\infty}(\bf{R}/(\sigma-1))$.

The indecomposable rank filtration is induced by the one associated with the rank spectral sequence $E^1_{n,d}=H_{n,d}(\bf{R}/(\sigma)) \Rightarrow H_d(\bf{R}/(\sigma-1))$ of \cref{lem:rank-spectral-sequences} \eqref{enum:rank-spectral-sequences-i}, under the canonical map $\rm{can} \colon H_*(\bf{R}/(\sigma-1)) \to \smash{H_*^{E_\infty}(\bf{R}/(\sigma-1))}$. It is exhaustive if and only if this canonical map is surjective. It admits an alternative description as
    \[\fil_n^\rm{ind} H_d^{E_\infty}(\bf{R}/(\sigma-1)) = \rm{im}\big[(H_{n,d}(\rkfil(\bf{R})/(\rkfil(\sigma)-1)) \to H_d(\bf{R}/(\sigma-1)) \to H_d^{E_\infty}(\bf{R}/(\sigma-1))\big].\]
From the existence of a map of spectral sequences 
    \[\big(E^1_{n,d}=H_{n,d}(\bf{R}/(\sigma)) \Rightarrow H_d(\bf{R}/(\sigma-1)\big) \lra \big({}^\infty E^1_{n,d}=H^{E_\infty}_{n,d}(\bf{R}/(\sigma)) \Rightarrow H_d^{E_\infty}(\bf{R}/(\sigma-1))\big)\]
of \cref{lem:rank-spectral-sequences} \eqref{enum:rank-spectral-sequences-iii}, we obtain the following:

\begin{lemma}\label{lem:indec-vs-rognes-rank-filtratino}
    For $(\bb{Q}^\rm{nu}[t] \xrightarrow{\sigma} \bf{R} \xrightarrow{\epsilon} \bb{Q}^\rm{nu}[t])$ as above, we have 
        \[\fil_n^\rm{ind} H_*^{E_\infty}(\bf{R}/(\sigma-1)) \subseteq \fil_n^\rm{rog} H_*^{E_\infty}(\bf{R}/(\sigma-1)).\]
\end{lemma}

\begin{example}\label{exam:rank-ce-identification} It is straightforward to describe the Rognes rank filtration and indecomposable rank filtration on $H_*(\rho_{\le \slopetwo}(\bf{R})/(\sigma-1))$. We have that
\[\fil^\rm{Rog}_n H_d^{E_\infty}(\rho_{\le \slopetwo}(\bf{R})/(\sigma-1)) \cong \begin{cases} \scr{G}_i(\bf{R}) & \text{if $d=2i-1$ and $n \geq i$} \\
0 & \text{else,}\end{cases}\]
from which one may read off the associated graded. Similarly, we have that 
\[\fil^\rm{ind}_n H_d^{E_\infty}(\rho_{\le \slopetwo}(\bf{R})/(\sigma-1)) \cong \begin{cases} \ker(\delta|_{\scr{G}_i(\bf{R})}) & \text{if $d=2i-1$ and $n \geq i$} \\
0 & \text{else.}\end{cases}\]
\end{example}

\begin{lemma}\label{lem:edge-in-ker-cobracket} Let $(\bb{Q}^\rm{nu}[t] \xrightarrow{\sigma} \bf{R} \xrightarrow{\epsilon} \bb{Q}^\rm{nu}[t])$ be as above with a standard vanishing line of slope 2. Then if $x \in \bigcup_{p \geq 0} \fil_p^\rm{ind} H_{2n-1}^{E_\infty}(\bf{R}/(\sigma-1)) \subseteq H_{2n-1}^{E_\infty}(\bf{R}/(\sigma-1))$, we have
        \[\rm{edge}_n(x) \in \ker(\delta|_{\scr{G}(\bf{R})_n}) \subseteq \scr{G}(\bf{R})_n.\] 
\end{lemma}

\begin{proof}By naturality, we may replace $\bf{R}$ with $\rho^\alg_{\le \slopetwo}(\bf{R})$ and use \cref{exam:edge-ce-identification,exam:rank-ce-identification} to identify the image of the union of the indecomposable rank filtration steps with the inclusion
\[\inc \colon H^\rm{CE}_{n,2n-1}(\scr{G}(\bf{R})) = \ker(\delta|_{\scr{G}(\bf{R})_n}) \lra \scr{G}_n(\bf{R}).\qedhere\]
\end{proof}

\begin{lemma}For $\bf{A} \in \Alg_{E^\rm{nu}_\infty}(\DQ)$, the following are equivalent:
\begin{enumerate}[(a)]
    \item $\bf{A}$ is free as an $E_\infty^\rm{nu}$-algebra,
    \item $H_*(\bf{A})$ is free as a (nonunital) graded-commutative algebra,
    \item $H_*(\bf{A}) \to H_*^{E_\infty}(\bf{A})$ is surjective.
\end{enumerate}
\end{lemma}

\begin{proof}
To prove the claim, we note the proof of \cite[Proposition 4.10]{KRS1} yields (a) $\Leftrightarrow$ (b). For (c) $\Rightarrow$ (a), we use that if $H_*(\bf{A}) \to H_*^{E_\infty}(\bf{A})$ is surjective then by picking lifts we can construct a map $\free_{E^\rm{nu}_\infty}(H_*^{E_\infty}(\bf{A})) \to \bf{A}$ which induces an isomorphism on homology and hence an equivalence. For (a) $\Rightarrow$ (c) we note that the map $H_*(\free_{E_\infty^\rm{nu}}(V)) \to \smash{H^{E_\infty}_*}(\free_{E_\infty^\rm{nu}}(V))$ can be identified with the map $\free_{E_\infty^\rm{nu}}(V) \to V$ sending all non-trivial products to zero.
\end{proof}

The previous results have the following crucial consequence.

\begin{corollary} \label{cor:image-of-edge-in-ker-cobracket}
Let $(\bb{Q}^\rm{nu}[t] \xrightarrow{\sigma} \bf{R} \xrightarrow{\epsilon} \bb{Q}^\rm{nu}[t])$ be as above with a standard vanishing line of slope 2. 
\begin{enumerate}[(i)]
    \item $H_*(\bf{R}/(\sigma-1))$ is free as an algebra if and only if the indecomposable rank filtration is exhaustive.
    \item If so, the edge homomorphism satisfies $\rm{im}(\rm{edge}_n) \subseteq \ker(\delta|_{\scr{G}_n(\bf{R})})$.
\end{enumerate} 
\end{corollary}

\begin{proof}The first part is proven by by specialising the lemma to $\bf{A}=\bf{R}/(\sigma-1)$, and the second part then follows using \cref{lem:edge-in-ker-cobracket}.\end{proof}

\subsubsection{The primitive rank filtration for algebras from spaces}
Suppose we have $E_\infty^\rm{u}$-algebra in $\Fun(\bb{N},\Spc)$ that is path-connected in each grading. Then by \cref{lem:rational-algebras-from-spaces} its rationalisation admits a canonical lift to $(\Alg_{E_\infty^\rm{u}}(\Fun(\bb{N},\DQ))_{/\bb{Q}[t]})^{\id_{\bb{Q}[t]}/}$, and by \cref{lem:rational-algebras-from-spaces-free} the homology of its stable algebra is free as an algebra.

Letting $\bf{R}_\bb{Q}$ denote the augmentation ideal the rationalisation, let us further suppose that $\bf{R}_\bb{Q}$ has a standard vanishing line of slope 2. Then the previous subsection gives us two exhaustive filtrations 
\[\fil^\rm{rog} H_*^{E_\infty}(\bf{R}_\bb{Q}/(\sigma-1)) \qquad \text{and} \qquad \fil^\rm{ind} H_*^{E_\infty}(\bf{R}_\bb{Q}/(\sigma-1))\]
of $H_*^{E_\infty}((\bf{R}_\bb{Q})/(\sigma-1))$, which is isomorphic to $\pi_{*}(\Omega_0^\infty \bf{R}^\gc) \otimes \bb{Q}$ by \cref{prop:stable-algebra-spc}. Moreover, we also know the first of these filtrations is associated with the spectral sequence of \cref{lem:rank-spectral-sequences} \eqref{enum:rank-spectral-sequences-ii} and its edge homomorphism has image in the kernel of the cobracket by \cref{cor:image-of-edge-in-ker-cobracket}.

We will now introduce yet another filtration and use it to refine the edge homomorphism. By the Milnor--Moore theorem \cite[Appendix]{MilnorMoore} in
\[\rm{prim}\,H_*(\bf{R}_\bb{Q}/(\sigma-1)) \lra H_*(\bf{R}_\bb{Q}/(\sigma-1)) \lra  H_*^{E_\infty}(\bf{R}_\bb{Q}/(\sigma-1))\]
the left map is injective, the right map is surjective, and their composition is an isomorphism. In particular, the composition is a preferred isomorphism between its domain and target.

\begin{definition}
    The \emph{primitive rank filtration} on $H_*^{E_\infty}(\bf{R}_\bb{Q}/(\sigma-1))$ is given by taking
    \[\fil_n^\rm{prim} H_d^{E_\infty}(\bf{R}_\bb{Q}/(\sigma-1)) \coloneq \rm{im}\big[H_{n,d}(\bf{R}_\bb{Q}) \to H_d(\bf{R}_\bb{Q}/(\sigma-1))\big] \cap \rm{prim}\,H_d(\bf{R}_\bb{Q}/(\sigma-1))\]
and identifying $\rm{prim}\,H_d(\bf{R}_\bb{Q}/(\sigma-1))$ with $H_d^{E_\infty}(\bf{R}_\bb{Q}/(\sigma-1))$.
\end{definition}

Projecting to $H_*^{E_\infty}(\bf{R}_\bb{Q}/(\sigma-1))$ we get that
\[\fil_n^\rm{prim} H_*^{E_\infty}(\bf{R}_\bb{Q}/(\sigma-1)) \subseteq \fil^\rm{ind} H_*^{E_\infty}(\bf{R}_\bb{Q}/(\sigma-1)).\]
A standard vanishing line of slope 2 for $\bf{R}_\bb{Q}$ implies homological stability by \cite[Theorem 18.3]{GKRW18}, so the primitive rank filtration is exhaustive. Thus the refined edge homomorphisms to be defined momentarily can detect all of algebraic $K$-theory rather than part of it. Consider the dashed map
\begin{equation}\label{eqn:prim-map} \begin{tikzcd} \fil^\rm{prim}_n H_d^{E_\infty}(\bf{R}_\bb{Q}/(\sigma-1)) \arrow[dashed]{dd} & \lar[two heads] \rm{stab}^{-1}(\fil^\rm{prim}_n H_d^{E_\infty}(\bf{R}_\bb{Q}/(\sigma-1))) \subseteq H_{n,d}(\bf{R}_\bb{Q}) \dar \\[-10pt]
& H_{n,d}(\bf{R}_\bb{Q}/(\sigma)) \dar \\[-10pt]
H_{n,d}^\rm{CE}(\scr{G}(\bf{R})) & H_{n,d}(\rho_{\le \slopetwo}^\alg(\bf{R}_\bb{Q})/(\sigma))\lar{\cong} \end{tikzcd} \end{equation}
defined by picking a lift along the surjective left-most map and applying the right-bottom map. This is well-defined because if $H_{n,d}(\bf{R}_\bb{Q})$ stabilises to zero then so does its image in $H_{n,d}(\rho_{\le \slopetwo}^\alg(\bf{R}_\bb{Q}))$ but there stabilisations maps are injective. By construction, the filtration step $\smash{\fil^\rm{prim}_{n-1} H_d^{E_\infty}(\bf{R}_\bb{Q}/(\sigma-1))}$ goes to zero and so induces a map on associated gradeds:

\begin{definition}\label{def:refined-edge-homomorphism} Suppose $\bf{R}$ has a standard vanishing line of slope 2. Then the \emph{refined edge homomorphism} is indued by the zigzag \eqref{eqn:prim-map}
 \[\rm{e}_{n,d} \colon \rm{gr}^\rm{prim}_n H_d^{E_\infty}(\bf{R}_\bb{Q}/(\sigma-1)) \lra H_{n,d}^\rm{CE}(\scr{G}(\bf{R})).\]
\end{definition}

It is natural in $\bf{R}$ and is related to the original edge homomorphism via:

\begin{lemma}\label{lem:filtered-map-to-gocharov-complex}
   The following diagram commutes
    \[\begin{tikzcd} \fil_n^\rm{prim} H_{2n-1}^{E_\infty}(\bf{R}_\bb{Q}/(\sigma-1)) \rar{\rm{pr}} \dar[swap]{\rm{inc}}& \rm{gr}^\rm{prim}_n H_{2n-1}^{E_\infty}(\bf{R}_\bb{Q}/(\sigma-1)) \dar{\rm{e}_{n,2n-1}} \\[-5pt]
    H_{2n-1}^{E_\infty}(\bf{R}_\bb{Q}/(\sigma-1)) \rar{\rm{edge}_n} & \scr{G}_n(\bf{R}).\end{tikzcd}\]
\end{lemma}

\begin{proof}The left-bottom composition is by \cref{lem:edge-homomorphism-in-rank-n} given by lifting an element to the preimage of $\rm{prim}\, H_{2n-1}(\bf{R}_\bb{Q}/(\sigma-1))$ in $H_{n,2n-1}(\bf{R}_\bb{Q})$ and taking its image under the canonical map $H_{n,2n-1}(\bf{R}_\bb{Q}) \to H^{E_\infty}_{n,2n-1}(\bf{R}_\bb{Q}) \cong \scr{G}_n(\bf{R})$ (which is independent of the choice of lift). The result follows since the following commutes by the definition of the refined edge homomorphism:
\[\begin{tikzcd}H_{n,2n-1}(\bf{R}_\bb{Q}) \dar[swap]{\rm{can}} \rar & H_{n,2n-1}(\bf{R}_\bb{Q}/(\sigma)) \dar{\rm{can}} \\[-5pt]
H^{E_\infty}_{n,2n-1}(\bf{R}_\bb{Q}) \rar{\cong} & H^{E_\infty}_{n,2n-1}(\bf{R}_\bb{Q}/(\sigma)).\end{tikzcd} \qedhere \]
\end{proof}

\section{Field of transcendence degree $1$ over finite fields}\label{sec:transcendence-degree-1}
Before commencing the discussion of number fields and the proof of \cref{thm:polyl-iso-number-field}, we will work out a case that may be of independent interest: that where $F$ is a field of transcendence degree $1$ over $\bb{F}_p$ for some prime number $p$, e.g.~$\bb{F}_q(t)$ for $q = p^n$. These fields are special because work of Harder allows us to fully describe the algebra $\BGLb(F)_\bb{Q}/(\sigma)$ and determine its $E_\infty$-indecomposables. We will see there is a preferred isomorphism of Lie coalgebras
        \[\scr{G}(F) \overset{\cong}\lra \rm{cofree}_{\rm{coLie}}(F^\times_\bb{Q}),\]
where $F^\times_\Q$ is placed in weight $1$. As an application, we deduce the distribution relations for classical polylogarithms in the case of fields of positive characteristic (see \cref{cor:char-p-distribution}).

\subsection{Recollection of the work of Harder} 
Our starting point will be the following result \cite[Satz 1]{Harder} for the homology of special linear groups:

\begin{theorem}[Harder] \label{thm:harder-sln} Let $F$ be a field of transcendence degree 1 over a finite field. Then for all $n \ge 1$ we have $\widetilde{H}_*(\SL_n(F);\bb{Q})=0$.\end{theorem}

More precisely, Harder studies global function fields, which are finitely generated function field of a smooth projective curve over $\bb{F}_q$, and a general field of transcendence degree 1 over $\bb{F}_p$ is a filtered colimit of such. The analogous statement for general linear groups is:

\begin{corollary}\label{cor:harder-gln} For all $n \geq 1$ the determinant induces an isomorphism
\[{\det}_* \colon H_*(\GL_n(F);\bb{Q}) \overset{\cong}\lra H_*(F^\times;\bb{Q}).\]
\end{corollary}

\begin{proof}The Lyndon--Hochschild--Serre spectral sequence associated with the short exact sequence of groups $1 \to \SL_n(F) \to \GL_n(F) \smash{\xrightarrow{\det}} F^\times \to 1$ is given by
    \[E^2_{p,q}=H_p(F^\times;H_q(\SL_n(F);\bb{Q})) \Longrightarrow H_{p+q}(\GL_n(F);\bb{Q}).\]
It is concentrated on the row $q=0$ by \cref{thm:harder-sln}, and there given by $E^2_{p,0} = H_p(F^\times;\bb{Q}) \cong \Lambda^p F^\times_\bb{Q}$. This implies the statement.\end{proof}

Since the determinant map is compatible with stabilisation, this implies:

\begin{corollary}\label{cor:harder-gln-stab} The iterated stabilisation map 
    \[(s_{n,m})_* \colon H_*(\GL_m(F);\bb{Q}) \lra  H_{*}(\GL_n(F);\bb{Q})\]
    is an isomorphism for any $n \geq m \ge 1$.
\end{corollary}

\begin{proof}
    The stabilisation map $s_{n,1} \colon \GL_1(F) \to \GL_n(F)$ has the property that $\det \circ s_{n,1}$ is equal to the usual identification of $\GL_1(F)$ with $F^\times$. Since $\det_* \colon H_*(\GL_n(F);\bb{Q}) \to H_*(F^\times;\bb{Q})$ is an isomorphism, so is $(s_{n,1})_* \colon H_{*}(\GL_1(F);\bb{Q}) \to  H_{*}(\GL_n(F);\bb{Q})$, and the statement is deduced by observing $s_{n,m} \circ s_{m,1} = s_{n,1}$
\end{proof}

\subsection{The $E_\infty$-indecomposables of $\BGLb(F)_\Q$} We use this to give a full description of the algebra $\BGLb(F)_\Q/(\sigma)$ and its $E_\infty$-indecomposables. The latter allows us to recover the cotangent complex of $\BGLb$ as an $E_\infty$-algebra since
\[\cot_{E_\infty}(\BGLb(F)_\Q) \simeq \Q\{\sigma\} \oplus \cot_{E_\infty}(\BGLb(F)_\Q/(\sigma)),\]
with splitting provided by the maps $\Q[t]^\rm{nu} \xrightarrow{\sigma} \BGLb(F)_\Q \xrightarrow{\epsilon} \Q[t]^\rm{nu}$.

\begin{theorem}\label{thm:harder-alg-indec}\,
    \begin{enumerate}[(i)]
        \item \label{enum:harder-alg-indec-i} There is an equivalence in $\Alg_{\rm{Com}}(\Fun(\bb{N},\DQ))$
        \[\BGLb(F)_\Q/(\sigma) \simeq \rm{triv}_{\rm{com}}(1_!(\Lambda^{\ge 1} F^\times_\bb{Q})),\] 
        where $1_!(\Lambda^{\ge 1} F^\times_\bb{Q}) \in \Fun(\bb{N},\DQ)$ denotes the bigraded vector space with $\Lambda^i F^\times_\bb{Q}$ in bidegree $(1,i)$ for each $i \ge 1$. 
        \item \label{enum:harder-alg-indec-ii} There is an equivalence in $\rm{coAlg}_{\rm{coLie}}(\Fun(\bb{N},\DQ))$
        \[\indec_{E_\infty} (\BGLb(F)_\Q/(\sigma))[1] \simeq \rm{cofree}_{\rm{coLie}}(s\, 1_!(\Lambda^{\ge 1} F^\times_\bb{Q})).\]
        \item \label{enum:harder-alg-indec-iii} There is an isomorphism of (graded) Lie coalgebras 
        \[\PolyL(F) \cong \cofree_\rm{coLie}(F^\times_\Q).\]
    \end{enumerate}
\end{theorem}

\begin{proof}
    Part \eqref{enum:harder-alg-indec-ii} follows from \eqref{enum:harder-alg-indec-i} by Koszul duality and part \eqref{enum:harder-alg-indec-iii} follows from part \eqref{enum:harder-alg-indec-ii} by truncation. To prove \eqref{enum:harder-alg-indec-i}, firstly, \cref{cor:harder-gln-stab} implies there is an equivalence in $\Alg_{\rm{Com}}(\Fun(\bb{N},\DQ))$
    \[H_{*,*}(\BGLb(F)_\bb{Q}/(\sigma)) \simeq \rm{triv}_{\rm{com}}(1_!(\Lambda^{\ge 1} F^\times_\bb{Q})),\]
    where the left-hand side is interpreted as a formal algebra. Secondly, the Frobenius $\varphi$ acts on $\BGLb(F)_\bb{Q}$ fixing $\sigma$, i.e.~in a way that the canonical sequence $\bb{Q}^\rm{nu}[t] \xrightarrow{\sigma} \BGLb(F)_\bb{Q} \xrightarrow{\epsilon} \bb{Q}^\rm{nu}[t]$ is a splitting of algebras with endomorphism, where $\varphi$ acts trivially on the outer terms. Thus, $\varphi$ acts on $\BGLb(F)_\bb{Q}/(\sigma)$ as an algebra. 
    By the above description of the homology of $\BGLb(F)_\bb{Q}/(\sigma)$ it acts by multiplication by $p^d$ on $H_{*,d}(\BGLb(F)_\bb{Q}/(\sigma))$ because it acts by multiplication by $p^d$ on $\Lambda^d F^\times_\bb{Q}$. Thus $\BGLb(F)_\bb{Q}/(\sigma)$ is formal by the bigraded variant of \cite[Theorem 12.7]{Sullivan} or \cite[Section 7.2]{EmprinHorel}.
\end{proof}

\begin{remark}\label{rem:transcedence-degree-1-mtm} We believe known results can be used to construct a Tannakian category of mixed Tate motives over a field $F$ of transcendence degree $1$ over a finite field, and a realisation functor inducing $\scr{G}(F) \to \scr{L}^{\rm{MTM}}(F)$. Then \cref{thm:harder-alg-indec} should imply the analogue of \cref{thm:polyl-iso-number-field}.\end{remark}

\subsection{Symbol map in the Goncharov Lie coalgebra}
We prefer an explicit description of the isomorphism of \cref{thm:harder-alg-indec} \eqref{enum:harder-alg-indec-iii} in terms of the generators of $\PolyL(F)$. To do so, we need the \emph{symbol map} of the Goncharov Lie coalgebra; here $F$ can be any field.

Let $\scr{L}$ be a reduced graded Lie coalgebra,  that is, $\scr{L}_0 = 0$. The canonical projection $\fgt_\coLie(\scr{L}) \to \scr{L}_1$ gives a canonical map of graded Lie coalgebras called the \emph{symbol} map
\[s \colon \scr{L} \lra \cofree_\coLie(\scr{L}_1).\]
It is an isomorphism if $\scr{L}$ is a cofree Lie coalgebra cogenerated in degree $1$.

\begin{lemma} \label{lem symbol map}
    The symbol map of the Goncharov Lie coalgebra is given by
    \begin{align*}s \colon \PolyL(F) &\lra \cofree_\coLie(F^\times_\bb{Q}) \\\CorG(x_0,\dots,x_n) &\longmapsto \sum_{\iota=((i_1,j_1),\compactldots,(i_n,j_n)) \in T(n)} \rm{sign}(\iota) [x_{i_1}-x_{j_1}] \otimes \compactcdots \otimes [x_{i_n}-x_{j_n}],\end{align*}
     where we omit terms where $x_{i_k}=x_{j_k}$ for some $1 \le k \le n$ and notation is as in \cite[Proposition 2.27]{KRS1}. In particular for $a \in F^{\times}$ and $n \ge 2$ we have (interpreted as $0$ for $a=1$)
    \[s(\LiG_n(a))=-[1-a] \otimes [a] \otimes \compactcdots \otimes [a] \in \coLie_n(F^\times_\Q).\]
\end{lemma}

\begin{proof}
    By \cite[Theorem 7.28]{KRS1}, there is a formal realisation map of $r^\rm{f} \colon \PolyL(F) \to \scr{L}^\rm{f}(F)$ to the Lie coalgebra of formal polylogarithms $\scr{L}^\rm{f}(F)$ of \cite{CMRR24} sending $\CorG(x_0,\compactldots,x_n)$ to $\rm{Cor}^\rm{f}(x_0,\compactldots,x_n)$. This is an isomorphism for $n=1$ since $\PolyL(F) \cong \smash{F^\times_\bb{Q}}$ via $\CorG(x_0,x_1) \in \PolyL(F) \mapsto (x_1-x_0) \in \smash{F^\times_\bb{Q}}$ by \cite[Lemma 7.8]{KRS1} and $\PolyLF_1(F) \cong F^\times$ via $\rm{Cor}^\rm{f}(x_0,x_1) \in \PolyLF_1(F) \mapsto (x_1-x_0) \in \smash{F^\times_\Q}$ by \cite[Lemma 1]{CMRR24}. Thus, one gets a commutative square
    \[
        \begin{tikzcd}
        \PolyL(F) \rar{s} \dar[swap]{r^\rm{f}} & \cofree_\coLie(F^\times) \dar[equal] \\[-7pt]
        \scr{L}^\rm{f}(F) \rar{s^\rm{f}} &  \cofree_\coLie(F^\times)
    \end{tikzcd}
    \]
    which reduces the computation to showing that  
     \[s^\rm{f}(\rm{Cor}^\rm{f}(x_0,\dots,x_n)) = \sum_{\iota=((i_1,j_1),\compactldots,(i_n,j_n)) \in T(n)} \rm{sign}(\iota) [x_{i_1}-x_{j_1}] \otimes \compactldots \otimes [x_{i_n}-x_{j_n}],\]
    where we omit any terms where $x_{i_k}=x_{j_k}$. This follows from \cite[(12)]{CMRR24} and the universal symbol formula of \cite[Section 2.6.1]{KRS1} in the generic case where $x_i \neq x_j$ for $i \neq j$, and in general by using specialisation maps \cite[Section 2.2]{CMRR24}. Finally, the computation for classical polylogarithms follows since 
     \[\LiG_n(a)=-\CorG(1,0,\compactldots,0,a)\]
    for $n \ge 2$ and hence only the first term in the symbol expression is non-zero. 
\end{proof}

By \cref{thm:harder-alg-indec} \eqref{enum:harder-alg-indec-iii} there is some isomorphism of Lie coalgebras $\scr{G}(F) \cong \rm{cofree}_{\rm{coLie}}(F^\times_\bb{Q})$, but in fact the symbol map of \cref{lem symbol map} is an explicit such isomorphism.

\subsection{Application to distribution relations}

\subsubsection{Field homomorphisms} The constructions in \cite{KRS1} are natural in field homomorphisms. If $\chi \colon F \to F'$ is a homomorphism of fields, then it induces a map $\chi_* \coloneq F' \otimes_F (-) \colon \Vect_F \to \Vect_{F'}$ of symmetric monoidal groupoids. As the notation suggests, this takes an object $V$ to $V^\chi \coloneq F' \otimes_F V$, and takes an isomorphism $A \colon V \to W$ to the induced isomorphism $\id_{F'} \otimes A \colon F' \otimes_F V \to F' \otimes_F W$, in terms of matrices given by applying $\chi$ to all entries.

\begin{example}If $F$ has characteristic $p \neq 0$, then the \emph{Frobenius} $\rm{Frob} \colon F \to F$ given by $\lambda \mapsto \lambda^p$ is a field endomorphism, and it is an automorphism if $F$ is perfect. We get an induced map $\rm{Frob}_* \colon \Vect \to \Vect$ by the above construction, for which we will have a use as it is known how the Frobenius acts on the rationalised algebraic $K$-theory groups of $F$.\end{example}

This induces a map of $E_\infty^\rm{nu}$-algebras $\chi \colon \BGLb(F) \to \BGLb(F')$, and in turn a map 
\begin{align*}\chi \colon \scr{G}(F) &\lra \scr{G}(F') \\
\CorG(x_0,\compactldots,x_n) &\longmapsto \CorG(\chi(x_0),\compactldots,\chi(x_n))\end{align*}
of Lie coalgebras, as can be seen by tracing through the constructions in \cite{KRS1}.

\subsubsection{Distribution relations for classical polylogarithms}
We finish with an application: a proof of the distribution relations for the ``classical polylogarithm'' elements $\LiG_n(x) \in \scr{G}_n(F)$ for $F$ \emph{any} field of characteristic $p$.

\begin{corollary}\label{cor:char-p-distribution}
Let $n\geq 1$ and let $F$ be a field of positive characteristic $p$. Let $d \in \N_{>0}$ and suppose that the polynomial $t^d-1$ splits in $F$. Then, for $x \in F^\times$ the following equation holds in $\PolyL_n(F)$
\[\LiG_n(x^d)=d^{n-1} \sum_{\zeta^d=1} \LiG_n(\zeta x),\]
where we sum over all roots of $t^d-1$ in $F$, counted with multiplicity. In particular, $\LiG_n(x^p)=p^n \LiG_n(x)$.
\end{corollary}

\begin{proof}
The second part follows from the first one as $t^p-1=(t-1)^p$ splits over any field $E$ of positive characteristic $p$. Thus, it suffices to prove the first part of the corollary.

By naturality, it suffices to verify the distribution relation in the subfield of $F$ generated by the $d$-th roots of unity and $x \in F$, and then we can assume without loss of generality that $\rm{trdeg}_{\bb{F}_p}(F) \le 1$ (depending on whether $x$ is algebraic or transcendental over $\bb{F}_p$). 
When $\rm{trdeg}_{\bb{F}_p}(F)=0$ the result is trivial since $\PolyL(F)=0$, so we can assume that $\rm{trdeg}_{\bb{F}_p}(F)=1$. There is a commutative diagram, see the proof of \cref{lem symbol map}
    \begin{center}
        \begin{tikzcd}
        \PolyL(F) \rar{s} \dar[swap]{r^\rm{f}} & \cofree_\coLie(F^\times_\Q) \dar[equal] \\[-7pt]
        \scr{L}^\rm{f}(F) \rar{s^\rm{f}} &  \cofree_\coLie(F^\times_\Q)
    \end{tikzcd}
    \end{center}
    and the top map is an isomorphism since $\rm{trdeg}_{\bb{F}_p}(F)=1$. Thus, the formal realisation $\smash{r^\rm{f}} \colon \smash{\PolyL(F)} \to \smash{\PolyLF(F)}$ must be injective. However, it also is surjective by definition (for any field), so it must be an isomorphism. Hence, it suffices to prove the distribution relation on $\PolyLF(F)$, where it holds as a consequence of \cite[Propositions 12, 13]{CMRR24}. 
\end{proof}

\section{Number fields} In this section we discuss the case where $F$ is a number field, i.e.~a finite extension of $\bb{Q}$. What makes this accessible is the work of Borel and Yang \cite{borelyang}, relying on that of Blasius, Franke, and Grunewald \cite{BFG}, which provides a description of the rational cohomology of $\GL_n(F)$. This uses automorphic methods and thus relies on deep work in analysis.

\subsection{Recollection of the work of Borel and Yang} \label{section: borel-yang}
We start by giving a short explanation of the work of Borel and Yang \cite{borelyang}---which builds on the work of Blasius, Franke, and Grunewald \cite{BFG}---and extract three statements from their work, see \cref{prop:borel-yang-stab}. 

Using automorphic techniques, they compute the rational cohomology of $\SL_n(F)$ and $\GL_n(F)$ for a number field $F$, and describe the stabilisation maps in terms of these computations. Suppose that $F$ has $r_1$ real embeddings and $r_2$ pairs of non-real complex embeddings. The following is \cite[Theorem 2.1, Section 4.1]{borelyang} (we always take $E=\bb{Q}$, suppressing it from the notation):

\begin{theorem}[Borel--Yang] \label{thm:borelyang-sln} For all $n \geq 1$ we have
\[H^*(\SL_n(F);\bb{Q}) \cong \begin{cases}
    S^\bullet_n \otimes \bigotimes_{r_1} \Lambda^*(e_{n,n}) & \text{if $n$ is even,} \\
    S^\bullet_n & \text{if $n$ is odd.}
\end{cases}\]
\end{theorem}

Here, firstly, the ``Euler classes'' $e_{n,n}$ are of degree $n$, one for each of the $r_1$ real embeddings $F \hookrightarrow \bb{R}$. Secondly, $S^\bullet_n$ is a tensor product \cite[Section 4.1]{borelyang}
\[S^\bullet_n \cong \left(\bigotimes_{r_1} \begin{cases} S^*(x_{n,5},x_{n,9},\ldots,x_{n,2n-3}) & \text{if $n$ is even} \\
S^*(x_{n,5},x_{n,9},\ldots,x_{n,2n-1}) & \text{if $n$ is odd} \end{cases}\right) \otimes \left(\bigotimes_{r_2} S^*(x_{n,3},x_{n,5},\ldots,x_{n,2n-1})\right)\]
where the ``Borel classes'' $x_{n,i}$ are of degree $i$, one for each embedding $F \hookrightarrow \bb{C}$, or $F \hookrightarrow \bb{R}$ and $i \equiv 1 (\rm{mod} 4)$, p.~693 loc.cit.. The inclusion $\rm{SL}_m(F) \to \rm{SL}_n(F)$ for $m<n$ induces a map 
\[j_{m,n} \colon H^*(\SL_n(F);\bb{Q}) \lra H^*(\SL_m(F);\bb{Q}).\]
By \cite[Section 4.2]{borelyang}, this takes each $x_{n,i}$ for a real or complex embedding to the corresponding $x_{m,i}$, to be interpreted as zero if $i$ is larger than $2m-1$ (if $m$ is odd or the embedding is complex) or $2m-3$ (if $m$ is even and the embedding is real), and takes each $e_{n,n}$ to zero. 

The homology of $\rm{SL}_n(F)$ is isomorphic to the linear dual of \cref{thm:borelyang-sln}, using the universal coefficients theorem and that these (co)homology groups are degreewise finite-dimensional; we denote the preduals of the $x_{n,i}$ by $\sf{x}_{n,i}$. The analogous statement for general linear groups is as follows (note that the preduals to the Euler classes $e_{n,n}$ have disappeared):

\begin{corollary}\label{cor:borelyang-gln} For all $ n \geq 1$ we have $H_*(\GL_n(F);\bb{Q}) \cong \Lambda^* F^\times_\bb{Q} \otimes (S^\bullet_n)^\vee$.
\end{corollary}

\begin{proof}We apply the Lyndon--Hochschild--Serre spectral sequence associated to the short exact sequence of groups $1 \to \SL_n(F) \to \GL_n(F) \smash{\xrightarrow{\det}} F^\times \to 1$, given by
\[E_{p,q}^2 = H_p(F^\times;H_q(\SL_n(F);\bb{Q})) \Longrightarrow H_{p+q}(\GL_n(F);\bb{Q}),\]
and by the dual of \cite[Theorem 7.5]{borelyang}, the $E_2$-page can be identified as the tensor product $H_p(F^\times) \otimes H_q(\SL_n(F);\bb{Q})_{F^\times}$ where one has $\smash{H^q(\SL_n(F);\bb{Q})_{F^\times}} = (S^\bullet_n)^\vee$, and the spectral sequence collapses at the $E^2$-page.\end{proof} 

Using compatibility of the spectral sequence of the above corollary with stabilisation, we get the following, implicit in the proof of \cite[Lemma 7.4]{borelyang}:

\begin{corollary}\label{cor-borelyang-gln-stab} The iterated stabilisation map
\[(s_{n,m})_* \colon H_*(\GL_m(F);\bb{Q}) \lra H_*(\GL_n(F);\bb{Q})\]
is given by $\id \otimes (j_{m,n})^\vee$ with respect to the isomorphisms of \cref{cor:borelyang-gln} for all $n \geq m \geq 1$.\end{corollary}

We need three consequences. The first consequence concerns the stabilisation map and its relative homology groups:

\begin{proposition}
    \label{prop:borel-yang-stab} \,
    \begin{enumerate}[(i)]
    \item \label{enum:borel-yang-stab-i} The stabilisation map is always injective and thus the canonical map
    \[H_*(\GL_n(F);\bb{Q}) \lra H_*(\GL_n(F),\GL_{n-1}(F);\bb{Q})\]
    is always surjective.
    \item \label{enum:borel-yang-stab-ii} For $n \geq 1$ we have that $H_*(\GL_n(F),\GL_{n-1}(F);\bb{Q})=0$ for $*<2n-1$.
    \item \label{enum:borel-yang-stab-iii} For $n \geq 2$ we have that
         \[\quad H_{2n-1}(\GL_n(F),\GL_{n-1}(F);\bb{Q}) \cong \begin{cases} 
        \bigoplus_{r_1} \bb{Q}\{\overline{\sf{x}}_{n,2n-1}\} \oplus \bigoplus_{r_2} \bb{Q}\{\overline{\sf{x}}_{n,2n-1}\} & \text{if $n$ is odd,} \\
        \bigoplus_{r_2} \bb{Q}\{\overline{\sf{x}}_{n,2n-1}\} & \text{if $n$ is even,}\end{cases}\]
    where the $\overline{\sf{x}}_{n,2n-1}$ have degree $2n-1$, given by the image of the $\sf{x}_{n,2n-1}$ under the canonical map $H_*(\GL_n(F);\bb{Q}) \to H_*(\GL_n(F),\GL_{n-1}(F);\bb{Q})$.
\end{enumerate}
\end{proposition}

\begin{proof}Consider the map in \cref{cor-borelyang-gln-stab}. Part \eqref{enum:borel-yang-stab-i} follows since $j_{n-1,n}$ is visibly surjective and hence its dual is injective. Part \eqref{enum:borel-yang-stab-ii} follows by \cref{cor:borelyang-gln} and inspection of $j_{n-1,n}^\vee$. Part \eqref{enum:borel-yang-stab-iii} follows by considering several cases: the map $S^\bullet_n \to S^\bullet_{n-1}$ is on a term for a real embedding  an isomorphism if $n$ is even and surjective with kernel the ideal generated by $x_{n,2n-1}$ if $n$ is odd, and on a term for a complex embedding surjective with kernel the ideal generated by $x_{n,2n-1}$. Also, there are no additional contributions from units by \cref{cor:borelyang-gln} and part (ii).\end{proof}

The cohomology classes $x_{n,i}$ and $e_{n,n}$ are the indecomposables for the cup product on $H^*(\SL_n(F);\bb{Q})$, so the homology classes $\sf{x}_{n,i}$ and $\sf{e}_{n,n}$ are the primitives for the coproduct on $H_*(\SL_n(F);\bb{Q})$ induced by the diagonal. The primitives in $H_*(\GL_n(F);\bb{Q})$ are given by
\[\prim \, H_*(\GL_n(F);\bb{Q})=F^\times_\Q \oplus (S^\bullet_n)^\vee\]
since the isomorphism of \cref{cor:borelyang-gln} is one of coalgebras because it comes from a Lyndon--Hochschild--Serre spectral sequence and there are no extension issues.

If $n \ge 2$ then $2n-1>n$, and hence there is an isomorphism $\prim\, H_{2n-1}(\SL_n(F);\bb{Q}) \xrightarrow{\cong} \prim\, H_{2n-1}(\GL_n(F);\bb{Q})$ with the latter given by
\[\prim\, H_{2n-1}(\GL_n(F);\bb{Q}) \cong \begin{cases} 
        \bigoplus_{r_1} \bb{Q}\{{\sf{x}}_{n,2n-1}\} \oplus \bigoplus_{r_2} \bb{Q}\{{\sf{x}}_{n,2n-1}\} & \text{if $n$ is odd,} \\
        \bigoplus_{r_2} \bb{Q}\{{\sf{x}}_{n,2n-1}\} & \text{if $n$ is even,}\end{cases}\] 
as degrees rule out contributions from the Euler classes and units. Thus, using \cref{prop:borel-yang-stab}\eqref{enum:borel-yang-stab-iii} we obtain the following relationship between primitive elements and relative homology of the stabilisation map: 

\begin{lemma}\label{lem:borel-yang-prim-rel-iso}For $n \geq 2$, in the commutative diagram
\[\begin{tikzcd} \rm{prim}\,H_{2n-1}(\SL_n(F);\bb{Q}) \rar{\cong} \dar{\cong} & \rm{prim}\, H_{2n-1}(\GL_n(F);\bb{Q}) \dar{\rm{inc}} \\[-5pt]
H_{2n-1}(\GL_n(F),\GL_{n-1}(F);\bb{Q}) & H_{2n-1}(\GL_n(F);\bb{Q}) \lar \end{tikzcd}\]
the top horizontal and left vertical maps are isomorphisms.\end{lemma}

Tensoring the above discussion with $\bb{R}$, the stabilisation map and Hurewicz map give rise to the horizontal maps in
\begin{equation}\label{eqn:borel-yang-regulators} \begin{tikzcd}[ampersand replacement=\&] \mathrm{prim}\,H_{2n-1}(\GL_n(F);\bb{R}) \rar \& \mathrm{prim}\,H_{2n-1}(\GL_\infty(F);\bb{R}) \dar \& \lar{\rm{Hur}} K_{2n-1}(F)_\bb{R} \\[-10pt]
\& \parbox{3cm}{$\begin{cases}\bb{R}^{r_1+r_2} & \text{if $n$ is even,} \\
\bb{R}^{r_2} & \text{if $n$ is odd,}\end{cases}$} \& \end{tikzcd}\end{equation}
while the vertical map is given by the Borel regulators for all real embeddings and conjugate pairs of complex embeddings, to be described momentarily in \cref{subsec:number-fields-injectivity}.

\begin{lemma}\label{lem:borel-yang-regulators}
The maps in \eqref{eqn:borel-yang-regulators} are isomorphisms of finite-dimensional real vector spaces.
\end{lemma}

\begin{proof}The left horizontal map is an isomorphism by \cref{cor-borelyang-gln-stab} and the right horizontal map is an isomorphism by the Milnor--Moore theorem \cite[Appendix]{MilnorMoore}. The vertical map is an isomorphism by Borel's computation of the rational algebraic $K$-theory of number fields \cite[Section 12]{BorelStable} and its interpretation in terms of regulator maps \cite{BorelReg,BorelRegErrata}.\end{proof}

\subsection{Technical $E_\infty$-algebra tools} The main result of this subsection is \cref{thm:image-edge-is-cobracket}, which will be used to compute the Goncharov Lie coalgebra for $\BGLb(F)_\bb{Q}$. The intuition is this: for a trivial $E_\infty$-algebra, its homology injects into $E_\infty$-homology and the Koszul dual is cofree, and the theorem says that under mild conditions one has a similar result in a range.

\medskip

We recall some notation. Recall that $\bb{Q}^\rm{nu}[t] \in \Alg_{E_\infty^\rm{nu}}(\Fun(\bb{N},\DQ))$ is the augmentation ideal of $\bb{Q}[t] = \free_{E_\infty^\rm{u}}(1_! 1_\DQ)$, and that for $\sigma \colon \bb{Q}^\rm{nu}[t] \to \bf{R}$ we let $\bf{R}/(\sigma)$ be the cofibre in $\Alg_{E_\infty^\rm{nu}}(\Fun(\bb{N},\DQ))$. If $\sigma \colon \bb{Q}^\rm{nu}[t] \to \bf{R}$ has a standard vanishing line of slope 2 as in \cref{def:std-vanishing-line}, we can truncate its indecomposables to obtain a shifted Lie coalgebra
\[\scr{G}(\bf{R})[-1] = \rho_{\le \slopetwo}(\indec_{E^\rm{nu}_\infty}(\bf{R}/(\sigma))).\]
Recall that an $E_\infty^\rm{nu}$-algebra $\bf{R}$ in $\Fun(\bb{N},\DQ)$ is \emph{reduced} if $\bf{R}(0) \simeq 0$ and let $\Alg_{E_\infty^\rm{nu}}^\rm{red}(\Fun(\bb{N},\DQ))$ denote the full subcategory of reduced $E_\infty^\rm{nu}$-algebras. We will see momentarily, in \cref{lem:transferring-vanishing-lines-abs}, that \cref{thm:image-edge-is-cobracket} \eqref{enum:image-edge-is-cobracket-hypothesis-1} implies that $\bf{R}$ has a standard vanishing line of slope 2. Finally, the stable algebra $\bf{R}/(\sigma-1)$ is defined by implicitly first unitalising.

\begin{theorem} \label{thm:image-edge-is-cobracket}
    Let $\big(\bb{Q}^\rm{nu}[t] \xrightarrow{\sigma} \bf{R} \xrightarrow{\epsilon} \bb{Q}^\rm{nu}[t]\big)\in (\Alg_{E_\infty^\rm{nu}}^\rm{red}(\Fun(\bb{N},\DQ))_{/\bb{Q}^\rm{nu}[t]})^{\id_{\bb{Q}^\rm{nu}[t]}/}$ satisfy
    \begin{enumerate}[(i)]
        \item \label{enum:image-edge-is-cobracket-hypothesis-1} $H_{n,d}(\bf{R}/(\sigma))=0$ for $d<2n-1$,
        \item \label{enum:image-edge-is-cobracket-hypothesis-2} the canonical maps $\rm{can}_n \colon H_{n,2n-1}(\bf{R}/(\sigma)) \to H_{n,2n-1}^{E_\infty}(\bf{R}/(\sigma))$ are injective for all $n \ge 1$,
        \item \label{enum:image-edge-is-cobracket-hypothesis-3} the maps $H_{n,2n-1}(\bf{R}) \to H_{n,2n-1}(\bf{R}/(\sigma))$ are surjective for all $n \ge 1$,
        \item \label{enum:image-edge-is-cobracket-hypothesis-4} the stable algebra $\bf{R}/(\sigma-1) \in \Alg_{E^\rm{u}_\infty}(\DQ)$ is free.
    \end{enumerate}
    Then we have that:
    \begin{enumerate}[\noindent (1)]
        \item \label{enum:image-edge-is-cobracket-conclusion-1}There exists an isomorphism of graded Lie coalgebras
        \[\scr{G}(\bf{R}) \overset{\cong}\lra \cofree_{\rm{coLie}}\Big(\bigoplus_n H_{n,2n-1}(\bf{R}/(\sigma))[1]\Big)\]
        so the canonical maps $\rm{can}_n \colon H_{n,2n-1}(\bf{R}/(\sigma)) \to H_{n,2n-1}^{E_\infty}(\bf{R}/(\sigma))$ agree with the inclusion of the cogenerators (after suspension).
        \item \label{enum:image-edge-is-cobracket-conclusion-2} The image of the canonical map is given by
        \[\rm{im}\big[H_{n,2n-1}(\bf{R}/(\sigma)) \xrightarrow{\rm{can}_n} \scr{G}_n(\bf{R}) = H_{n,2n-1}^{E_\infty}(\bf{R}/(\sigma))\big] = \ker(\delta|_{\scr{G}_n(\bf{R})}).\]
        \item \label{enum:image-edge-is-cobracket-conclusion-3} The image of the edge homomorphism is given by
        \[\rm{im}\big[ H_{2n-1}^{E_\infty}(\bf{R}/(\sigma-1)) \xrightarrow{\rm{edge}_n} \scr{G}_n(\bf{R})=H_{n,2n-1}^{E_\infty}(\bf{R}/(\sigma))\big] = \ker(\delta|_{\scr{G}_n(\bf{R})}).\]
    \end{enumerate}
\end{theorem}

\begin{remark}It may be of interest that parts \eqref{enum:image-edge-is-cobracket-conclusion-1} and \eqref{enum:image-edge-is-cobracket-conclusion-2} only use hypotheses \eqref{enum:image-edge-is-cobracket-hypothesis-1} and \eqref{enum:image-edge-is-cobracket-hypothesis-2}.
\end{remark}

\subsubsection{Transfer of vanishing lines} For any map $f \colon \bf{A} \to \bf{B}$ in $\Alg_{E_\infty^\rm{nu}}(\Fun(\bb{N},\DQ))$ there is a \emph{relative bar spectral sequence} of the form
\[E^1_{n,p,q}=H_{n,q}(\bf{B}^{\otimes p},\bf{A}^{\otimes p}) \Longrightarrow H_{n,p+q-1}^{E_1}(\bf{B},\bf{A}),\]
arising by filtering by skeleta the bar constructions of the unitalisations (see \cite[Proposition 14.5]{GKRW18}). We start with two lemmas obtained using this bar spectral sequence: 

\begin{lemma}\label{lem:transferring-vanishing-lines-abs}
    Let $\bf{A}\in \Alg_{E_\infty^\rm{nu}}^\rm{red}(\Fun(\bb{N},\DQ))$ be such that $H_{n,d}(\bf{A})=0$ for $d<2n-1$, then $\smash{H_{n,d}^{E_\infty}(\bf{A})}=0$ for $d<2n-1$ as well.
\end{lemma}

\begin{proof}
In the bar spectral sequence $E^1_{n,p,q}=H_{n,q}(\bf{A}^{\otimes p}) \Rightarrow H_{n,p+q-1}^{E_1}(\bf{A})$, using the assumption and the K\"unneth theorem, we have $\smash{E^1_{n,p,q}} = 0$ for $q<2n-p$ and we conclude $\smash{H_{n,d}^{E_1}(\bf{A})}=0$ for $d<2n-1$. Transferring vanishing lines up using \cite[Theorem 14.4]{GKRW18} (with $\rho(n)=2n$), we obtain that $\smash{H_{n,d}^{E_\infty}(\bf{A})}=0$ for $d<2n-1$ as well.   
\end{proof}

\begin{lemma} \label{lem:transferring-vanishing-line-rel}
    Let $f \colon \bf{A} \to \bf{B} \in \Alg_{E_\infty^\rm{nu}}^\rm{red}(\Fun(\bb{N},\DQ))$ be a map such that 
    \begin{enumerate}[(i)]
        \item \label{enum:transferring-vanishing-line-rel-i} $H_{n,d}(\bf{A})=0$ for $d<2n-1$,
        \item \label{enum:transferring-vanishing-line-rel-ii} $H_{n,d}(\bf{B})=0$ for $d \neq 2n-1$,
        \item \label{enum:transferring-vanishing-line-rel-iii} $H_{n,2n-1}(f)$ is an isomorphism for all $n$.
    \end{enumerate}
Then $H_{n,d}^{E_\infty}(\bf{B},\bf{A})=0$ for $d \le 2n$. 
\end{lemma}

\begin{proof}
    In the relative bar spectral sequence $E^1_{n,p,q}=H_{n,q}(\bf{B}^{\otimes p}, \bf{A}^{\otimes p}) \Rightarrow H_{n,p+q-1}^{E_1}(\bf{B},\bf{A})$, by construction $E^1_{n,p,q}$ vanishes for $p \le 0$. By assumption \eqref{enum:transferring-vanishing-line-rel-i}, we have that $H_{n,q}(\bf{A}^{\otimes p})=0$ for $q<2n-p$ and that there is an isomorphism
        \[H_{n,2n-p}(\bf{A}^{\otimes p}) \cong \bigoplus_{\stackrel{n_1,\dots, n_p>0}{\text{such that} \sum n_i= n}} H_{n_1,2n_1-1}(\bf{A}) \otimes \cdots \otimes H_{n_p, 2 n_p-1}(\bf{A}).\]
    By assumption \eqref{enum:transferring-vanishing-line-rel-ii} and using that the K\"unneth theorem involves no Tor-terms when working rationally, we similarly have that $H_{n,q}(\bf{B}^{\otimes p})=0$ for $q \neq 2n-p$ and that there is an isomorphism
        \[H_{n,2n-p}(\bf{B}^{\otimes p})= \bigoplus_{\stackrel{n_1,\dots, n_p>0}{\text{such that} \sum n_i= n}} H_{n_1,2n_1-1}(\bf{B}) \otimes \cdots \otimes H_{n_p, 2 n_p-1}(\bf{B}).\]
    Then using assumption \eqref{enum:transferring-vanishing-line-rel-iii}, we see that $H_{n,q}(\bf{B}^{\otimes p}, \bf{A}^{\otimes p})=0$ for $q \le 2n-p+1$. In conclusion, we have that $\smash{E^1_{n,p,q}}=0$ for $q \le 2n-p+1$ and obtain that $\smash{H_{n,d}^{E_1}(\bf{B},\bf{A})}=0$ for $d \le 2n$. 
    
    Since by Lemma \ref{lem:transferring-vanishing-lines-abs} we have that $\smash{H_{n,d}^{E_1}(\bf{A})}=0=\smash{H_{n,d}^{E_1}(\bf{B})}$ for $d <2n-1$, we can transfer the vanishing lines up using \cite[Proposition 14.5]{GKRW18} (with $\rho(n)=2n$ and $\sigma(n)=2n+1$) to obtain that $H_{n,d}^{E_\infty}(\bf{B},\bf{A})=0$ for $d \le 2n$ as well.
    \end{proof}

    \begin{remark} \label{rem E1 homology version}
        The same statement holds for $E_1$-algebras and $E_1$-homology with the same proof, without need to transfer vanishing lines up. 
    \end{remark}

\subsubsection{Proof of \cref{thm:image-edge-is-cobracket}} We will split the proof into several steps.

\medskip

\noindent \textbf{Step 1.} Using assumption \eqref{enum:image-edge-is-cobracket-hypothesis-2} that the canonical maps are injective, we may choose maps 
    \[s_n \colon H_{n,2n-1}^{E_\infty}(\bf{R}/(\sigma)) \lra H_{n,2n-1}(\bf{R}/(\sigma))\]
    of $\bb{Q}$-vector spaces such that $s_n \circ \rm{can}_n= \id$. Consider the morphism $f$ in $\Fun(\bb{N},\DQ)$ given by 
    \[\cot_{E^\rm{nu}_\infty}(\bf{R}/(\sigma)) \to \rho_{\le \slopetwo} (\cot_{E^\rm{nu}_\infty}(\bf{R}/(\sigma))) \simeq \bigoplus_{n \ge 1} H_{n,2n-1}^{E_\infty}(\bf{R}/(\sigma)) \xrightarrow{\bigoplus_{n \ge 1} s_n} \bigoplus_{n \ge 1}H_{n,2n-1}(\bf{R}/(\sigma)), \]
    where we have used assumption \eqref{enum:image-edge-is-cobracket-hypothesis-1} and \cref{lem:transferring-vanishing-lines-abs} to deduce that $\rho_{\le \slopetwo}(\cot_{E^\rm{nu}_\infty}(\bf{R}/(\sigma)))$ lies in the heart of the $t$-structure on $\Fun(\bb{N},\DQ)$ defined by the abstract connectivity $n \mapsto 2n-1$, so is equivalent to its homology. The adjoint of $f$ under the adjunction $\cot_{E^\rm{nu}_\infty} \dashv \triv_{E^\rm{nu}_\infty}$ is the following morphism in $\Alg_{E_\infty^\rm{nu}}(\Fun(\bb{N},\DQ))$
    \[F \colon \bf{R}/(\sigma) \lra \bf{T} \coloneq \triv_{E^\rm{nu}_\infty}\big(\bigoplus_{n \ge 1} H_{n,2n-1}(\bf{R}/(\sigma))\big),\]
    given explicitly by $\triv_{E^\rm{nu}_\infty}(f) \circ \eta$ where $\eta$ is a unit for the adjunction $\cot_{E^\rm{nu}_\infty} \dashv \triv_{E^\rm{nu}_\infty}$.
 
\medskip

\noindent \textbf{Step 2.} We prove that the relative $E_\infty$-homology groups $H_{n,d}^{E_\infty}(\bf{T},\bf{R}/(\sigma))$ of $F$ vanish for $d \le 2n$. Assumption \eqref{enum:image-edge-is-cobracket-hypothesis-1} says $H_{n,d}(\bf{R}/(\sigma))=0$ for $d<2n-1$, and by definition of $\bf{T}$ we know that $H_{n,d}(\bf{T})=0$ for $d \neq 2n-1$. Thus, by \cref{lem:transferring-vanishing-line-rel} it suffices to prove that $H_{n,2n-1}(F)$ is an isomorphism for any $n \ge 1$. By definition, $H_{n,2n-1}(F)$ is the dashed composition 
     \[\begin{tikzcd}H_{n,2n-1}(\bf{R}/(\sigma)) \rar{\eta_*} \arrow[dashed]{rd}[swap]{F_*} &[25pt] H_{n,2n-1}(\triv_{E^\rm{nu}_\infty}(\cot_{E^\rm{nu}_\infty}(\bf{R}/(\sigma)))) \dar{\triv_{E^\rm{nu}_\infty}(f)_*} \\[-5pt]
     & H_{n,2n-1}(\bf{T}) \simeq H_{n,2n-1}(\bf{R}/(\sigma)).\end{tikzcd}\]
Under the canonical identification $H_{n,2n-1}(\triv_{E^\rm{nu}_\infty}(\cot_{E^\rm{nu}_\infty}(\bf{R}/(\sigma)))) \cong H_{n,2n-1}^{E_\infty}(\bf{R}/(\sigma))$ the above composition agrees with 
    \[H_{n,2n-1}(\bf{R}/(\sigma)) \xrightarrow{\rm{can}_n} H_{n,2n-1}^{E_\infty}(\bf{R}/(\sigma)) \xrightarrow{s_n} H_{n,2n-1}(\bf{R}/(\sigma)),\]
which by construction is the identity, implying the required result. 

\medskip

\noindent \textbf{Step 3.} Now we prove parts \eqref{enum:image-edge-is-cobracket-conclusion-1} and \eqref{enum:image-edge-is-cobracket-conclusion-2} of the theorem. By the previous step we know that 
    \[H_{n,2n-1}^{E_\infty}(F) \colon H_{n,2n-1}^{E_\infty}(\bf{R}/(\sigma)) \overset{\cong}\lra H_{n,2n-1}^{E_\infty}(\bf{T})\]
is an isomorphism for each $n \ge 1$.  Thus, the following map becomes an equivalence upon truncating using $\rho_{\leq 2\bullet}$ 
    \[(\indec_{E_\infty^\rm{nu}}(\bf{R}/(\sigma)))[1] \xrightarrow{\indec_{E_\infty^\rm{nu}}(F)[1]}  (\indec_{E^\rm{nu}_\infty}(\bf{T}))[1] \simeq \cofree_{\rm{coLie}}\Big(\bigoplus_{n \ge 1} H_{n,2n-1}(\bf{R}/(\sigma))[1]\Big).\] 
    This identifies the Goncharov Lie coalgebra $\scr{G}(\bf{R})$ as a cofree Lie coalgebra. The identification of the maps $\rm{can}_n \colon H_{n,2n-1}(\bf{R}/(\sigma)) \to \smash{H_{n,2n-1}^{E_\infty}}(\bf{R}/(\sigma))$ with inclusions is a consequence of the construction. This completes the proof of part \eqref{enum:image-edge-is-cobracket-conclusion-1} and part \eqref{enum:image-edge-is-cobracket-conclusion-2} follows from the fact that the kernel of the cobracket on a cofree Lie coalgebra is equal to the inclusion of the cogenerators.
    
\medskip

\noindent  \textbf{Step 4.} Now we can prove part \eqref{enum:image-edge-is-cobracket-conclusion-3} of the theorem. The fact that $\rm{im}(\rm{edge}_n) \subseteq \ker(\delta|_{\scr{G}_n(\bf{R})})$ follows from assumption \eqref{enum:image-edge-is-cobracket-hypothesis-4} by \cref{cor:image-of-edge-in-ker-cobracket}. We want to prove $\ker(\delta|_{\scr{G}_n(\bf{R})}) \subset \rm{im}(\rm{edge}_n)$ and to do so, consider the map of rank spectral sequences from \cref{lem:rank-spectral-sequences}
    \begin{equation}\label{eqn:rank-ss}\begin{tikzcd}
        E^1_{n,d}=H_{n,d}(\bf{R}/(\sigma)) \Longrightarrow H_d(\bf{R}/(\sigma-1)) \dar \\[-10pt]
        {}^\infty E^1_{n,d}= H^{E_\infty}_{n,d}(\bf{R}/(\sigma)) \Longrightarrow H_d^{E_\infty}(\bf{R}/(\sigma-1)).
    \end{tikzcd}\end{equation}
Part \eqref{enum:image-edge-is-cobracket-conclusion-2} says that $\ker(\delta|_{\scr{G}_n(\bf{R})})$ is the image of the canonical map $\rm{can}_n \colon H_{n,2n-1}(\bf{R}/(\sigma)) \to \smash{H_{n,2n-1}^{E_\infty}(\bf{R}/(\sigma))}$, so we have an identification
    \[\rm{im}\big[E^1_{n,2n-1} \to {}^\infty E^1_{n,2n-1}\big]= \ker(\delta|_{\scr{G}_n(\bf{R})}) \subseteq {}^\infty E^1_{n,2n-1} \cong \scr{G}_n(\bf{R}).\] 
    Moreover, by the vanishing of assumption \eqref{enum:image-edge-is-cobracket-hypothesis-1} the stabilisation map $H_{n,2n-1}(\bf{R}) \xrightarrow{\simeq} H_{n',2n-1}(\bf{R})$ is surjective for any $n' \ge n$ and hence the stabilisation map $H_{n,2n-1}(\bf{R}) \to H_{2n-1}(\bf{R}/(\sigma-1))$ is surjective. Assumption \eqref{enum:image-edge-is-cobracket-hypothesis-3} says the right map is surjective in the exact sequence
    \[H_{n-1,2n-1}(\bf{R}) \overset{\sigma}\lra H_{n,2n-1}(\bf{R}) \lra H_{n,2n-1}(\bf{R}/(\sigma)) \lra 0,\]
    so that $\gr^\rm{rk}_n H_{2n-1}(\bf{R}/(\sigma-1)) \cong H_{n,2n-1}(\bf{R}/(\sigma))$. It follows that the edge homomorphism $\rm{edge}_n \colon H_{2n-1}(\bf{R}/(\sigma-1)) \to H_{n,2n-1}(\bf{R}/(\sigma))$ of the first spectral sequence in \eqref{eqn:rank-ss} is surjective. Thus, $\ker(\delta|_{\scr{G}_n(\bf{R})}) \subseteq \rm{im}(\rm{edge}_n)$ follows by commutativity of the diagram
     \[\begin{tikzcd}
        H_{2n-1}(\bf{R}/(\sigma-1)) \rar[two heads] \dar & E^1_{n,2n-1} = H_{n,2n-1}(\bf{R}/(\sigma)) \dar \\[-5pt] 
        H^{E_\infty}_{2n-1}(\bf{R}/(\sigma-1)) \rar{\rm{edge}_n} & {}^\infty E^1_{n,2n-1} = H_{n,2n-1}^{E_\infty}(\bf{R}/(\sigma)) \cong \scr{G}_n(\bf{R}) 
     \end{tikzcd}\]
    from the naturality of the edge homomorphism.

\subsection{Applications to number fields} 

The first result of this subsection is: 

\begin{theorem} \label{thm:number-fields-goncharov}
    If $F$ is a number field then $\bf{R}=\BGLb(F)_\bb{Q}$ satisfies the assumptions \eqref{enum:image-edge-is-cobracket-hypothesis-1}--\eqref{enum:image-edge-is-cobracket-hypothesis-4} of \cref{thm:image-edge-is-cobracket}. Thus there exists an isomorphism of Lie coalgebras
    \[\PolyL_\bullet(F) \cong \rm{cofree}_\rm{coLie}\Big(\bigoplus_{n \ge 1} K_{2n-1}(F)_\bb{Q}[1]\Big)\]
    with the property that the edge homomorphisms $K_{2n-1}(F)_\bb{Q} \to \PolyL_n(F)$ correspond to the inclusion of the cogenerators. 
\end{theorem}

\begin{proof}Using \cref{lem:rational-algebras-from-spaces}, $\BGLb(F)_\bb{Q}$ lifts to $(\bb{Q}^\rm{nu}[t] \to \BGLb(F)_\bb{Q} \to \bb{Q}^\rm{nu}[t])$ and we verify this satisfies the assumptions of \cref{thm:image-edge-is-cobracket}, listed below for the convenience of the reader:
    \begin{enumerate}[(i)]
        \item that $H_{n,d}(\BGLb(F)_\bb{Q}/(\sigma))=0$ for $d<2n-1$ is verified in \cref{prop:borel-yang-stab} \eqref{enum:borel-yang-stab-ii},
        \item that the canonical maps $\rm{can}_n \colon H_{n,2n-1}(\BGLb(F)_\bb{Q}/(\sigma)) \to H_{n,2n-1}^{E_\infty}(\BGLb(F)_\bb{Q}/(\sigma))$ are injective for all $n \ge 1$ will be verified in \cref{prop:injectivity-number-fields},
        \item that the maps $H_{n,2n-1}(\BGLb(F)_\bb{Q}) \to H_{n,2n-1}(\BGLb(F)_\Q/(\sigma))$ are surjective for all $n \ge 1$ is verified in \cref{prop:borel-yang-stab} \eqref{enum:borel-yang-stab-i},
        \item that the stable algebra $(\BGLb(F)_\bb{Q})/(\sigma-1) \in \Alg_{E_\infty}(\DQ)$ is free follows from \cref{lem:rational-algebras-from-spaces-free}.
    \end{enumerate}

At this point we know that $\scr{G}(F)$ is a cofree Lie coalgebra on the relative homology groups $H_{n,2n-1}(\BGLb(F)/(\sigma))$. Using \cref{lem:borel-yang-prim-rel-iso} and \cref{lem:borel-yang-regulators}, however, both natural maps
\[H_{n,2n-1}(\BGLb(F)_\bb{Q}/(\sigma)) \overset{\cong}\longleftarrow \rm{prim}\, H_{2n-1}(\GL_n(F)) \overset{\cong}\lra K_{2n-1}(F)_\bb{Q}\]
are isomorphisms. That we can identify the canonical map and edge homomorphism 
\begin{align*}\rm{can}_n \colon H_{n,2n-1}(\BGLb(F)_\bb{Q}/(\sigma)) &\lra H_{n,2n-1}^{E_\infty}(\BGLb(F)_\bb{Q}/(\sigma)) \\
\rm{edge}_n \colon K_{2n-1}(F)_\bb{Q} &\lra H_{n,2n-1}^{E_\infty}(\BGLb(F)_\bb{Q}/(\sigma))\end{align*} 
using the isomorphism $H_{n,2n-1}(\BGLb(F)_\bb{Q}/(\sigma)) \cong K_{2n-1}(F)_\bb{Q}$ now follows from \cref{lem:edge-homomorphism-in-rank-n}.
\end{proof}

\subsection{Verifying the injectivity condition for number fields.}\label{subsec:number-fields-injectivity}

In the previous subsection we applied Theorem \ref{thm:image-edge-is-cobracket} to the $E_\infty$-algebra $\bf{R}= \BGLb(F)_\bb{Q}$ for $F$ a number field, up to the verification of assumption \eqref{enum:image-edge-is-cobracket-hypothesis-2}. This will use the explicit construction of the Borel regulator. This is done in this subsection, the main result of which is as follows: 

\begin{proposition} \label{prop:injectivity-number-fields}
For a number field $F$ the canonical maps 
\[\rm{can}_n \colon H_{n,2n-1}(\BGLb(F)_\bb{Q}/(\sigma)) \lra H_{n,2n-1}^{E_\infty}(\BGLb(F)_\bb{Q}/(\sigma))\]
are injective for any $n \ge 1$.     
\end{proposition}

Using the universal coefficients theorem, we may replace $\bb{Q}$ by $\bb{R}$. Until we say otherwise, we work 1-categorically in this subsection. Let us now outline the strategy to prove it. We first construct an augmented strictly associative algebra model $C_*(\BGLb(F);\bb{R})$ for $\BGLb(F)_\bb{R}$ in $\Fun(\bb{N},\rm{Ch}_\bb{R})$. We then construct an augmented strictly associative algebra $\bf{C}_*(\mathfrak{gl},\mathfrak{u})$ in $\Fun(\bb{N},\rm{Ch}_\bb{R})$ by taking relative Lie algebra chains and a map of augmented strictly associative algebras $\smallint_{\Delta(-)} \colon C_*(\BGLb(F);\bb{R}) \to \bf{C}_*(\mathfrak{gl},\mathfrak{u})$. We combine the map induced by $\smallint_{\Delta(-)}$ on $E_1$-homology with the splitting result of Barr \cite[Proposition C.16]{KRS1} to reduce \cref{prop:injectivity-number-fields} to a computation for $\bf{C}_*(\mathfrak{gl},\mathfrak{u})$.

\subsubsection{Constructing $C_*(\BGLb(F);\bb{R})$}
Recall the functors
\[\rm{Grp} \xrightarrow{B_\bullet} \rm{sSet} \xrightarrow{\bb{R}[-]} \rm{sVect}_\bb{R} \overset{C}\lra \rm{Ch}_\bb{R}\]
given by (i) the bar construction, (ii) the levelwise free vector space $\bb{R}[X]$ on the simplices of a simplicial set $X$, and (iii) the unnormalised chain complex $C(M)$ of a simplicial vector space $M$: the composition sends a simplicial set to its chains with $\bb{R}$-coefficients. The first is symmetric monoidal with respect to cartesian products. The second symmetric monoidal with respect to cartesian product on $\rm{sSet}$ and levelwise tensor product on $\rm{sVect}_\bb{R}$. The third is lax symmetric monoidal with respect to the levelwise tensor product on $\rm{sVect}_\bb{R}$ and tensor product of chain complexes on $\rm{Ch}_{\bb{R}}$, with lax symmetric monoidality given by
\[C(X_p) \otimes C(Y_q) \ni x_p \otimes y_q \longmapsto \sum_{(\mu,\nu) \in \rm{Sh}(p,q)} \rm{sgn}(\mu,\nu) (s_\nu(x_p) \otimes s_\mu(y_p)) \in C(X \otimes Y)_{p+q}\]
where the sum is over the set of $(p,q)$-shuffles (permutations of $\{0,\ldots,p+q-1\}$ so that the restriction $\mu = (\mu_1,\ldots,\mu_p)$ to the first $p$ elements and the restriction $\nu = (\nu_1,\ldots,\nu_q)$ to the last $q$ elements preserve the standard order), $\rm{sgn}(\mu,\nu)$ is the sign of permutations, and we set $s_\mu = s_{\mu_p} \circ \cdots \circ s_{\mu_1}$ and $s_\nu = s_{\nu_q} \circ \cdots \circ s_{\nu_1}$. Finally, using Day convolution this induces a symmetric lax monoidality on 
\[C_*(-;\bb{R}) \colon \Fun(\bb{N},\rm{sSet}) \lra \Fun(\bb{N},\rm{Ch}_\bb{R}).\]

Applied to the associative algebra object $\rm{GL}^\rm{perm}(F)$ in $\rm{Grp}$ given by $\rm{GL}_n(F)$ under block sum, we obtained an associative algebra 
\[C_*(\BGLb(F);\bb{R}) \coloneqq C_*(B_\bullet \rm{GL}^\rm{perm}(F);\bb{R}) \in \Alg^\rm{aug}_{\rm{Ass}^\rm{u}}(\Fun(\bb{N},\rm{Ch}_\bb{R})),\] 
which admits a unique augmentation since it is equal to the singular chains on a point in rank $0$. Using the comparison results of \cite[Appendix B.4]{KRS1}, this models the underlying augmented $E^\rm{u}_1$-algebra of $\BGLb(F)_\bb{R} \simeq \BGLb(F)_\bb{Q} \otimes_\bb{Q} \bb{R}$:

\begin{lemma}There is an equivalence $C_*(\BGLb(F);\bb{R}) \simeq \BGLb(F)_\bb{R} \in \Alg^\rm{aug}_{E^\rm{u}_1}(\Fun(\bb{N},\scr{D}_{\bb{R}}))$.
\end{lemma}

\subsubsection{Constructing $\bf{C}_*(\mathfrak{gl},\mathfrak{u})$} Let us recall the definition of relative Lie algebra (co)chains, see e.g.~ \cite[Chapitre II]{Guichardet}, \cite[Chapter I]{BorelWallach}, \cite[Chapter 5]{BurgosGil}.

\begin{definition}\label{def:relative-lie-algebra-homology-cohomology} Let $\mathfrak{g}$ be a Lie algebra over $\bb{R}$ and let $\mathfrak{h} \subset \mathfrak{g}$ be a Lie subalgebra. 
\begin{enumerate}[(i)]
\item The relative Lie algebra chains $C_*(\mathfrak{g},\mathfrak{h})$ are defined as
\[
   C_*(\mathfrak{g},\mathfrak{h}) \coloneq \frac{\Lambda^*(\mathfrak{g}/\mathfrak{h})}{\Big \langle \sum_{i=1}^p \overline{X}_1 \wedge \cdots \wedge \overline{[X,X_i]} \wedge \cdots \wedge \overline{X}_p) \mid \forall X \in \mathfrak{h}\text{, }\forall \overline{X}_1, \dots, \overline{X}_p \in \mathfrak{g}/\mathfrak{h} \Big \rangle}, 
\]
with differential given by
\[d(\overline{X}_0 \wedge \cdots \wedge \overline{X}_p) = \sum_{i<j} (-1)^{i+j} \ol{[X_i,X_j]} \wedge \overline{X}_0 \wedge \cdots \wedge \widehat{\overline{X}_i} \wedge \cdots \wedge\widehat{\overline{X}_j} \wedge \cdots \wedge X_p.\]
\item The relative Lie algebra cochains $C^*(\mathfrak{g},\mathfrak{h})$ are defined as
\[
   C^*(\mathfrak{g},\mathfrak{h}) \coloneq \left\{\varphi \in \rm{Hom}(\Lambda^*(\mathfrak{g}/\mathfrak{h}),\R) \,\middle|\,  \parbox{6.5cm}{\centering $\sum_{i=1}^p \varphi(\overline{X}_1,\dots,\overline{[X,X_i]},\dots,\overline{X}_p)=0$ $\forall X \in \mathfrak{h}$, $\forall \overline{X}_1, \dots, \overline{X}_p \in \mathfrak{g}/\mathfrak{h} $}\right\}, 
\]
with differential given by
\[d\varphi(\overline{X}_0,\ldots,\overline{X}_p) = \sum_{i<j} (-1)^{i+j} \varphi(\ol{[X_i,X_j]},\overline{X}_0,\ldots,\widehat{\overline{X}_i},\ldots,\widehat{\overline{X}_j},\ldots,X_p).\]
\end{enumerate}
\end{definition}

By construction we have maps
\[C_*(\fr{g}) \twoheadrightarrow C_*(\fr{g},\fr{h}), \qquad C^*(\mathfrak{g},\mathfrak{h}) \hookrightarrow C^*(\mathfrak{g}), \qquad C^*(\fr{g},\fr{h}) \cong C_*(\fr{g},\fr{h})^\vee,\]
where on the left and middle we have the usual Chevalley--Eilenberg chains and cochains. If $\fr{g}$ is finite-dimensional, then the double-dual map $C_*(\fr{g},\fr{h}) \overset{\cong}\lra C^*(\fr{g},\fr{h})^\vee$ is an isomorphism and we may think of relative Lie algebra chains as the dual to relative Lie algebra cochains.

Let $\fr{gl}_n \coloneq \rm{Lie}(\GL_n(\bb{C}))$ be the Lie algebra of the Lie group $\GL_n(\bb{C})$ and $\fr{u}_n \coloneq \rm{Lie}(\rm{U}(n))$ be the Lie subalgebra of its maximal compact Lie subgroup $\rm{U}(n)$.

\begin{lemma} \label{lem:relative-lie-algebra-computation-gl-u}\,
    \begin{enumerate}[(i)]
    \item \label{enum:relative-lie-algebra-computation-gl-u-i} $C^*(\fr{gl}_n,\fr{u}_n)$ has trivial differential so is isomorphic to its cohomology.
    \item \label{enum:relative-lie-algebra-computation-gl-u-ii} There is an isomorphism $H^*(\fr{gl}_n,\fr{u}_n) \cong H^*(\rm{U}(n);\bb{R}) \cong \Lambda_\R^* \langle \ol{c}_{n,1}, \ol{c}_{n,3},\ldots,\ol{c}_{n,2n-1} \rangle$ with $\ol{c}_{n,2i-1}$ primitive of degree $2i-1$ and transgression given by the $i$th Chern class.
    \end{enumerate}
\end{lemma}

\begin{proof}
    For part \eqref{enum:relative-lie-algebra-computation-gl-u-i}, we use the standard basis to identify $\fr{gl}_n$ with the Lie algebra $\fr{mat}_n$ of $(n \times n)$-matrices with complex entries and Lie bracket given by the commutator, $[A,B] \coloneq AB - BA$ (we view $\fr{gl}_n$ as a real Lie algebra here). Under this identification $\fr{u}_n \subset \mathfrak{gl}_n$ is given by the Lie subalgebra of skew-Hermitian matrices. Let $\fr{p}_n \subset \fr{gl}_n$ denote the real vector subspace of Hermitian matrices, so that there is a splitting of real vector spaces
     \[\fr{gl}_n = \fr{u}_n \oplus \fr{p}_n\]
    inducing an isomorphism $\fr{p}_n \cong \fr{gl}_n/\fr{u}_n$ of real vector spaces. Let $\Theta$ be the Cartan involution on $\fr{gl}_n$ given by $\Theta(A) \coloneq - A^\dagger$. This preserves the Lie bracket by inspection and acts by $+1$ on $\fr{u}_n$ and by $-1$ on $\fr{p}_n$. Thus, $\Theta$ defines an involution on the complex $C^*(\fr{gl}_n,\fr{u}_n)$ and acts by $(-1)^p$ on $C^p(\fr{gl}_n,\fr{u}_n)$, implying the required triviality of differentials. 

    For part \eqref{enum:relative-lie-algebra-computation-gl-u-ii}, by the above involution we have that $[\fr{u}_n,\fr{p}_n] \subset \fr{p}_n$ and obtain an isomorphism of cochain complexes
    \[
    C^*(\fr{gl}_n,\fr{u}_n) \cong \left\{\varphi \in \rm{Hom}(\Lambda^*\fr{p}_n,\R) \,\middle|\,  \parbox{6.5cm}{\centering $\sum_{i=1}^p \varphi(X_1,\dots,[X,X_i],\dots,X_p)=0$ $\forall X \in \fr{u}_n$, $\forall X_1, \dots, X_p \in \fr{p}_n $}\right\}, 
    \]
    Multiplication by $i$ induces an isomorphism $i \colon \fr{u}_n \xrightarrow{\cong} \fr{p}_n$ of real vector spaces, which in turn induces an isomorphism $i^* \colon \rm{Hom}(\Lambda^*\fr{p}_n,\R) \to \rm{Hom}(\Lambda^*\fr{u}_n,\R)$ of graded vector spaces, and from this we obtain an isomorphism of cochain complexes
     \[
    C^*(\fr{gl}_n,\fr{u}_n) \cong \left\{\varphi \in \rm{Hom}(\Lambda^*\fr{u}_n,\R) \,\middle|\,  \parbox{6.5cm}{\centering $\sum_{i=1}^p\varphi(X_1,\dots,[X,X_i],\dots,X_p)=0$ $\forall X,X_1, \dots, X_p \in \fr{u}_n $}\right\}.\]
    The right is a subcomplex of the usual Chevalley--Eilenberg cochain $C^*_\rm{CE}(\mathfrak{u}_n)$ and under the identification of the latter with left invariant forms on $\rm{U}(n)$, this is subcomplex of bi-invariant forms and its inclusion is a quasi-isomorphism \cite[Theorem 12.1]{chevalley1948cohomology}. This implies the result by our previous discussion about the relationship between de Rham cohomology and Lie algebra cohomology, and well-known computations of the cohomology of unitary groups as an exterior algebra on its primitives, whose transgressions are given by the Chern classes.
\end{proof}

Let $\rm{LieAlg}^2_\bb{R}$ denote the 1-category whose objects are pairs $\fr{h} \subseteq \fr{g}$ of Lie algebras over $\bb{R}$ and whose morphisms are maps of Lie algebras $\fr{g} \to \fr{g}'$ sending $\fr{h}$ into $\fr{h}'$. Relative Lie algebra chains are natural in pairs $\fr{h} \subseteq \fr{g}$ and lift to a functor $C_*(-,-) \colon \rm{LieAlg}^2_\bb{R} \to \rm{Ch}_\bb{R}$ of 1-categories. The domain has a symmetric monoidal structure given by direct sum of Lie algebras and the target has a symmetric monoidal structure given by tensor product of chain complexes; we observe that the map $\Lambda^*(\fr{g}_1/\fr{h}_1) \otimes \Lambda^*(\fr{g}_2/\fr{h}_2) \to \Lambda^*((\fr{g}_1 \oplus \fr{g}_2)/(\fr{h}_1 \oplus \fr{h}_2))$ induced by inclusions yields an lax symmetric monoidality on $C_*(-,-)$; this computation is left to the reader. 

Block sum makes the functor $n \mapsto (\fr{u}_n \subset \fr{gl}_n)$ into a strictly associative unital algebra in $\Fun(\bb{N},\rm{LieAlg}^2_\bb{R})$, and applying $C_*(-,-)$ we get 
\[\bf{C}_*(\fr{gl},\fr{u}) \coloneq \big(n \mapsto C_*(\fr{gl}_n,\fr{u}_n)\big) \in \rm{Alg}^\rm{aug}_{\rm{Ass}^\rm{u}}(\Fun(\bb{N},\rm{Ch}_\bb{R}))\]
which admits a unique augmentation since it is equal to $\bb{R}$ in rank 0. Let us investigate the multiplicative structure: we denote by 
\[\sigma \in H_{1,0}(\bf{C}_*(\fr{gl},\fr{u})) \cong \R \qquad \text{and} \qquad  \ol{\sf{c}}_{n,2i-1} \in H_{n,2i-1}(\bf{C}_*(\fr{gl},\fr{u}))\]
the canonical generator corresponding to $1 \in \Lambda_\R^0 \fr{gl}_1/\fr{u}_1$, and the unique primitive class which pairs to one against $\ol{c}_{n,2i-1}$. 

\begin{lemma}\label{lem:cts-cohomology-maps} \,
\begin{enumerate}[(i)]
    \item \label{enum:cts-cohomology-maps-i} $\bf{C}_*(\fr{gl},\fr{u})$ lifts uniquely to a strictly commutative algebra.
    \item \label{enum:cts-cohomology-maps-ii} We have $\sigma \ol{\sf{c}}_{n-1,2j-1} = \ol{\sf{c}}_{n,2j-1}$.
\end{enumerate}
\end{lemma}

\begin{proof}The comparison produced in the proof of \cref{lem:relative-lie-algebra-computation-gl-u} is compatible with block sums. Using this, \eqref{enum:cts-cohomology-maps-i} follows since the inclusion $\rm{U}(n) \times \rm{U}(m) \to \rm{U}(n+m)$ is homotopic to the inclusion $\rm{U}(m) \times \rm{U}(n) \to \rm{U}(n+m)$, and \eqref{enum:cts-cohomology-maps-ii} follows since the Chern classes are stable, transgression is natural and stabilization sends primitives to primitives.
\end{proof}

In particular, it makes sense to define the $E_\infty^\rm{u}$-algebra $\bf{C}_*(\fr{gl},\fr{u})/(\sigma)$, whose homology agrees with the relative homology of the map given by multiplication by $\sigma$ and is explicitly given by:

\begin{lemma}\label{lem:relative-lie-algebra-mod-sigma} We have
    \begin{enumerate}[(i)]
        \item $H_{n,d}(\bf{C}_*(\fr{gl},\fr{u})/(\sigma)) = 0$ for $d<2n-1$. 
        \item $H_{n,2n-1}(\bf{C}_*(\fr{gl},\fr{u})/(\sigma)) = \R \langle \ol{\sf{c}}_{n,2n-1} \rangle$ for each $n \ge 1$. 
    \end{enumerate}
\end{lemma}

\subsubsection{Constructing $\smallint_{\Delta(-)}$}
Suppose that $G$ is a Lie group and $H \subset G$ is a subgroup. The quotient $G/H$ is a pointed smooth manifold with a smooth left $G$-action such that the quotient map $G \to G/H$ induces an isomorphism $T_{eH}(G/H) \cong \mathfrak{g}/\mathfrak{h}$ where $\mathfrak{g}=\rm{Lie}(G)$ and $\mathfrak{h}=\rm{Lie}(H)$. Moreover, evaluation at the basepoint gives an isomorphism of cochain complexes 
\[\Omega^*(G/H)^G \xrightarrow{\rm{ev}_{eH}} C^*(\fr{g},\fr{h})\]
whose domain is the left $G$-invariant de Rham forms on $G/H$ and whose target is relative Lie algebra cohomology as in \cref{def:relative-lie-algebra-homology-cohomology}. This map has an inverse given by extending the given form to a left $G$-invariant one on $G/H$:
\begin{align*}C^*(\fr{g},\fr{h}) &\lra \Omega^*(G/H)^G \\
\varphi &\longmapsto \widetilde{\varphi} \coloneq \big(gH \mapsto (g^{-1} \cdot -)^* (\varphi) \in \rm{Hom}(\Lambda^* T_{gH}(G/H),\bb{R})\big)\end{align*} 
A left $G$-invariant Riemannian metric on $G/H$ is determined by an appropriate choice of inner product at the basepoint and then the left action of $G$ is by isometries. When $H$ is maximal compact, this gives a way to produce a map of chain complexes by integrating over geodesic simplices \cite[\S 1]{DUPONT1976233}. For a generator $[g_1|\dots|g_p] \in C_p(B_\bullet G;\bb{R})$ we let $\Delta(g_1,\compactldots,g_p) \subset G/H$ denote the geodesic simplex with vertices $(eH, g_1 H, g_1 \cdot g_2 H, \dots, g_1 \cdots g_p H)$. This is defined inductively: $\Delta(g_1)$ is the geodesic arc from $eH$ to $g_1H$ and $\Delta(g_1,\ldots,g_p)$ is the geodesic cone on $g_1 \cdot \Delta(g_2,\ldots,g_p)$ with top point $eH$. A geodesic simplex always exists and is uniquely determined by its vertices as $G/H$ is diffeomorphic to a Euclidean space.

\begin{lemma} \label{lem:integration-chain-map}
 Integration defines a chain map
\begin{align*}
    {\textstyle \int}_{\Delta(-)} \colon  C_*(B_\bullet G;\bb{R}) &\lra C^*(\fr{g}, \fr{h})^\vee \\
    [g_1|\compactldots|g_p] &\longmapsto \big( \varphi \mapsto {\textstyle \int}_{\Delta(g_1,\dots,g_p)} \widetilde{\varphi} \big).
\end{align*}
\end{lemma}

\begin{proof}We need to verify compatibility with differentials, and to do so it is convenient to allow formal $\bb{R}$-linear combinations of geodesic simplices and their translates by $G$, $\bb{R}$-linearly extending $\Delta$. Doing so, we compute:
\begin{align*}
  \hspace{-.2cm} \Delta\big(d([g_1|\compactldots|g_p])) &= \Delta( [g_2|\compactldots|g_p]+ {\textstyle \sum}_{i=1}^{p-1} (-1)^i [g_1|\compactldots|g_ig_{i+1}|\compactldots|g_p] + (-1)^p [g_1|\compactldots|g_{p-1}]\big) \\
  &= -g_1 \cdot \Delta(g_2,\dots,g_p)+ \Delta(g_2,\dots,g_p)+ \partial \Delta(g_1,\dots,g_p)
\end{align*}
and hence for any $\varphi \in C^*(\fr{g}, \fr{h})$ we have
\begin{align*}
    \int_{\Delta(d([g_1|\compactldots|g_p]))} \widetilde{\varphi} &= -\int_{g_1 \cdot \Delta(g_2,\dots,g_p)} \widetilde{\varphi} + \int_{\Delta(g_2,\dots,g_p)} \widetilde{\varphi} + \int_{\partial \Delta(g_1,\dots,g_p)} \varphi \\
    &=- \int_{\Delta(g_2,\dots,g_p)} g_1^* \widetilde{\varphi} + \int_{\Delta(g_2,\dots,g_p)} \widetilde{\varphi} + \int_{\Delta(g_1,\dots,g_p)} d \varphi = \int_{\Delta([g_1|\dots|g_p])} d \varphi
\end{align*}
by Stokes' theorem plus the left $G$-invariance of $\widetilde{\varphi}$. Thus, we indeed get a chain map.  
\end{proof}

We take $G = \rm{GL}_n(\bb{C})$ and $H = \rm{U}(n)$ and prove that taking \textit{Frobenius inner product} 
\begin{align*}\lambda \colon \mathfrak{gl}_n(\bb{C}) \otimes \mathfrak{gl}_n(\bb{C}) &\lra \bb{R} \\
 (X,Y) &\longmapsto \rm{Re}(\rm{tr}(X^\dagger Y)),\end{align*}
to produce a left $\rm{GL}_n(\bb{C})$-invariant Riemannian metric on the homogeneous space $M_n \coloneq \rm{GL}_n(\bb{C})/\rm{U}(n)$, the integration maps assemble to a map of augmented strictly associative algebras. Namely, this choice is such that the inclusion $M_n \times M_m \subset M_{n+m}$ induced by the inclusion $\rm{GL}_n(\bb{C}) \times \rm{GL}_m(\bb{C}) \subset \rm{GL}_{n+m}(\bb{C})$ by block sum, is an isometric embedding.

\begin{theorem} \label{thm:integration-assocative-algebra}
    The integration maps of \cref{lem:integration-chain-map} defined using the Frobenius inner products, assemble to a map of augmented strictly associative algebras
    \[C_*(\BGLb(\bb{C});\bb{R}) \xrightarrow{\int_{\Delta(-)}} \bf{C}^*(\fr{gl},\fr{u})^\vee \xleftarrow{\cong} \bf{C}_*(\fr{gl},\fr{u}).\]
\end{theorem}

\begin{proof}
    We know that $\smash{\int_{\Delta(-)}}$ is a chain map so it remains to check its compatibility with the binary products, i.e.~that for $n, m \ge 1$ the diagram
    \[\begin{tikzcd}
        C_p(B_\bullet \rm{GL}_n(\bb{C});\bb{R}) \otimes C_q(B_\bullet \rm{GL}_m(\bb{C});\bb{R}) \rar{\oplus} \dar[swap]{\int_{\Delta(-)} \otimes \int_{\Delta(-)}} & C_{p+q}(B_\bullet \rm{GL}_{n+m}(\bb{C});\bb{R}) \dar{\int_{\Delta(-)}} \\[-5pt]
        C^p(\mathfrak{gl}_n,\mathfrak{u}_n)^\vee \otimes  C^q(\mathfrak{gl}_m,\mathfrak{u}_m)^\vee \rar{\oplus} & C^{p+q}(\mathfrak{gl}_{n+m},\mathfrak{u}_{n+m})^\vee.
    \end{tikzcd}\]
    commutes, where the horizontal maps are induced by block sum. 
    
    We first compute the bottom-left composition on a chain $[g_1|\compactldots|g_p] \otimes [h_1|\compactcdots|h_q]$. Tracing through the definitions, its value on $\varphi \in C^{p+q}(\mathfrak{gl}_{n+m},\mathfrak{u}_{n+m},\R)$ is given by restricting the inputs of $\varphi$ to obtain a sum of elements $\sum_i \phi_i \otimes \psi_i \in C^p(\mathfrak{gl}_n,\mathfrak{u}_n) \otimes C^q(\mathfrak{gl}_m,\mathfrak{u}_m)$ and then taking 
    \[\sum_i \Big( \int_{\Delta(g_1,\compactldots,g_p)} \widetilde{\phi_i} \Big)  \Big( \int_{\Delta(h_1,\compactldots,h_q)} \widetilde{\psi_i} \Big) \in \bb{R}.\]
    
    We next compute the right-top composition on the same chain $[g_1|\compactldots|g_p] \otimes [h_1|\compactcdots|h_q]$. By the explicit formula for the Eilenberg--Zilber map, we have 
    \[\oplus\big([g_1|\compactldots|g_p] \otimes [h_1|\compactcdots|h_q]) = \sum_{(\mu,\nu) \in \rm{Sh}(p,q)} \rm{sgn}(\mu,\nu) \big(s_{\nu}([g_1|\compactldots|g_p]) \oplus s_{\mu}([h_1|\compactldots|h_q])\big)\]
    and when we apply $\Delta(-)$ to the right side we get a sum of $p+q$-simplices in $M_{n} \times M_m \subset M_{n+m}$ whose union is precisely the $(p+q)$-dimensional prism $\Delta(g_1,\dots,g_p) \times \Delta(h_1,\dots,h_q) \subset M_n \times M_m$ with correct orientation. Thus, given $\varphi \in C^{p+q}(\mathfrak{gl}_{n+m},\mathfrak{u}_{n+m})$ we first restrict it to $\varphi_1 \otimes \varphi_2 \in C^p(\mathfrak{gl}_n(\C),\mathfrak{u}_n) \otimes C^q(\mathfrak{gl}_m(\C),\mathfrak{u}_m)$ and then we evaluate it to 
    \[\sum_i \int_{\Delta(g_1,\dots,g_p) \times \Delta(h_1,\dots,h_q)} \pr_1^*(\widetilde{\phi_i}) \cup \pr_2^*(\widetilde{\psi_i})\]
    for the projections $\pr_1 \colon \Delta(g_1,\dots,g_p) \times \Delta(h_1,\dots,h_q) \to  \Delta(g_1,\dots,g_p)$ and $\pr_2 \colon \Delta(g_1,\dots,g_p) \times \Delta(h_1,\dots,h_q) \to \Delta(h_1,\dots,h_q)$, which decomposes as 
    \[\sum_i \Big(\int_{\Delta(g_1,\dots,g_p) } \widetilde{\phi_i} \Big) \Big( \int_{\Delta(h_1,\dots,h_q)} \widetilde{\psi_i} \Big) \in \bb{R}\]
    by Fubini's theorem. Thus, both compositions agree, as required.\end{proof}

\begin{remark}We expect that ${\textstyle \int}_{\Delta(-)}$ can be upgraded to a map of algebras over the Barratt--Eccles operad. This would simplify some later arguments, but is not required for the proof.\end{remark}

\subsubsection{Proof of  \cref{prop:injectivity-number-fields}} To complete the proof of \cref{prop:injectivity-number-fields} we connect the integration map $\smallint_{\Delta(-)}$ with the Borel class. This is the evaluation map
\[- \cap \rm{Bo}_{n,2n-1} \colon H_*(\BGL_n(\bb{C});\bb{R}) \lra \bb{R}\] 
given by taking the cap product against the cohomology class $\rm{Bo}_{n,2n-1} \in H^{2n-1}(\BGL_n(\bb{C});\bb{R})$ that is the image of the continuous cohomology class $\ol{c}_{n,2n-1} \in H_\rm{cts}^{2n-1}(\rm{GL}_n(\bb{C});\bb{R})$ that corresponds to the $n$th Chern class under the van Est isomorphism \cite{vanEst}, cf.~\cref{lem:relative-lie-algebra-computation-gl-u}. By \cite[Proposition 1.5]{DUPONT1976233} the inverse of the van Est isomorphism is the map
\[H^*(\mathfrak{gl}_n,\mathfrak{u}_n) \overset{\simeq}\lra H^*_\rm{cts}(\rm{GL}_n(\C);\bb{R})\]
induced by the integration map $\int_{\Delta(-)}$ of \cref{lem:integration-chain-map}. Hence we have \cite[Theorem 1.1]{DUPONT1976233}:

\begin{lemma} \label{lem integration map and borel regulator}
    Fixing a choice of a representative of $\ol{c}_{n,2n-1}$, the Borel class $\rm{Bo}_{n,2n-1}$ is represented (up to a nonzero constant) by the cocycle given by the composition 
    \[C_{2n-1}(\BGL_n(\C);\R) \xrightarrow{\smallint_{\Delta(-)}} C_{2n-1}(\fr{gl}_n,\fr{u}_n) \xrightarrow{\rm{ev}_{\ol{c}_{n,2n-1}}} \R.\]
\end{lemma}

A homologous construction using linear simplices rather than geodesic simplices appears in \cite[Section 1.2]{Igusa}. A representative of $\ol{c}_{n,2n-1}$ can be obtained using Chern--Weil theory, leading to the formulas in \cite{Hamida} or \cite[Section 1.2]{Igusa}. We now prove that the canonical maps for $\BGLb(F)_\bb{Q}$ are injective in bidegrees $(n,2n-1)$:

\begin{proof}[Proof of \cref{prop:injectivity-number-fields}] The case $n=1$ is trivial as the canonical map $\smash{H_{1,1}(\BGLb(F)_\bb{Q}/(\sigma))} \to \smash{H_{1,1}^{E_\infty}(\BGLb(F)_\bb{Q}/(\sigma))}$ is an isomorphism for any field $F$. Thus we may assume $n \ge 2$. 

\medskip

Using the universal coefficients theorem we may replace $\bb{Q}$-coefficients by $\bb{R}$-coefficients. As a consequence of \cref{thm:integration-assocative-algebra}, for any embedding $F \hookrightarrow \bb{C}$ the composition 
\[\BGLb(F)_\bb{R} \lra \BGLb(\bb{C})_\bb{R} \lra \bf{C}_*(\fr{gl},\fr{u})\]
is a map of augmented $E_1^\rm{u}$-algebras, so we may consider the commutative diagram
    \[\begin{tikzcd}[column sep=small, row sep=small] H_{n,2n-1}(\BGLb(F)_\bb{R}) \arrow{rr} \arrow{dd} \arrow[two heads]{rd} &[-15pt] &[-15pt] H_{2n-1}(\bf{C}_*(\fr{gl},\fr{u})) \arrow{dd} \arrow{rd} & \\
    & H_{n,2n-1}(\BGLb(F)_\bb{R}/(\sigma)) \arrow{dd} & & H_{2n-1}(\bf{C}_*(\fr{gl},\fr{u})/(\sigma)) \arrow{dd} \\
    H^{E_\infty}_{n,2n-1}(\BGLb(F)_\bb{R}) \arrow[hook]{dd} \arrow{rd}{\cong} & & H^{E_\infty}_{2n-1}(\bf{C}_*(\fr{gl},\fr{u})) \arrow[hook]{dd} \arrow{rd}{\cong} & \\
    & H^{E_\infty}_{n,2n-1}(\BGLb(F)_\bb{R}/(\sigma)) \arrow[hook]{dd} & & H^{E_\infty}_{2n-1}(\bf{C}_*(\fr{gl},\fr{u})/(\sigma)) \arrow[hook]{dd} \\
    H^{E_1}_{n,2n-1}(\BGLb(F)_\bb{R})  \arrow{rd} \arrow{rr} & & H^{E_1}_{2n-1}(\bf{C}_*(\fr{gl},\fr{u})) \arrow{rd} & \\
    & H^{E_1}_{n,2n-1}(\BGLb(F)_\bb{R}/(\sigma)) & & H^{E_1}_{2n-1}(\bf{C}_*(\fr{gl},\fr{u})/(\sigma)) \end{tikzcd}\]
where the top-left diagonal map is surjective as a consequence of \cref{lem:borel-yang-prim-rel-iso}, the bottom vertical maps from $E_\infty$- to $E_1$-homology are the splittings from \cite[Proposition C.16]{KRS1} so in particular split injective, and middle diagonal map are isomorphisms since $n \geq 2$.

We start with a nonzero element $\ol{x} \in H_{n,2n-1}(\BGLb(F)_\R/(\sigma))$, and a lift $x \in H_{n,2n-1}(\BGLb(F)_\bb{R})$. By the commutativity of the left side, it suffices to prove $x$ maps to a nonzero element of $H^{E_1}_{n,2n-1}(\BGLb(F)_\bb{R})$. We do so by moving to the right side. By \cref{lem:borel-yang-regulators} there exists an embedding $\nu \colon F \hookrightarrow \C$ such that
\[H_{2n-1}(\GL_n(F);\bb{R}) \lra H_{2n-1}(\GL_n(\bb{C});\bb{R}) \xrightarrow{- \cap \rm{Bo}_{n,2n-1}} \bb{R}\]
has nonzero value on $x$. By \cref{lem integration map and borel regulator} the image $y$ of $x$ in $H_{n,2n-1}(\bf{C}_*(\fr{gl},\fr{u}))$ is nonzero and in fact has a component that is a nonzero multiple of $\ol{\sf{c}}_{n,2n-1}$, so its image $\ol{y}$ in $H_{n,2n-1}(\bf{C}_*(\fr{gl},\fr{u})/(\sigma))$ is also nonzero by \cref{lem:transferring-vanishing-line-rel}. By commutativity of the back square, it suffices to prove $y$ maps to a nonzero element of $\smash{H^{E_1}_{n,2n-1}(\bf{C}_*(\fr{gl},\fr{u}))}$ and this will follow when we prove that $\ol{y}$ maps to a nonzero element of $\smash{H^{E_\infty}_{n,2n-1}(\bf{C}_*(\fr{gl},\fr{u})/(\sigma))}$.

There is a map of $E_\infty$-algebras to a trivial $E_\infty$-algebra
\[\bf{C}_*(\fr{gl},\fr{u}) \lra \bf{T} \coloneq \triv_{E_\infty}(\bb{R}\{\ol{\sf{c}}_{1,1}, \ol{\sf{c}}_{2,3}, \ol{\sf{c}}_{3,5},\ldots\})\]
since $\bf{C}_*(\fr{gl},\fr{u})$ is isomorphic to its own homology, and there is a map in homology by first mapping to $H_{*,*}(\bf{C}_*(\fr{gl},\fr{u})/\sigma)$ and then using \cref{lem:relative-lie-algebra-mod-sigma}. Using the analogue of \cref{lem:transferring-vanishing-line-rel} for $E_1$-homology, see \cref{rem E1 homology version}, we deduce that the map $\bf{C}_*(\fr{gl},\fr{u})/\sigma \to \bf{T}$ is an isomorphism in homology and $E_1$-homology in bidegrees $(n,d)$ with $d \le 2n-1$. Thus, to prove the statement about $\ol{y}$ we may verify the injectivity result for $\bf{T}$ where it is trivial since its $E_\infty$-homology is the cofree Lie coalgebra on the image of the classes $\ol{\sf{c}}_{n,2n-1}$ by Koszul duality. 
\end{proof}

\section{A cocycle for the Goncharov class} 

For a field $F$, the canonical map from homology to $E_\infty$-homology yields a map
\[H_*(\GL_n(F)) \lra H_{*-2n+1}(\GL_n(F);\scr{G}_n(F)).\]
For $*=2n-1$, this can be interpreted as an element $\rm{g}_n  \in H^{2n-1}(\GL_n(F);\scr{G}_n(F))$ by the universal coefficients theorem, which we call the \emph{Goncharov class}; the canonical map is given by the cap product against it. The goal of this section is to give an explicit cocycle representative, with notation to be defined later:

\begin{theorem}\label{thm:goncharov-class} The Goncharov class $\rm{g}_n \in H^{2n-1}(\GL_n(F);\scr{G}_n(F))$ is represented by the inhomogeneous cocycle
\begin{align*}
    \sf{g}_n \colon \GL_n(F)^{2n-1} &\lra \scr{G}_n(F) \\
    [g_1|\compactcdots|g_{2n-1}] &\longmapsto (-1)^{n-1} D_{\ell_{n-1},\ell_n}\big(\pi([\ell_0,\compactldots,\ell_{n-1}] \otimes [\ell_n,\compactldots,\ell_{2n-1}])\big)
\end{align*}
where $\ell_i \coloneqq g_1 \compactcdots g_i \cdot \ell$ for $\ell \subseteq F^n$ a fixed choice of line.
\end{theorem}

In the next section we identify, for $F$ a number field with embedding into $\bb{C}$, its composition with Goncharov's real period map \cite{Gon19} as a nonzero multiple of the Borel class.

\subsection{Group cohomology and extensions} Fix a discrete group $G$. We work over $\bb{Q}$ (though any field will do) and drop $\bb{Q}$ from the notation: e.g.~a \emph{$G$-module} is a left $\ds{Q}[G]$-module. Let $\scr{D} \coloneq \scr{D}(\ds{Q}[G])$ be the derived category of $G$-modules. Group cohomology with coefficients a $G$-module $M$ can be defined in terms of Ext-functors
\[H^d(G;M) = \rm{Ext}^d_{\bb{Q}[G]}(\bb{Q},M) \coloneq \pi_0\,\rm{Hom}_\scr{D}(\bb{Q},M[d]).\]
Yoneda identified this in terms of extensions \cite[Theorem 3.5]{Yoneda-homology}: an extension of $G$-modules of $M$ by $N$ of degree $d$
\[0 \to N \to N_{d-1} \to \cdots \to N_0 \to M \to 0\]
can be viewed as a quasi-isomorphism of chain complexes $(0 \to N \to N_{d-1} \to \cdots \to N_0) \xrightarrow{\simeq} M$ and hence gives an element in $\rm{Ext}^d_{\bb{Q}[G]}(M,N) = \pi_0\,\rm{Hom}_\scr{D}(M,N[d])$ via the zig-zag
\[\begin{tikzcd}
0 \rar & 0 \rar & 0 \rar & \cdots \rar & M \\[-11pt]
    0 \rar \dar \uar & N \rar \dar[swap]{\id\,} \uar & N_{d-1} \uar \rar \dar & \cdots \rar & N_0 \dar \uar{\simeq\,} \\[-11pt]
    0 \rar & N \rar & 0 \rar & \cdots \rar & 0. 
\end{tikzcd}\]
This construction induces for $d>0$ an isomorphism between $\rm{Ext}^d_{\bb{Q}[G]}(M,N)$ and equivalence classes of degree $d$ extensions of $M$ by $N$, where two such extensions are equivalent if there is a zig-zag of maps of extensions between them so that the maps are the identity on $M$ and $N$. We will explain how to think of cup and cap products in this language, using bar constructions. These are well-known facts in homological algebra: we state them precisely to fix conventions, but forego detailed proofs.

\subsubsection{Composition products} The \emph{composition product} on $\rm{Ext}$-groups is the composition map (tensor products are over $\bb{Q}$ unless mentioned otherwise)
\[ \circ \colon \pi_0\,\rm{Hom}_\scr{D}(B,C[t]) \otimes  \pi_0\,\rm{Hom}_\scr{D}(A,B[s]) \lra \pi_0\,\rm{Hom}_\scr{D}(A,C[s+t])\]
sending $g \otimes f \mapsto g[s] \circ f$. Thus, under the correspondence between $\rm{Ext}$-groups and extensions one computes the composition product as follows: given extensions 
\[0 \to C \to C_{s-1} \to \compactcdots \to C_0 \to B \to 0, \qquad 0 \to B \to B_{t-1} \to \compactcdots \to B_0 \to A \to 0\]
their composition product is represented by the extension 
 \[0 \to C \to C_{s-1} \to \compactcdots \to C_0 \to B_{t-1} \to \compactcdots \to B_0 \to A \to 0.\]

\subsubsection{Cup products}
The category of $G$-modules and its associated derived category $\scr{D}$, admit a symmetric monoidal structure $\otimes$ defined as follows: if $M, N$ are $G$-modules then $M \otimes N$ is a $\bb{Q}[G] \otimes \bb{Q}[G]$-module, and we can view it as a $\bb{Q}[G]$-module via the coproduct $\bb{Q}[G] \to \bb{Q}[G] \otimes_\bb{Q} \bb{Q}[G]$ induced by $g \mapsto g \otimes g$ for $g \in G$. This symmetric monoidal structure induces in turn a \emph{cup product} 
\[\cup \colon \pi_0\,\rm{Hom}_\scr{D}(A,B[s]) \otimes \pi_0\,\rm{Hom}_\scr{D}(A',B'[t]) \lra \pi_0\,\rm{Hom}_\scr{D}(A \otimes A',(B \otimes B')[s+t])\]
by applying $\otimes$ to morphisms and using the canonical equivalence $B[s] \otimes B'[t] \simeq (B \otimes B')[s+t]$. We now explain how to compute this at the level of extensions. 

\begin{lemma} \label{lem:ext-cup-products}
For a pair of extensions
\[0 \to B \to B_{s-1} \to \compactcdots \to B_0 \to A \to 0, \qquad 0 \to B' \to B'_{t-1} \to \compactcdots \to B'_0 \to A' \to 0,\]
the following two extensions both compute their cup product:
\begin{enumerate}[(i)]
    \item \label{enum:ext-cup-products-i} The extension 
\[0 \to (B \otimes B') \to (B \otimes B')_{s+t-1} \to \compactcdots \to (B \otimes B')_0 \to A \otimes A' \to 0,\]
where for $0 \le d \le s+t-1$ we let $(B \otimes B')_d \coloneq \bigoplus_{i+j=d} B_i \otimes B'_j$, with the convention that $B_s \coloneq B$ and $B'_t\coloneq B'$. The maps are given by the maps of the given two extensions with usual Koszul signs for $B_i, B'_i$ in degree $i$ for each $i \ge 0$. 
\item \label{enum:ext-cup-products-ii} The extension corresponding to the composition product of 
 \begin{align*}&0 \to B \otimes B' \to B_{s-1} \otimes B' \to \compactcdots \to B_0 \otimes B' \to A \otimes B' \to 0 \qquad \text{and} \\
 &0 \to A \otimes B' \to A \otimes B'_{t-1} \to \compactcdots \to A \otimes B'_0 \to A \otimes A' \to 0, \end{align*}
 that is, the extension given by
  \[0 \to  B \otimes B' \to B_{s-1} \otimes B' \to \compactcdots \to B_0 \otimes B'\to A \otimes B'_{t-1} \to \compactcdots \to A \otimes B'_0 \to A \otimes A' \to 0.\]
\end{enumerate}
\end{lemma}

\subsubsection{Cup products in group cohomology} Taking $A=A'=\bb{Q}$, $B=M$, and $B'=N$, we obtain the cup product in group cohomology $\cup \colon H^s(G;M) \otimes H^t(G;N) \to H^{s+t}(G;M \otimes N)$. To give an explicit natural chain-level description, we use the bar resolution $\epsilon \colon B_\bullet(\bb{Q}[G],\bb{Q}[G],\bb{Q}) \xrightarrow{\simeq} \bb{Q}$ of $G$-modules given by
\begin{equation}\label{eqn:non-homogeneous-g-mod} B_p(\bb{Q}[G],\bb{Q}[G],\bb{Q}) \coloneq \free_{\bb{Q}[G]}(\bb{Q}[G^p]) \cong \bb{Q}[G^{p+1}].\end{equation} 
This is the free $G$-module on $\bb{Q}[G^p]$, so has a $\bb{Q}[G]$-basis denoted as $[g_1|\compactcdots|g_p] \in G^p$, and its differential $d_p \colon B_p(\bb{Q}[G],\bb{Q}[G],\bb{Q}) \to B_{p-1}(\bb{Q}[G],\bb{Q}[G],\bb{Q})$ is given on this basis element by
    \[[g_1|\compactcdots|g_p] \mapsto g_1{\star}[g_2|\compactcdots|g_p] + \sum_{i=1}^{p-1} (-1)^i [g_1|\compactcdots|g_i g_{i+1}|\compactcdots|g_p] + (-1)^p [g_1|\compactcdots|g_{p-1}],
    \]
where $\star$ denotes the multiplication by elements of $G$ with respect to the free $G$-module structure. For $M$ a $G$-module, we define the \emph{inhomogeneous cochains with coefficients in $M$} as
\[C^\bullet(G;M) \coloneq \rm{Hom}_G(B_\bullet(\bb{Q}[G],\bb{Q}[G],\bb{Q}),M) \cong \rm{Hom}_\bb{Q}(\bb{Q}[G^\bullet],M)\]
with the differential induced by that of bar resolution. A cocycle $\phi \in C^s(G;M)$ defines $[\phi] \in \pi_0\,\rm{Hom}_\scr{D}(\bb{Q},M[s])$ via the zig-zag
\[\bb{Q} \underset{\simeq}{\overset{\epsilon}\longleftarrow} B_\bullet(\bb{Q}[G],\bb{Q}[G],\bb{Q}) \overset{\phi}{\lra} M[s]\]
where the cocycle equation is used to say that the right map is a chain map.

\begin{lemma} \label{lem chain description of cup product}
Given inhomogeneous cocycles $\phi \in C^s(G;M)$ and $\psi \in C^t(G;N)$, the inhomogeneous cocycle
\[[g_1|\compactldots|g_{s+t}] \longmapsto \phi(g_1,\compactldots,g_s) \otimes g_1 \compactcdots g_s {\star}\psi(g_{s+1},\compactldots,g_{s+t})\]
represents the cup product $\phi \cup \psi \in C^{s+t}(G;M \otimes N)$.
\end{lemma}

\begin{proof}
By definition, the cup product is given by $[\phi] \otimes [\psi] \in \rm{Hom}_\scr{D}(\bb{Q} \otimes \bb{Q}, (M \otimes_\bb{Q} N)[s+t])$. To represent it by an inhomogeneous cocycle we need an equivalence 
\[\Delta \colon B_\bullet(\bb{Q}[G],\bb{Q}[G],\bb{Q}) \smash{\xrightarrow{\simeq}} B_\bullet(\bb{Q}[G],\bb{Q}[G],\bb{Q}) \otimes B_\bullet(\bb{Q}[G],\bb{Q}[G],\bb{Q})\] 
which lifts $\bb{Q} \simeq \bb{Q} \otimes \bb{Q}$, and take the composition $(\phi \otimes \psi) \circ \Delta \colon B_\bullet(\bb{Q}[G],\bb{Q}[G],\bb{Q}) \to M \otimes N$. By \cite[page 226]{Yoneda-homology} or \cite[V.(1.3)]{Brown} (the sign difference is due to an ordering convention for the composition product), a choice of $\Delta$ is given by $\Delta_p[g_1|\compactldots|g_p] \coloneq \sum_{i=0}^{p} [g_1|\compactldots|g_i] \otimes g_1 \compactcdots g_i {\star}[g_{i+1}|\compactcdots|g_p]$, with the convention $[\;]=1 \in \bb{Q}=\bb{Q}[G^0]$. This yields the formula.
\end{proof}

\subsubsection{Cap products in group cohomology}
Dually, for a $G$-module $M$ we have \emph{group homology}
\[H_d(G;M) \coloneq\rm{Tor}_d^\scr{D}(\bb{Q},M)= H_d(\bb{Q} \otimes_G M),\]
where we view $\bb{Q} \otimes_G(-) \coloneq \bb{Q} \otimes_{\bb{Q}[G]}(-)$ as a functor $\scr{D}=\scr{D}(\bb{Q}[G]) \to \scr{D}(\bb{Q})$. As discussed above, any extension $0 \to N \to N_{t-1} \to \cdots \to N_0 \to M \to 0$ of $G$-modules gives a (homotopy class of) map $M \to N[t]$ in $\scr{D}$ and induces, by applying the functor $H_{s+t}(\bb{Q} \otimes_G -)$, canonical maps $H_{s+t}(G;M) \to H_s(G;N)$. This construction is linear in the extension and induces the \emph{cap product} pairing
\[\cap \colon H_{s+t}(G;M) \otimes \rm{Ext}_G^t(M,N) \lra H_s(G;N).\]
The relevant case for us is the cap product $\cap \colon H_{s+t}(G;\bb{Q}) \otimes_\bb{Q} H^t(G;M) \xrightarrow{\cap} H_s(G;M)$ between group homology and group cohomology. Before giving an explicit chain-level description, recall that for a $G$-module $M$ we have \emph{inhomogeneous chains with coefficients in $M$} given by
\[C_p(G;M) \coloneq B_p(\bb{Q},\bb{Q}[G],\bb{Q}[G]) \otimes_G M \cong \bb{Q}[G^{\bullet}] \otimes M\]
with differential given on $[g_1|\compactcdots|g_p] \otimes m$ by
\[[g_2|\compactcdots|g_p] \otimes m +\sum_{i=1}^{p-1} (-1)^i[g_1| \compactcdots| g_i g_{i+1}| \compactcdots | g_p] \otimes m + (-1)^p [g_1| \compactcdots |g_{p-1}] \otimes g_p\cdot m.\]
When $M=\bb{Q}$ we use the notation $[g_1|\compactcdots|g_p] \coloneq [g_1|\compactcdots|g_p] \otimes 1$. This construction can be generalised to $N \in \scr{D}$ by representing $N$ by a chain complex $N_\bullet$ and taking the totalisation of the double complex $(\bb{Q}[G^{p}] \otimes N_q$.

\begin{lemma} \label{lem chain description cap product}
Let $\phi \in C^t(G;M)$ be an inhomogeneous cocycle, then the chain map
\begin{align*}- \cap \phi \colon C_{s+t}(G;\bb{Q}) &\lra C_s(G;M) \\
[g_1|\compactcdots|g_{s+t}] &\longmapsto [g_1|\compactcdots|g_s] \otimes  \phi(g_{s+1},\compactldots,g_{s+t})\end{align*}
induces the cap product with the class represented by $\phi$ in cohomology.
\end{lemma}

\begin{proof}
    We view the cocycle $\phi$ as a chain map $\phi \colon B_\bullet(\bb{Q}[G],\bb{Q}[G],\bb{Q}) \to M[t]$. By definition of the cap product, the map in question is given by the zig-zag
    \[C_*(G;\bb{Q}) \xleftarrow[\simeq]{C_*(G,\epsilon)} C_*(G;B_\bullet(\bb{Q}[G],\bb{Q}[G],\bb{Q})) \xrightarrow{C_\bullet(G,\phi)} C_*(G;M[t])=C_{*-t}(G;M).\]
    To get the formula, note an explicit chain homotopy inverse to $C_*(G,\epsilon)$ is given by 
    \[[g_1|\compactcdots|g_n] \longmapsto \sum_{t=0}^n [g_1|\compactcdots|g_t] \otimes (g_{t+1},\compactldots,g_n) \in \bigoplus_{t=0}^n \bb{Q}[G^t] \otimes C_{n-t}.\qedhere\]
\end{proof}

\subsection{Steinberg cocycles} 

\subsubsection{The Steinberg cocycle} We next understand a cocycle representing the cohomology class $\rm{st}_n \in H^{n-1}(\GL_n(F),\St_n(F))$ associated with the extension 
\[ 0 \to \St_n(F) \to C_{n-2}(T_n) \to \cdots \to C_0(T_n) \to \bb{Q} \to 0\]
given by the augmented rational cellular chains on the semisimplicial Tits building $T_n \coloneq T(F^n)$, whose set of $p$-simplices is given by flags $0 \subsetneq V_0 \subsetneq\cdots \subsetneq V_p \subsetneq F^n$ and where the $i$th face maps forgets the $i$th subspace. We will refer to any representative of $\rm{st}_n$ as a \emph{Steinberg cocycle}.

When studying the group $\GL_n(F)$, there is another resolution of $\bb{Q}$ than the inhomogeneous bar resolution of the previous subsection. Let $L_\bullet$ be the full simplicial set on the set of lines, i.e.~its set of $p$-simplices $L_p \coloneq \bb{P}(F^n)^{p+1}$ is given by $(p+1)$-tuples of lines and the $i$th face map forgets the $i$th line. As $||L_\bullet|| \simeq *$, the associated augmented simplicial chain complex
\[\cdots \to \bb{Q}[L_1] \to \bb{Q}[L_0] \to \bb{Q} \to 0\]
is acyclic. This yields an equivalence $\bb{Q}[L_\bullet] \simeq \bb{Q}$ in $\scr{D} \coloneq \scr{D}(\bb{Q}[\GL_n(F)])$, where $\GL_n(F)$ acts in the evident manner on lines. It can be convenient to describe the cup product in terms of this resolution by lines rather than the inhomogeneous bar construction.

\begin{lemma} \label{lem cup product lines}
    Let $M, N$ be $\GL_n(F)$-modules and let $\phi \colon \ds{Q}[L_\bullet] \to M[s]$ and $\psi \colon \ds{Q}[L_\bullet] \to N[t]$ be chain maps. Then the cup product of the corresponding cohomology classes represented by $\phi$ and $\psi$ is represented by the chain map 
    \begin{align*}\ds{Q}[L_\bullet] &\lra  (M \otimes N)[s+t] \\
    (\ell_0,\dots,\ell_{s+t}) &\longmapsto \phi(\ell_0,\dots,\ell_s) \otimes \psi(\ell_s,\dots,\ell_{s+t}).\end{align*}
\end{lemma}

\begin{proof}
As in Lemma \ref{lem chain description of cup product}, in order to compute the cup product we need to produce a chain equivalence $\delta \colon \ds{k}[L_\bullet] \to \ds{Q}[L_\bullet] \otimes \ds{Q}[L_\bullet]$, lifting the canonical equivalence $\ds{Q} \simeq \ds{Q} \otimes_\ds{Q} \ds{Q}$. The following formula works: $\delta_p(\ell_0,\dots,\ell_p) \coloneq \sum_{i=0}^p (\ell_0,\dots,\ell_i) \otimes (\ell_i,\dots,\ell_p)$.\end{proof}

\begin{lemma}\label{lem:equivalent-extension-steinberg}  The Steinberg class $\rm{st}_n \in H^{n-1}(\GL_n(F);\St_n(F))$ is represented by the ``linear'' $(n-1)$-cocycle 
\begin{align*}\sf{st}_n \colon \bb{Q}[L_{n-1}] &\lra \St_n(F) \\
(\ell_0,\compactldots,\ell_{n-1}) &\longmapsto \begin{cases} (-1)^{n-1}[\ell_0,\compactldots,\ell_{n-1}] & \text{if generic,} \\ 0 & \text{otherwise.}\end{cases}\end{align*}
\end{lemma}

\begin{proof}We first prove that there is a commutative diagram 
    \[ \begin{tikzcd}[column sep=small]
    \cdots \rar & \bb{Q}[L_{n}] \dar{0} \rar{d} & \bb{Q}[L_{n-1}] \dar{d} \rar{d} & \bb{Q}[L_{n-2}] \dar{\id} \rar{d} & \cdots \rar{d} & \bb{Q}[L_0] \dar{\id} \rar & \bb{Q} \dar{\id} \rar & 0 \\[-5pt]
  \cdots \rar & 0 \rar \dar & \Lambda \rar{\inc} \dar{a}  & \bb{Q}[L_{n-2}] \rar{d} \dar{a_{n-2}}  & \cdots \rar{d} & \bb{Q}[L_0] \rar \dar{a_0}  & \bb{Q} \rar \dar{\id} & 0\\[-5pt]
  \cdots \rar & 0 \rar & \St_n(F) \rar & C_{n-2}(T_n) \rar   & \cdots \rar & C_0(T_n) \rar & k  \rar & 0
\end{tikzcd} \]
where $\Lambda \coloneq \ker(\bb{Q}[L_{n-2}] \to \bb{Q}[L_{n-3}])$, $a=a_{n-2}|_\Lambda$, and for $0 \le p \le n-2$ 
\[a_p(\ell_0,\compactldots,\ell_p) \coloneq \begin{cases} \sum_{\sigma \in \fr{S}_{p+1}} (-1)^\sigma (\ell_{\sigma(0)}\subsetneq \ell_{\sigma(0)}+ \ell_{\sigma(1)} \subsetneq \cdots \subsetneq \ell_{\sigma(0)}+\compactcdots+ \ell_{\sigma(p)}) & \text{if generic,} \\ 0 & \text{otherwise.} \end{cases}\]
Let $L_\bullet^{(n-2)}$ denote the $(n-2)$-skeleton of $L_\bullet$, the semisimplicial set containing only the $k$-simplices for $k \leq n-2$. Let $T'_\bullet$ denote the simplicial set obtained from the semisimplicial set $T_\bullet$ by freely adjoining degeneracies, which amounts to allowing $V_i = V_{i+1}$. There are maps
\[|L^{(n-2)}_\bullet| \overset{\cong}\longrightarrow |\rm{simp}\, L^{(n-2)}_\bullet| \xrightarrow{\rm{span}} |(T'_n)_\bullet| \overset{\simeq}\lra |(T_n)_\bullet|.\]
Firstly, the left homeomorphism is that from the geometric realisation of a semisimplicial set to its barycentric subdivision (obtained by passing to the nerve of the poset of simplices). This is a cellular map and the induced map on cellular chains is given by
\[(\ell_0,\compactldots,\ell_p) \longmapsto \sum_{\sigma \in \fr{S}_{p+1}} (-1)^\sigma (\ell_{\sigma([0])}) \subsetneq (\ell_{\sigma([1])}) \subsetneq \cdots \subsetneq (\ell_{\sigma([p])})\] 
where $(\ell_{\sigma([j])}) \coloneq (\ell_{i_0},\compactldots,\ell_{i_j})$ if $\sigma([j]) = \{i_0,\ldots,i_j\}$ with $i_0<\cdots<i_j$.

Secondly, the middle map is induced by a semisimplicial map $\rm{simp}\, L^{(n-2)}_\bullet$ to $(T'_n)_\bullet$ induced by sending an ordered collection of lines $(\ell_0,\compactldots,\ell_p)$ to its span $\ell_0+\cdots+\ell_p$. Finally, the right homotopy equivalence is that from thick geometric realisation of a simplicial set to its thin realisation. This is a cellular map and the induced map on cellular chains is given by sending degenerate simplices to zero. Composing these maps on cellular chains, we get the chain map $a_\bullet$ in the statement. This yields the commutative diagram.

\smallskip

We next prove that composition $a \circ d \colon \bb{Q}[L_{n-1}] \to \St_n(F)$ is given by 
\[(\ell_0,\compactldots,\ell_{n-1}) \mapsto \begin{cases} (-1)^{n-1}[\ell_0,\compactldots,\ell_{n-1}] & \text{if generic,} \\ 0 & \text{otherwise,} \end{cases}\]
where $[\ell_0,\compactldots,\ell_{n-1}]$ is the apartment class. We first observe that if $(\ell_0,\compactldots,\ell_{n-1})$ are generic, i.e.~span $F^n$, applying $a \circ d$ yields  
\[(-1)^{n-1}\sum_{\sigma \in \fr{S}_{n}} (-1)^\sigma (\rm{span}\,\ell_{\sigma([0])}  \subsetneq \rm{span}\,\ell_{\sigma([1])} 
 \subsetneq \compactcdots \subsetneq \rm{span}\,\ell_{\sigma([n-1])}), \] 
which is by definition the apartment class $(-1)^{n-1}[\ell_0,\dots,\ell_{n-1}] \in \St_n(F)$, see \cite[(23)]{CharltonRadchenkoRudenko}. We claim that if they are not generic, then applying $a \circ d$ vanishes. To see this, we note that the map $\rm{span}$ extends to $\rm{simp}\, L^\rm{prop}_\bullet$ where $L^\rm{prop}_\bullet \subset L_\bullet$ is the subcomplex of lines which do not span, but the composition (which is in fact an equivalence \cite[Theorem 1]{KahnSun})
\[|L^\rm{prop}_\bullet| \overset{\cong}\longrightarrow |\rm{simp}\, L^\rm{prop}_\bullet| \xrightarrow{\rm{span}} |(T'_n)_\bullet| \overset{\simeq}\lra |(T_n)_\bullet|\]
induces a chain map $a^\rm{prop}$ extending $a$ and in degree $n$ this is zero because the target vanishes then, which implies the claim.
\end{proof}

We can relate the linear resolution to the inhomogeneous bar resolution from \eqref{eqn:non-homogeneous-g-mod} as follows: for a fixed line $\ell \in \bb{P}^1(F^n)$ there is a chain map
\begin{equation} \label{eqn:comparison-of-two-resolutions} \begin{aligned} \sf{L} \colon B_\bullet(\bb{Q}[\GL_n(F)],\bb{Q}[\GL_n(F)],\bb{Q}) &\lra \bb{Q}[L_\bullet] \\
   [g_1|\compactcdots|g_p] &\longmapsto (\ell,g_1 \cdot \ell, \compactldots, g_1 \compactcdots g_p \cdot \ell), \end{aligned}
\end{equation}
compatible with their augmentations to $\bb{Q}$, hence representing $\id_\bb{Q}$ in $\scr{D}$. Under $\sf{L}$ from \eqref{eqn:comparison-of-two-resolutions}, the map $\Delta$ from the proof of \cref{lem chain description of cup product} is compatible with the map $\delta$ from the proof of \cref{lem cup product lines}, so cup products computed using either diagonal map agree. We obtain from this discussion an inhomogeneous cocycle $\sf{st}^B_n$ representing $\rm{st}_n$ by taking $\sf{st}_n \circ \sf{L}$ with $\sf{L}$ as in \eqref{eqn:comparison-of-two-resolutions}:

\begin{lemma} \label{lem:formula-steinberg-class}
    The Steinberg class $\rm{st}_n \in H^{n-1}(\GL_n(F);\St_n(F))$ is represented by the inhomogeneous $(n-1)$-cocycle 
    \begin{align*}\sf{st}^B_n \colon \GL_n(F)^{n-1}&\lra \St_n(F) \\
    [g_1|\compactldots|g_{n-1}] &\longmapsto (-1)^{n-1} [\ell_0,\ell_1,\ell_2,  \compactldots,\ell_{n-1}]\end{align*}
    for a fixed choice of line $\ell \in \bb{P}^1(F)$ and $\ell_i \coloneqq g_1\compactcdots g_i \cdot \ell$, where the right-hand side is 0 if the apartment is degenerate. 
\end{lemma}

\begin{proposition}\label{prop:canonical-e1-cap-with-st} The canonical map
\[H_d(\GL_n(F)) = H_{n,d}(\BGLb(F)_\bb{Q}) \lra H^{E_1}_{n,d}(\BGLb(F)_\bb{Q}) \cong H_{d-n+1}(\GL_n(F);\St_n(F))\]
is given by the cap product against $\rm{st}_n$ and hence on inhomogeneous chains given by the cap product against $\sf{st}_n^B$.
\end{proposition}

\begin{proof}Using \cite[Section 2.2, C.3]{KRS1}, the canonical map is induced by taking the map $\ul{\bb{Q}} \to B^\rm{As}_{\para} (\ul{\bb{Q}})[-1]$ in $\Fun(\Vect,\rm{Ch}_\bb{Q})$, applying $\dim_!$, and evaluating at $n \in \bb{N}$. Equivalently, it is induced by composition with the map $\ul{\bb{Q}}(F^n) \to B^\rm{As}_{\para} (\ul{\bb{Q}})(F^n)[-1]$ in $\scr{D}$. Using the proof of \cite[Corollary 4.12]{KRS1} we can identify this map with the inclusion
\[\bb{Q} \lra (C_{n-2}(T_n) \to \cdots \to C_0(T_n) \to \bb{Q})\]
into the bottom term of the augmented simplicial chains of the Tits building, where $C_p(T_n)$ is in degree $p+1$. The identification of the target with $H_{d-n+1}(\GL_n(F);\St_n(F))$ is induced by the equivalence
\[\St_n(F)[n-1] \overset{\inc}\lra (C_{n-2}(T_n) \to \cdots \to C_0(T_n) \to \bb{Q}),\]
and we saw above this zigzag is implemented by the cap product against the class represented by the extension $(0 \to \St_n(F) \to C_{n-2}(T_n) \to \cdots \to C_0(T_n) \to \bb{Q} \to 0)$.
\end{proof}

\subsubsection{The infinite Steinberg cocycle}

Since we work in characteristic zero  we can use the commutative bar construction. Recall that $\pi \colon \St_n(F) \otimes \St_n(F) \to \St^\infty_n(F)$ denotes the canonical projection onto indecomposables.

\begin{proposition}\label{prop:formula-for-stinfty-class} The cohomology class $\rm{st}^\infty_n \in H^{2n-2}(\GL_n(F);\St_n^\infty(F))$ corresponding to the extension 
     \[0 \to \St^\infty_n(F) \to (B_n^\rm{Com} \St)_n(F) \to \compactcdots \to (B_2^\rm{Com}\St)_n(F) \to C_{n-2}(T_n) \to \compactcdots\to \ds{Q} \to 0\]
    is given by $\pi(\rm{st}_n \cup \rm{st}_n)$. In particular, it is represented by the ``linear'' $(2n-2)$-cocycle $\sf{st}_n^\infty \coloneq \pi(\sf{st}_n \cup \sf{st}_n)$
    \begin{align*}\bb{Q}[L_{2n-3}] &\lra \St_n^\infty(F) \\
    (\ell_0,\compactldots,\ell_{2n-3}) &\longmapsto \pi \big( [\ell_0,\ell_1, \compactldots, \ell_{n-1}] \otimes [\ell_{n-1},\ell_n, \compactldots, \ell_{2n-3}] \big)\end{align*}
     for a fixed choice of line $\ell \in \bb{P}^1(F)$ and $\ell_i \coloneqq g_1\compactcdots g_i \cdot \ell$, as well as the inhomogeneous $(2n-2)$-cocycle $\sf{st}_n^{\infty,\rm{B}} \coloneq \sf{st}_n^\infty \circ \sf{L}$ admitting a formula analogous to \cref{lem:formula-steinberg-class}.
\end{proposition}

\begin{proof}There is a commutative diagram of extensions to the extension in the statement from the analogous extension where $\StL_n(F)$ is replaced by $\StH_n(F) \cong \St_n(F) \otimes \St_n(F)$ and the commutative bar construction is replaced by the associative bar construction, induced by $\pi$ and the projection from associative to commutative bar construction. This shows that the class in question is given by applying $\pi$ to the class in $H^{2n-2}(\GL_n(F);\St_n(F) \otimes \St_n(F))$ defined by the former. Using \cite[Theorem 23]{CharltonRadchenkoRudenko} we identify this extension as the composition product of two extensions
    \begin{align*}&(0 \to \St_n(F) \to C_{n-2}(T_n) \to \compactcdots \to C_0(T_n) \to \ds{Q}) \to 0) \otimes \St_n(F) \qquad \text{and} \\
    &0 \to \St_n(F) \to C_{n-2}(T_n) \to \compactcdots \to C_0(T_n) \to \ds{Q} \to 0,\end{align*}
and by \cref{lem:ext-cup-products} \eqref{enum:ext-cup-products-ii} this represents $\rm{st}_n \cup \rm{st}_n$. Finally, the formula is obtained by using the chain-level description of cup product of cocycles from \cref{lem chain description of cup product}. 
\end{proof}

Arguing as in \cref{prop:canonical-e1-cap-with-st}, we obtain:

\begin{proposition} \label{prop:canonical-map-chain-level}
    The canonical map from homology to $E_\infty$-homology
    \[H_d(\GL_n(F)) = H_{n,d}(\BGLb(F)) \lra H^{E_\infty}_{n,d}(\BGLb(F)) \cong H_{d-2n+2}(\GL_n(F);\St_n^\infty(F))\]
    is given by taking cap product with $\rm{st}_n^\infty = \pi(\rm{st}_n \cup \rm{st}_n)$ and hence on inhomogeneous chains given by the cap product against $\sf{st}_n^{\infty,\rm{B}}$ as in \cref{prop:formula-for-stinfty-class}.
\end{proposition}

\subsection{A cocycle for the Goncharov class} 

\subsubsection{Dual decomposition operator} \label{sec:dual-decomposition}
Let $V$ be an $n$-dimensional vector space. We define $\FL(V)$
to have the same generators, relations, and $\GL(V)$-action as
$\FI(V)$ \cite[Section 5.1.2]{KRS1}; we will use the notation
\[
\FL[v_1,\compactldots,v_n]\in \FL(V)
\]
for its generators. The notation is different because we will use a different map to $\StL(V)$ and hence $\FL(V)$ ought to be interpreted differently from $\FI(V)$. Steinberg polylogarithms \cite[Definition 28]{CharltonRadchenkoRudenko} are elements of the space $\StH(V)$ defined by the formula
\[
\mathrm{L}[v_1,\dots,v_n]=[v_n,v_{n}+v_{n-1},\dots,v_n+\dots+v_1]\otimes [v_n,v_{n-1},\dots,v_1];
\]
we denote their projections to $\StL(V)$ by the same symbol. The map 
\begin{align*}\mathrm{L}\colon \FL(V) &\lra \StL(V)\\
\FL[v_1,\compactldots,v_n] &\longmapsto \mathrm{L}[v_1,\compactldots,v_n]\end{align*}
is well-defined by \cite[Proposition 33]{CharltonRadchenkoRudenko}, and \cite[Proposition 37]{CharltonRadchenkoRudenko} implies that the map is surjective.
 Consider the map
\begin{align*}
\rm{D}^{\rm{F}}\colon \FI(V) &\lra \FL(V^{\vee})  \\
\FI[v_1,\dots,v_n] &\longmapsto (-1)^n\FL[v^n,\dots,v^1]
\end{align*}
where the basis $v^j$ is dual to the basis $v_i$.  \cite[Proposition 32]{CharltonRadchenkoRudenko} implies that the following diagram commutes:
\[\begin{tikzcd}\FI(V) \dar[swap]{\rm{D}^{\rm{F}}} \rar{\rm{I}} & \StL(V) \dar{\rm{D}} \\[-5pt]
 \FL(V^{\vee}) \rar{\rm{L}} & \StL(V^{\vee}) .
\end{tikzcd}\]    

The kernel $\FLR(V)$ of the map $\rm{L}$ fits in the commutative diagram
\[\begin{tikzcd}0 \rar & \FIR(V) \dar{\rm{D}^{\rm{F}}} \rar &\FI(V) \dar{\rm{D}^{\rm{F}}}\rar{\rm{I}} & \StL(V) \dar{\rm{D}}  \rar &0\\[-5pt]
 0 \rar & \FLR(V^{\vee})  \rar &\FL(V^{\vee}) \rar{\rm{L}} & \StL(V^{\vee})\rar &0 .
\end{tikzcd}\]  
Both $H_{0}(\GL(V);\FIR(V))$ and $H_{0}(\GL(V);\FLR(V))$ are canonically isomorphic to $\PolyL_n(F)$. The results of \cite[Section 8.2]{KRS1} imply that the induced map
\begin{equation}\label{eqn: iso FI to FL}
     H_{0}(\GL(V);\FIR(V)) \stackrel{D_*}{\lra}  H_{0}(\GL(V^{\vee});\FLR(V^{\vee}))
\end{equation}
is given by multiplication by $-1$.

In \cite[Section 5.1.2]{KRS1} we described a decomposition operator $D^\FI_h  \colon \StL(V) \to \FI(V)$ for a nonzero functional $h \in V^\vee$. The \emph{dual decomposition operator} is defined as follows:
\begin{definition} Consider a nonzero vector $v\in V$. The dual decomposition operator $D^\FL_v$ is a composition
\[
\StL(V)\stackrel{D}{\lra} \StL(V^{\vee}) \stackrel{D^\FI_v}{\lra} \FI(V^{\vee}) \stackrel{\rm{D}^{\rm{F}}}{\lra} \FL(V),
\]
 where $D^\FI_v$  is the decomposition operator associated to the nonzero functional $v$ on $V^{\vee}$.
\end{definition}

Next we give an explicit formula for the dual decomposition operator. Observe that the projective dual to \cite[Proposition 2.32]{KRS1} says that the composition
\[\StL(V) \lra (B^\rm{Com}\SSt)^v_n(V)\]
of the symbol map with the projection onto those terms $[P_1|\compactcdots|P_n]$ so that $P_1,\ldots,P_n$ and $v$ are in general position is an isomorphism. For each such term there exist unique vectors $v_i\in P_i$ such that $v=v_1+\dots+v_n$, and we use this to define the map
\begin{align*}
\rm{C}^{\FL}_v \colon (B^\rm{Com}\SSt)^v_n(V) &\lra \FL(V) \\
[P_1|\compactcdots|P_n] &\longmapsto \FL[v_n,\compactldots,v_1].\end{align*} 

\begin{proposition}\label{prop: dual decomposition} The dual decomposition operator $D^\FL_v$ equals the composition 
\[
\StL(V) \lra  (B^\rm{Com}\SSt)_n(V)\lra  (B^\rm{Com}\SSt)^v_n(V) \stackrel{\rm{C}^{\FL}_v}{\lra} \FL(V).
\] 
with left map the symbol and the middle map the projection.
\end{proposition}

We start with a lemma relating the duality map $\vee$ defined in \cite[8.1.2]{KRS1} and the symbol. For that, we define the map 
\[
\vee\colon  (B^\rm{As}\SSt)_n(V)\lra  (B^\rm{As}\SSt)_n(V^{\vee})
\]
by the formula
\[
\vee([P_1|\dots|P_n])=[P^n|\dots| P^1], \quad P^i=\Bigl(\bigoplus_{j\neq i} P_j\Bigr )^{\perp}.
\]

\begin{lemma}\label{lem: duality and symbol}  The following diagram is commutative:
\begin{equation}
\begin{tikzcd}
\StH(V) \rar{s} \dar{\vee}& (B^\rm{As}\SSt)_n(V)  \dar{\vee}  \\[-5pt]
\StH(V^{\vee}) \rar{s}&  (B^\rm{As}\SSt)_n(V^{\vee})
\end{tikzcd}
\end{equation}
\end{lemma}

\begin{proof}
It suffices to check the claim on elements of the form
\[
[a_0,\ldots,a_{n-1}]\otimes[b_0,\ldots,b_{n-1}],
\]
since apartment classes span the Steinberg module. Let
$\alpha_0,\ldots,\alpha_{n-1}$ and
$\beta_0,\ldots,\beta_{n-1}$ be the bases of $V^\vee$
dual to the $a$- and $b$-bases, respectively. For $\sigma,\tau\in S_n$ and $0\leq i,j\leq n-1$, set
\begin{align*}
&F_i^\sigma
=
\langle a_{\sigma(0)},\ldots,a_{\sigma(i)}\rangle,
\qquad
G_j^\tau
=
\langle b_{\tau(n-j-1)},\ldots,b_{\tau(n-1)}\rangle,\\
&\mathcal F_i^\sigma
=\langle \alpha_{\sigma(0)},\ldots,\alpha_{\sigma(i)}\rangle,
\qquad
\mathcal G_j^\tau=\langle \beta_{\tau(n-j-1)},\ldots,\beta_{\tau(n-1)}\rangle.
\end{align*}
By \cite[Lemma 19]{CharltonRadchenkoRudenko},
\[
s\bigl(
[a_0,\ldots,a_{n-1}]
\otimes
[b_0,\ldots,b_{n-1}]
\bigr)
=
\sum_{\sigma,\tau\in S_n}
(-1)^\sigma(-1)^\tau
[P_0^{\sigma,\tau}|\cdots|P_{n-1}^{\sigma,\tau}],
\]
where
$
P_i^{\sigma,\tau}
:=
F_i^\sigma\cap G_{n-i-1}^\tau,$
and terms for which the two flags are not in general position are
omitted. Similarly,
\[
s\bigl(
[\alpha_0,\ldots,\alpha_{n-1}]
\otimes
[\beta_0,\ldots,\beta_{n-1}]
\bigr)
=
\sum_{\sigma,\tau\in S_n}
(-1)^\sigma(-1)^\tau
[
\mathcal F_0^\sigma\cap\mathcal G_{n-1}^\tau
|\cdots|
\mathcal F_{n-1}^\sigma\cap\mathcal G_0^\tau
].
\]

Consider the permutation 
\[
\rho(i)=n-1-i.
\]
We claim that $\vee$ sends the summand indexed by $(\sigma,\tau)$
in the first expression to the summand indexed by
$(\sigma\rho,\tau\rho)$ in the second expression. Suppose that $F^\sigma$ and $G^\tau$ are in general position.
Then
\[
V=P_0^{\sigma,\tau}\oplus\cdots\oplus P_{n-1}^{\sigma,\tau},
\]
and
\[
F_i^\sigma=P_0^{\sigma,\tau}\oplus\cdots\oplus P_i^{\sigma,\tau},
\qquad
G_{n-i-1}^\tau=P_i^{\sigma,\tau}\oplus\cdots\oplus P_{n-1}^{\sigma,\tau}.
\]
Consequently, with the conventions
$F_{-1}^\sigma=G_{-1}^\tau=0$, we have
\[
\bigoplus_{j\neq i}P_j^{\sigma,\tau}
=
F_{i-1}^\sigma\oplus G_{n-i-2}^\tau.
\]
Taking annihilators gives
\[
\left(
\bigoplus_{j\neq i}P_j^{\sigma,\tau}
\right)^\perp
=
(F_{i-1}^\sigma)^\perp
\cap
(G_{n-i-2}^\tau)^\perp.
\]
Since the $\alpha$'s and $\beta$'s are the corresponding dual
bases,
\[
(F_{i-1}^\sigma)^\perp
=
\langle
\alpha_{\sigma(i)},\ldots,\alpha_{\sigma(n-1)}
\rangle
=
\mathcal F_{n-1-i}^{\sigma\rho},
\]
and
\[
(G_{n-i-2}^\tau)^\perp
=
\langle
\beta_{\tau(0)},\ldots,\beta_{\tau(i)}
\rangle
=
\mathcal G_i^{\tau\rho}.
\]
Therefore
\[
\left(
\bigoplus_{j\neq i}P_j^{\sigma,\tau}
\right)^\perp
=
\mathcal F_{n-1-i}^{\sigma\rho}
\cap
\mathcal G_i^{\tau\rho}.
\]

By definition of $\vee$, we thus have
\[
\vee
\bigl([P_0^{\sigma,\tau}|\cdots|P_{n-1}^{\sigma,\tau}]\bigr)=
[P_{\sigma,\tau}^{n-1}|\cdots|P_{\sigma,\tau}^0]=
[
\mathcal F_0^{\sigma\rho}
 \cap \mathcal G_{n-1}^{\tau\rho}
|\cdots|
\mathcal F_{n-1}^{\sigma\rho}
 \cap \mathcal G_0^{\tau\rho}
],
\]
where $P_{\sigma,\tau}^i
:=
\left(
\bigoplus_{j\neq i}P_j^{\sigma,\tau}
\right)^\perp.$
Taking annihilators also shows that $F^\sigma$ and $G^\tau$ are
in general position if and only if
$\mathcal F^{\sigma\rho}$ and $\mathcal G^{\tau\rho}$ are in
general position. Thus the omitted terms correspond under the same
bijection.

Finally,
\[
\rm{sgn}(\sigma\rho)\:\rm{sgn}(\tau\rho)
=
\rm{sgn}(\sigma)\:\rm{sgn}(\tau).
\]
Reindexing the sum by
\[
(\sigma,\tau)\longmapsto(\sigma\rho,\tau\rho)
\]
therefore gives $\vee\circ s=s\circ \vee,$
which proves the commutativity of the diagram.
\end{proof}

\begin{proof}[Proof of \cref{prop: dual decomposition}]
Consider a diagram  
\begin{equation}\label{eqn: dual decomp operator}\begin{tikzcd}\StL(V) \rar \dar{\rm{D}}& (B^\rm{Com}\SSt)_n(V)\rar  \dar{\rm{D}} &(B^\rm{Com}\SSt)^v_n(V) \rar \dar{\rm{D}} & \FL(V) \dar{(\rm{D}^{\rm{F}})^{-1}} \\[-5pt]
\StL(V^{\vee}) \rar&  (B^\rm{Com}\SSt)_n(V^{\vee})\rar  &(B^\rm{Com}\SSt)^v_n(V^{\vee}) \rar& \FI(V^{\vee})
\end{tikzcd}\end{equation}
where $\rm{D}$ acts on $(B^\rm{Com}\SSt)_n(V)$ via 
\[
D([P_1|\dots|P_n])=[P^1|\dots| P^n] \quad \text{for}\quad P^i=\Bigl(\bigoplus_{j\neq i} P_j\Bigr )^{\perp}.
\] 
 We claim that  \eqref{eqn: dual decomp operator} is a commutative diagram.  The middle square of \eqref{eqn: dual decomp operator} clearly commutes;  the commutativity of the right square can be checked directly. \cref{lem: duality and symbol}  implies that for $x\in \StH(V)$ with $
s(x)=\sum n_i [(P_1)_i|\cdots|(P_n)_i]$, we have 
\[
s(\vee(x))=\sum n_i [(P^n)_i|\cdots|(P^1)_i].
\]
Up to shuffles, we have $[a_n|\dots|a_1]=(-1)^{n-1}[a_1|\dots|a_n]$, so the commutativity of the first square  follows from \cite[Lemma~8.1]{KRS1}.
The composition of the maps in the bottom row is precisely the decomposition operator $\rm{C}^{\FI}_v$ associated with the functional $v$ on $V^{\vee}$. This implies the statement.
\end{proof}

Informally, the dual decomposition operator $D^\FL_v$ writes the image of $\rm{FL}[v_1,\compactldots,v_n]$ as a sum of elements $\rm{FL}[w_1,\compactldots,w_n]$ with $w_1,\ldots,w_n$ in general position with $v$ and $v = w_1+\cdots+w_n$. Rewriting this in terms of the Steinberg iterated integrals instead of Steinberg multiple polylogarithms using \cite[Section 3.5]{CharltonRadchenkoRudenko}, we see $\rm{FI}[u_1,\compactldots,u_n]$ where $u_1 = v$.

\subsubsection{A cocycle for the Goncharov class} We now return to the Goncharov class associated with the linear map
\[H_{2n-1}(\GL_n(F)) \lra \scr{G}_n(F),\]
given by the composition of the canonical map $H_{2n-1}(\GL_n(F)) \to H_1(\GL_n(F);\StL_n(F))$ with the connecting homomorphism $H_1(\GL_n(F);\StL_n(F)) \to H_0(\GL_n(F);\rm{FIR}_n(F)) = \scr{G}_n(F)$. We have a description of the former on the chain level, so it remains to give one of the latter.

\smallskip

For a pair of nonzero vectors $v,v'$ we have a map
\[\widetilde{D}_{v,v'} \coloneq D^\FL_v-D^\FL_{v'} \colon \StL(V) \lra \FLR(V)\]
and projecting to the space of $\GL(V)$-coinvariants yields a map $D_{v,v'}\colon \StL(V)\to \PolyL(V)$, which for $n \geq 2$, by the invariance of correlators under scaling, only depends on the lines $\ell,\ell'$ spanned by $v,v'$. The connecting homomorphism in a long exact sequence associated with the short exact sequence 
    \[0 \lra \FLR(V) \lra \FL(V) \lra \StL(V) \lra 0\] 
induces an isomorphism $H_1(\GL(V);\StL(V)) \overset{\cong}\lra H_0(\GL(V);\FLR(V)) \cong \scr{G}_n(F)$.

\begin{lemma} \label{lem formula for boundary map from St_infty to FIR}
    Fixing a nonzero vector $v \in V$, this isomorphism is induced by the map 
    \begin{align*}C_1(\GL(V);\StL(V)) = \bb{Q}[\GL(V)] \otimes \StL(V) &\lra C_0(\GL(V);\FLR(V))= \FLR(V) \\
    [g] \otimes a &\longmapsto \widetilde{D}_{v,g \cdot v}(g \cdot a).\end{align*}
\end{lemma}

\begin{proof}
    Observe that the dual decomposition operator $D_v^{\FL} \colon \StL(V) \to \FL(V)$ gives a splitting as vector spaces. We will use this to describe the boundary map explicitly: let $x = \sum_i [g_i] \otimes a_i \in C_1(\GL(V);\StL(V))$ be a cycle. Using $D_v^{\FL}$ we can lift this to a chain in $C_1(\GL(V);\FL(V))$ given by $\widetilde{x} = \sum_i [g_i] \otimes D_v^{\FL}(a_i)$. Next, we take the differential of this chain to obtain the following element of $C_0(\GL(V);\FL(V))= \FL(V)$:
    \begin{align*} d \widetilde{x} &= \sum_i (D_v^{\FL}(a_i)- g_i \cdot D_v^{\FL}(a_i)) \\
    &=\sum_i (D_v^{\FL}(a_i- g_i \cdot a_i)+ D_v^{\FL}(g_i \cdot a_i) - D_{g_i v}^{\FL}(g_i \cdot a_i)) \\
    &= D_v^{\FL}(d(x))+\sum_i \widetilde{D}_{v,g_i \cdot v}(g_i \cdot a_i),\end{align*}
    where we have used that $g_i \cdot D_v(a_i)= D_{g_i \cdot v}(g_i \cdot a_i)$. Since $x$ was a cycle then the first term vanishes and the second one lies in $\FLR(V)$. Thus, the boundary map applied $x= \sum_i [g_i] \otimes a_i$ is given by formula in the statement of the lemma.
\end{proof}

Now suppose $n \geq 2$. By definition, we have $D_{x,y}+D_{y,z}=D_{x,z}$ and so we have a chain map
\begin{align*}\sf{D} \colon \ds{k}[L_\bullet] &\lra \rm{Hom}(\StL_n(F),\scr{G}_n(F))[1] \\
(\ell_0,\ell_1) &\longmapsto D_{\ell_0,\ell_1}.\end{align*}
Moreover, viewing $\rm{Hom}(\StL_n(F),\scr{G}_n(F))$ as a left $\GL_n(F)$-module via $(g \cdot f)(a)= f(g^{-1} \cdot a)$, we have $g \cdot D_{\ell_0,\ell_1}= D_{g \cdot \ell_0, g \cdot \ell_1}$ because for any $a \in \StL_n(F)$ we have 
 \[\widetilde{D}_{g \cdot \ell_0, g \cdot \ell_1}(a)= \widetilde{D}_{g \cdot \ell_0, g \cdot \ell_1}(g \cdot g^{-1} \cdot a)= g \cdot \widetilde{D}_{\ell_0,\ell_1}(g^{-1} \cdot a),\]
and taking coinvariants one gets that $D_{g \cdot \ell_0, g \cdot \ell_1}(a)= D_{\ell_0,\ell_1}(g^{-1} \cdot a) \in \scr{G}_n(F)$. Therefore, using $\ds{Q}[L_\bullet] \simeq \ds{Q}$, this defines a map $\pi_0\, \rm{Hom}_{\scr{D}}(\ds{Q}, \rm{Hom}(\StL_n(F),\scr{G}_n(F))[1])$, or equivalently a class $\rm{D}_n \in H^1(\GL_n(F);\rm{Hom}(\StL_n(F),\scr{G}_n(F)))$.

\begin{proposition}\label{prop:canonical-is-cap-gamman} The canonical map from homology to $E_\infty$-homology
    \[\rm{can}_n\colon H_{2n-1}(\GL_n(F)) = H_{n,2n-1}(\BGLb(F)_\bb{Q}) \lra H_{n,2n-1}^{E_\infty}(\BGLb(F)_\bb{Q}) \cong \scr{G}_n(F)\]
is given by taking the cap product against the cohomology class
\[ \rm{ev}_n (\rm{D}_n \cup \pi(\rm{st}_n \cup \rm{st}_n)) \in H^{2n-1}(\GL_n(F);\scr{G}_n(F)),\]
where $\rm{ev}_n \colon \rm{Hom}(\StL_n(F),\scr{G}_n(F)) \otimes \StL_n(F) \to \scr{G}_n(F)$ is the evaluation map.
\end{proposition}

\begin{proof}Under the indicated isomorphisms the canonical map is given by the composition
\[H_{2n-1}(\GL_n(F)) \lra H_1(\GL_n(F);\StL_n(F)) \lra H_0(\GL_n(F);\FLR(V)),\]
and combining \cref{prop:canonical-map-chain-level} with \cref{lem formula for boundary map from St_infty to FIR} the first map is given by cap product against $\pi(\rm{st}_n \cup \rm{st}_n)$ and the second map by cap product against $\rm{D}_n$ followed by evaluation.
\end{proof}

We can now prove the formula given in \cref{thm:goncharov-class} for the Goncharov class.

\begin{proof}[Proof of \cref{thm:goncharov-class}] Consider the map
\begin{align*}\sf{g}_n \colon \bb{Q}[L_{2n-1}] &\lra \scr{G}_n(F) \\
(\ell_0,\compactldots,\ell_{2n-1}) &\longmapsto (-1)^{n-1} D_{\ell_{n-1},\ell_n}(\pi ([\ell_0,\compactldots,\ell_{n-1}] \otimes [\ell_n,\compactldots,\ell_{2n-1}])).\end{align*}
Then $\sf{g}^B_n = \sf{g}_n \circ\sf{L}$ with $\sf{L} \colon C_{2n-1}(\GL_n(F)) \to \bb{Q}[L_{2n-1}]$ from \eqref{eqn:comparison-of-two-resolutions} is exactly given by the formula in \cref{thm:goncharov-class}. By \cref{lem cup product lines} this represents $(-1)^{n-1} \rm{ev}_n(\pi(\sf{st}_n \cup \sf{D} \cup \sf{st}_n)) = \rm{ev}_n(\sf{D} \cup \pi(\sf{st}_n \cup \sf{st}_n))$ using graded commutativity of the cup product. By \cref{prop:canonical-is-cap-gamman} this implements the canonical map on the chain level.\end{proof}

\begin{corollary}\label{cor:gn-cocycle-equation} $\sf{g}_n$ satisfies the $(2n-1)$-cocycle equation
    \[\sum_{i=0}^{2n} (-1)^i \sf{g}_n(\ell_0,\compactldots,\hat{\ell}_i, \compactldots, \ell_{2n})=0.\]
\end{corollary}

\subsection{Explicit formula for the Goncharov cocycle on generic configurations}
Our next goal is to give an explicit formula for the Goncharov class 
when the configuration of lines $\ell_0,\dots,\ell_{2n-1}$ is generic, in the sense that any $n$ out of $2n$ lines give a direct sum decomposition of $V$. This suffices to describe it as a measurable cocycle. To do so, we will express an element $\sf{g}_n(\ell_0,\compactldots,\ell_{2n-1}) \in \PolyL_n(F)$ in terms of the cluster Grassmannian polylogarithms that were introduced in \cite{MR22}.

\subsubsection{Cluster Grassmannian polylogarithms}
\label{sec:cluster grassm polylogs}
Assume $n\geq 2$. Let $v_0,\ldots,v_{2n-1}$ be a generic ordered tuple of vectors in $V$, and fix a volume form $0\neq\omega\in\Lambda^nV^\vee$.  For an ordered $n$-tuple of vectors in general position, the expression $\omega(v_{i_1},\compactldots,v_{i_n})$ agrees with a Pl{\"u}cker coordinate up to a sign.

For $0\leq k\leq n-1$, define an ordered $(n-1)$-tuple
\[
I_k \coloneq (v_k,v_{k+1},\compactldots,v_{n-2},v_n,v_{n+1},\compactldots,v_{n+k-1})
\]
with empty ranges omitted, and define
\[
 R_k(v_0,\compactldots,v_{2n-1}) \coloneq \frac{\omega(I_k,v_{2n-1})}
                  {\omega(I_k,v_{n-1})}.
\]
For an ordered $(n-2)$-tuple $S$ of vectors, we have the projected cross-ratio
\[
 [S | v_0,v_1,v_2,v_3]
 \coloneq \frac{\omega(S,v_0,v_1)\omega(S,v_2,v_3)}
          {\omega(S,v_0,v_3)\omega(S,v_2,v_1)},
\]
which also equals to the cross-ratio of the projections of the vectors $v_0, v_1, v_2, v_3$ to the quotient of the space $V$ by the span of the vectors in $S$. Finally, for $0\leq k\leq n-2$ we define
\[
 A_k \coloneq [v_{k+1},\compactldots,v_{n-2},v_n,\compactldots,v_{n+k-1}| v_k,v_{n-1},v_{n+k},v_{2n-1}].
\]
A direct determinant calculation gives
\begin{equation}\label{eq:Ak-R}
 A_k=\frac{R_{k+1}(v_0,\compactldots,v_{2n-1})}{R_k(v_0,\compactldots,v_{2n-1})}.
\end{equation}
The next definition uses the (unnormalised) alternation operator $\rm{Alt}_{S}(f) \coloneq \sum_{\sigma \in \mathfrak{S}_S}(-1)^{\sigma} \sigma f$.

\begin{definition}\label{def:GLi-revised}
For $n\geq 2$, $m\geq n-1$ and $m\geq2$, define
\begin{align}
 \GLiG_m(v_0,\compactldots,v_{2n-1})
 :=\mathrm{Alt}_{\{0,\compactldots,n-2\}}
    \mathrm{Alt}_{\{n,\compactldots,2n-2\}}
 \LiG_{m-n+1;1,\compactldots,1}(A_0,\compactldots,A_{n-2}).
\end{align}
\end{definition}

Equivalently, using \eqref{eq:Ak-R}, \cite[Definition 7.15]{KRS1}, and rescaling invariance, we have
\[
 \GLiG_m(v_0,\compactldots,v_{2n-1})
 =(-1)^{n-1}
 \mathrm{Alt}_{\{0,\compactldots,n-2\}}
\mathrm{Alt}_{\{n,\compactldots,2n-2\}} \ItG \!\left(
 0;\underbrace{0,\compactldots,0}_{m-n+1},
 R_0,\compactldots,R_{n-2};R_{n-1}\right).
\]

\begin{example} For $n=m=3$ we have
\begin{align*}
 \GLiG_3(v_0,v_1,v_2,v_3,v_4,v_{5})=&\quad\: \LiG_{1;1,1}([v_1|v_0,v_2,v_3,v_5],[v_3|v_1,v_2,v_4,v_5])\\
 &-\LiG_{1;1,1}([v_0|v_1,v_2,v_3,v_5],[v_3|v_0,v_2,v_4,v_5])\\
 &+\LiG_{1;1,1}([v_0|v_1,v_2,v_4,v_5],[v_4|v_0,v_2,v_3,v_5])\\
 &-\LiG_{1;1,1}([v_1|v_0,v_2,v_4,v_5],[v_4|v_1,v_2,v_3,v_5]).
\end{align*}
\end{example}

The cluster Grassmannian polylogarithm depends only on the configuration of lines $\ell_0,\ldots,\ell_{2n-1}$ spanned by the vectors $v_0,\dots,v_{2n-1}$, so we will sometimes use the notation $\GLiG_{m}(\ell_0,\compactldots,\ell_{2n-1})$ instead. The formal realisation of the element defined above coincides with the cluster Grassmannian polylogarithm defined in \cite[Section 4.2]{MR22}.

\begin{remark}
The adjective ``cluster'' refers to the symbol of the element
$\GLiG_m$.  Label the Pl{\"u}cker coordinates on
$\rm{Gr}(n,2n)$ by $p_I$, where $I\subset\{0,\ldots,2n-1\}$ and
$|I|=n$.  Two such subsets $I$ and $J$ are called \emph{weakly separated}
if there are no cyclically ordered indices $a,b,c,d$ with
$a,c\in I\setminus J$ and $b,d\in J\setminus I$; equivalently,
$I\setminus J$ and $J\setminus I$ can be separated by a chord of the
circle.  Maximal pairwise weakly separated collections are precisely the Pl\"ucker clusters.  The symbol of $\operatorname{GLi}^{\mathrm G}_m$
admits an integrable representative
\[
 \sum_\alpha c_\alpha\,
 p_{I_{\alpha,1}}\otimes\cdots\otimes p_{I_{\alpha,m}}
\]
such that, for every $\alpha$, the sets
$I_{\alpha,1},\ldots,I_{\alpha,m}$ are pairwise weakly separated, and
hence all letters of each tensor word belong to a single Pl\"ucker
cluster.  In this sense its symbol is a Pl\"ucker cluster-integrable
symbol \cite[Theorem~1.5(1) and Corollary~4.5]{MR22}.
\end{remark}

\subsubsection{The proof of the explicit formula}
The main result of this section is the following theorem:

\begin{theorem}\label{thm:explicit formula for the cocycle on generic configurations}
For $n\geq 2$ and a generic configuration of lines $\ell_0,\ldots,\ell_{2n-1}$, we have
\[
\sf{g}_n(\ell_0,\compactldots,\ell_{2n-1})
 =(-1)^{\binom{n}{2}+1}
 \GLiG_n(\ell_0,\compactldots,\ell_{n-1},
                       \ell_{2n-1},\compactldots,\ell_n).\]
\end{theorem}

\begin{proof} 
Choose vectors $a_0,\dots,a_{n-1}$ spanning the lines $\ell_0,\dots,\ell_{n-1}$ and $b_0,\dots,b_{n-1}$ spanning the lines $\ell_n,\dots,\ell_{2n-1}$. Introduce the notation
\[
T(a,b) \coloneq \pi ([a_0,\compactldots,a_{n-1}] \otimes [b_0,\compactldots,b_{n-1}])\in \StL(V).
\]
Let $\alpha_0,\ldots,\alpha_{n-1}$ and
$\beta_0,\ldots,\beta_{n-1}$ be the bases of $V^\vee$ dual to the $a$- and $b$-bases. By definition $\sf{g}_n(\ell_0,\compactldots,\ell_{2n-1})
 =(-1)^{n-1} D_{a_{n-1},b_0}(T(a,b))\in \PolyL_n(F)$.

We start by computing the symbol of the dual element $
\vee(T(a,b))=\pi ([\alpha_0,\compactldots,\alpha_{n-1}]\otimes [\beta_0,\compactldots,\beta_{n-1}])$ using \cite[Lemma 19]{CharltonRadchenkoRudenko}. For permutations $\sigma,\tau$ of  the set $\{0,\dots,n-1\}$ and $0\leq i,j\leq n-1$ we define subspaces
\[
\mathcal{F}_i^{\sigma}=\langle \alpha_{\sigma(0)},\dots,\alpha_{\sigma(i)} \rangle, \quad \text{and} \quad \scr{G}_j^{\tau}=\langle \beta_{\tau(n-j-1)},\dots, \beta_{\tau(n-1)}\rangle;
\]
The following identity holds in $(B^\rm{Com}\SSt)(V^{\vee})$:
\[
s([\alpha_0,\compactldots,\alpha_{n-1}]\otimes [\beta_0,\compactldots,\beta_{n-1}])=\sum_{\sigma,\tau\in\mathfrak{S}_n}(-1)^{\sigma}(-1)^{\tau}[\,\mathcal{F}_0^{\sigma}{\cap}\scr{G}_{n-1}^{\tau}\,|\,\mathcal{F}_1^{\sigma}{\cap}\scr{G}_{n-2}^{\tau}\,|\cdots|\,\mathcal{F}_{n-1}^{\sigma}{\cap} \scr{G}_{0}^{\tau}\,],
\]
where we set $[\,\mathcal{F}_0^{\sigma}{\cap} \scr{G}_{n-1}^{\tau}\,|\,\mathcal{F}_1^{\sigma}{\cap}\scr{G}_{n-2}^{\tau}\,|\cdots|\,\mathcal{F}_{n-1}^{\sigma}{\cap} \scr{G}_{0}^{\tau}\,]=0$ if the flags $\mathcal{F}^{\sigma}$ and $\scr{G}^{\tau}$ are not in general position. The lines $l_0,\dots, l_{2n-1}$ are in general position and so, in our case, the flags $\mathcal{F}^{\sigma}$ and $\scr{G}^{\tau}$ are in general position for any permutations $\sigma$ and $\tau$.

We use duality to rewrite the expression above using annihilators. We have
\begin{align*}
&\langle \alpha_{\sigma(0)},\dots, \alpha_{\sigma(i-1)}\rangle=a_{\sigma(i)}^{\perp}\cap \dots \cap a_{\sigma(n-1)}^{\perp}\\
&\langle \beta_{\tau(n-j)},\dots, \beta_{\tau(n-1)}\rangle= b_{\tau(0)}^{\perp}\cap \dots \cap b_{\tau(n-j-1)}^{\perp}.
\end{align*} 
Thus, for any $\sigma,\tau \in \mathfrak{S}_n$  and $0\leq k \leq n-1$ we have 
\[\mathcal{F}_{k}^{\sigma} \cap \scr{G}_{n-k-1}^{\tau}
 =\bigcap_{r=k+1}^{n-1}a_{\sigma(r)}^\perp \cap \bigcap_{s=0}^{k-1}b_{\tau(s)}^\perp  \subset V^{\vee};\]
we denote this line by $P_k^{\sigma,\tau}$.
Then we get
\[
s([\alpha_0,\dots,\alpha_{n-1}]\otimes [\beta_0,\dots,\beta_{n-1}])= \sum_{\sigma,\tau\in S_n}(-1)^{\sigma}(-1)^{\tau}
 [P_0^{\sigma,\tau}|\cdots|P_{n-1}^{\sigma,\tau}].
\]

Next, we compute the action of the operator $D_{a_{n-1},b_0}$ on $T(a,b)$. We have 
\begin{equation}\label{eqn: element to project}
\widetilde{D}_{a_{n-1},b_0}T(a,b)=\rm{D}^\FL_{a_{n-1}}T(a,b)-\rm{D}^\FL_{b_{0}}T(a,b) \in \FLR(V)
\end{equation}
and $D_{a_{n-1},b_0} T(a,b)$ is the projection of this element to $\PolyL_n(F)$. The map
\[
     H_{0}(\GL(V);\FIR(V)) \stackrel{D_*}{\lra}  H_{0}(\GL(V^{\vee});\FLR(V^{\vee}))
\]
is given by multiplication by $-1$, see \eqref{eqn: iso FI to FL}. Thus, the projection of \eqref{eqn: element to project} to $\PolyL_n(F)$ equals the projection of 
\[
-\Bigl(\rm{D}^\FI_{a_{n-1}} D(T(a,b))-\rm{D}^\FI_{b_{0}}D(T(a,b))\Bigr)\in \FIR(V^{\vee})
\] 
to $\PolyL_n(F)$. This projection can be computed by composing with the map $E_{b_0}$ defined in \cite[Section 5.2.2]{KRS1}, viewing $b_0$ as a functional on $V^{\vee}$. Observe that we have $\smash{E_{b_0}\circ \rm{D}^\FI_{b_{0}}} =0$ because evaluating at $b_0$ any element $\FC[0:h^1:\compactcdots:h^n]$ with $h^i(b_0)=1$ gives $\CorG(0,1,\compactldots,1)$ and this vanishes in $\PolyL_n(F)$. Thus we have
\[
D_{a_{n-1},b_0}\bigl(T(a,b)\bigr)=-E_{b_0}\circ \rm{D}^\FI_{a_{n-1}}D(T(a,b)).
\]

To evaluate this expression, we recall that by \cite[Lemma 8.1]{KRS1} the symbol of $D(T(a,b))$ differs from that for $\vee(T(a,b))$ by the sign $(-1)^{n-1}$:
\[
s(D(T(a,b)))=(-1)^{n-1}\sum_{\sigma,\tau\in S_n}(-1)^{\sigma}(-1)^{\tau}
 [P_0^{\sigma,\tau}|\cdots|P_{n-1}^{\sigma,\tau}].
\]
To compute $\rm{D}^\FI_{a_{n-1}}$ we omit the terms on the symbol which are not in general position with respect to the functional $a_{n-1}$. The surviving terms correspond to $\sigma, \tau$ such that 
\[
\bigcap_{r=k+1}^{n-1}a_{\sigma(r)}^\perp \cap \bigcap_{s=0}^{k-1}b_{\tau(s)}^\perp \not \subset a_{n-1}^{\perp} \quad \text{for} \quad k=0,\dots,n-1.
\] 
These are precisely the terms with $\sigma(0)=n-1$.
For such a term, let
$q_k^{\sigma,\tau}\in P_k^{\sigma,\tau}$ be normalised by
$q_k^{\sigma,\tau}(a_{n-1})=1$.  Explicitly,
\begin{equation}\label{eq:q-sigma-tau}
 q_k^{\sigma,\tau}(x)=
 \frac{\omega(a_{\sigma(k+1)},\ldots,a_{\sigma(n-1)},
                 b_{\tau(0)},\ldots,b_{\tau(k-1)},x)}
      {\omega(a_{\sigma(k+1)},\ldots,a_{\sigma(n-1)},
                 b_{\tau(0)},\ldots,b_{\tau(k-1)},a_{n-1})}
\end{equation}
and thus
\[
\rm{D}^\FI_{a_{n-1}}D(T(a,b))=(-1)^n(-1)^{n-1}\sum_{\substack{\sigma,\tau\in \mathfrak{S}_n\\\sigma(0)=n-1}}
  (-1)^{\sigma}(-1)^{\tau}
  \FI\bigl[q_0^{\sigma,\tau},\compactldots,q_{n-1}^{\sigma,\tau}\bigr].
\]
Put $x_k^{\sigma,\tau}=q_k^{\sigma,\tau}(b_0)\in F$.  From the above, we obtain 
\[
    D_{a_{n-1},b_0}([a_0,\compactldots,a_{n-1}]\otimes [b_0,\compactldots,b_{n-1}])
 =(-1)^n\sum_{\substack{\sigma,\tau\in \mathfrak{S}_n\\\sigma(0)=n-1}}
  (-1)^{\sigma}(-1)^{\tau}
  \CorG(0,x_0^{\sigma,\tau},\compactldots,x_{n-1}^{\sigma,\tau});
\]
 notice that the sign cancelled because of the difference between $\FI$ and $\FC$.
 
Next we observe that terms of the expression above with $\tau^{-1}(0)\leq n-3$ vanish. Indeed, if $k>\tau^{-1}(0)$, then $x_k^{\sigma,\tau}=0$, because the numerator in \eqref{eq:q-sigma-tau} contains $b_0$ twice. For $k\leq r$, the value depends only on the entries of $\tau$ preceding $0$.
If $r\leq n-3$, swapping two entries after $0$ reverses $\textup{sgn}(\tau)$ and leaves the
correlator unchanged, so these terms cancel. 

The remaining permutations $\tau$ satisfy $\tau(n-1)=0$ or $\tau(n-2)=0$. The permutations of the first type contribute
\[
(-1)^n\sum_{\substack{\sigma,\tau\in \mathfrak{S}_n\\\sigma(0)=n-1\\\tau(n-1)=0 }}
  (-1)^{\sigma}(-1)^{\tau}
  \CorG(0,x_0^{\sigma,\tau},\compactldots,x_{n-1}^{\sigma,\tau}).
\]
The permutations of the second type contribute
\[
(-1)^n\sum_{\substack{\sigma,\tau\in \mathfrak{S}_n\\\sigma(0)=n-1\\\tau(n-2)=0 }}
  (-1)^{\sigma}(-1)^{\tau}
  \CorG(0,x_0^{\sigma,\tau},\compactldots,x_{n-2}^{\sigma,\tau},0)
\]
because if $\tau(n-2)=0$ then $x_{n-1}^{\sigma,\tau}$ vanishes. Notice that the latter expression depends only on the value of $\tau$ on $0,\dots,n-3$. We pair each permutation with $\tau(n-2)=0$ with the permutation obtained by interchanging its last two positions.  We obtain
\begin{align*}
 &D_{a_{n-1},b_0}([a_0,\dots,a_{n-1}]\otimes [b_0,\dots,b_{n-1}])\\
 &=(-1)^n\sum_{\substack{\sigma,\tau\in \mathfrak{S}_n\\\sigma(0)=n-1\\\tau(n-1)=0}}(-1)^{\sigma}(-1)^{\tau}
  \Bigl(\CorG(0,x_0^{\sigma,\tau},\compactldots,x_{n-1}^{\sigma,\tau})-  \CorG(0,x_0^{\sigma,\tau},\compactldots,x_{n-2}^{\sigma,\tau},0)\Bigr)\\
  &=(-1)^n\sum_{\substack{\sigma,\tau\in \mathfrak{S}_n\\\sigma(0)=n-1\\\tau(n-1)=0}}(-1)^{\sigma}(-1)^{\tau}\ItG(0;0,x_0^{\sigma,\tau},\compactldots,x_{n-2}^{\sigma,\tau};x_{n-1}^{\sigma,\tau}).
\end{align*}

For any $\rho\in\mathfrak S_{\{0,\ldots,n-2\}}$ and  $\eta\in\mathfrak{S}_{\{1,\ldots,n-1\}}$ define permutations
\begin{align*}
&\sigma_\rho=
\begin{pmatrix}
  0 & 1 & \dots &n-1 \\
  n-1 & \rho(0) & \dots & \rho(n-2)
\end{pmatrix},\\
&\tau_\eta=
\begin{pmatrix}
  0 & 1 & \dots  &n-2 &n-1 \\
  \eta(n-1) & \eta(n-2)& \dots & \eta(1)&0
\end{pmatrix}.
\end{align*}

We have $\operatorname{sgn}(\sigma_\rho) =
(-1)^{n-1}\rm{sgn}(\rho)$ and $\rm{sgn}(\tau_\eta)
= (-1)^{\binom n2}\rm{sgn}(\eta)$ and thus 
\begin{align*}
 &D_{a_{n-1},b_0}(T(a,b))=(-1)^{\binom{n}{2}-1}\sum_{\rho,\eta}(-1)^{\rho}(-1)^{\eta}\ItG\bigl(0;0,x_0^{\sigma_\rho,\tau_\eta},\compactldots,x_{n-2}^{\sigma_\rho,\tau_\eta};x_{n-1}^{\sigma_\rho,\tau_\eta}\bigr).
\end{align*}

For $\rho=\eta=\mathrm{id}$, we have
\begin{equation}\label{eq:x-equals-R}
x_k^{\sigma,\tau}
=\frac{
 \omega(a_k,\ldots,a_{n-2},
        b_{n-1},\ldots,b_{n-k},b_0)
}{
 \omega(a_k,\ldots,a_{n-2},
        b_{n-1},\ldots,b_{n-k},a_{n-1})
}
=
R_k(a_0,\compactldots,a_{n-1},b_{n-1},\compactldots,b_0).
\end{equation}
For general $\rho$ and $\eta$, these are the same ordered ratios
after permuting $a_0,\ldots,a_{n-2}$ by $\rho$ and
$b_{n-1},\ldots,b_1$ by $\eta$. Hence
\begin{align}
D_{a_{n-1},b_0}(T(a,b))
&=
(-1)^{\binom n2-1}
\mathrm{Alt}_{\{0,\compactldots,n-2\}}
\mathrm{Alt}_{\{n,\compactldots,2n-2\}}
\ItG(0;0,R_0,\compactldots,R_{n-2};R_{n-1})
\nonumber\\
&=
(-1)^{\binom{n}{2}+n}
\GLiG_n(a_0,\compactldots,a_{n-1},
       b_{n-1},\compactldots,b_0).
\label{eq:Gamma-explicit}
\end{align}
The second equality uses the factor $(-1)^{n-1}$ in the iterated-integral formula for $\GLiG_n$.

Finally, we get the following formula proving the theorem:
\begin{align*}
\sf g_n(\ell_0,\compactldots,\ell_{2n-1})
&=
(-1)^{n-1}
(-1)^{n+\binom{n}{2}}
\GLiG_n(\ell_0,\compactldots,\ell_{n-1},
       \ell_{2n-1},\compactldots,\ell_n)\\
&=
(-1)^{\binom{n}{2} +1}
\GLiG_n(\ell_0,\compactldots,\ell_{n-1},
       \ell_{2n-1},\compactldots,\ell_n).\qedhere
\end{align*}
\end{proof}

\subsection{The dual cocycle equation}
Let $V$ be an $n$-dimensional vector space, and let
$\Conf_N(\mathbb P(V))$ be the space of generic ordered configurations
of $N$ points in $\mathbb P(V)$, modulo the action of $\PGL(V)$. Here \emph{generic} means that every $n$ points are in linear general position. If the spaces $V$ and $W$ are isomorphic, then the spaces $\Conf_N(\mathbb P(V))$ and $\Conf_N(\mathbb P(W))$ are canonically isomorphic: any two isomorphisms $V\to W$ differ by an element of
$\GL(W)$ and therefore induce the same map on equivalence
classes. To simplify the notation, we will denote the configuration $(x_0,\dots,x_{N-1})$ with $x_{i}=\langle v_i \rangle$ by $(v_0,\dots,v_{N-1})$. 

Let $\Gr^\circ(n,N)\subset \Gr(n,N)$ be the open locus in the Grassmannian on which all Pl\"ucker coordinates are nonzero.  A point of $\Gr^\circ(n,N)$ may be
represented by an $(n \times N)$-matrix
\[
M=[v_0\ \cdots\ v_{N-1}],
\]
of rank $n$, whose columns determine the points $x_i=[v_i]\in\mathbb P(V)$.  Left multiplication of $M$ by $\GL(V)$ changes the choice of basis
of $V$, whereas right multiplication by a diagonal matrix independently
rescales the chosen lifts $v_i$ of the projective points $x_i$. The Gelfand--MacPherson correspondence \cite[Section~2.2]{Kap93} identifies
\[
\Conf_N(\mathbb P(V))
\cong \Gr^\circ(n,N)/T_N,
\qquad
\text{with } T_N \coloneq \mathbb G_m^N/\mathbb G_m,
\]
where $T_N$ acts by independently rescaling the columns of $M$.

We have \emph{deletion maps} for $0\leq i \leq N-1$
\begin{align*}
a_i\colon \Conf_{N}(\mathbb P(V))&\lra \Conf_{N-1}(\mathbb P(V)) \\
(v_0,\dots,v_{N-1}) &\longmapsto (v_0,\dots,\widehat{v_i},\dots, v_{N-1}) \end{align*}
We also have projection maps for $0\leq i\leq N-1$
\begin{align*}
b_i\colon \Conf_N(\mathbb P(V))
&\lra
\Conf_{N-1}\bigl(\mathbb P(V/\langle v_i\rangle)\bigr) \\
(v_0,\ldots,v_{N-1}) &\longmapsto 
\bigl(\overline v_0,\dots, \hat{\overline v}_i,\dots, \overline v_{N-1}\bigr),
\end{align*}
where $\overline v_j$ is the image of $v_j$ in the quotient $V/\langle v_i\rangle$.

The \emph{association map} \cite[Section~2.3]{Kap93}, \cite[Section 7]{Gon95b} is a map
\[
 A_{n,N} \colon \Conf_N(\mathbb P^{n-1})
 \overset{\cong}\lra \Conf_N(\mathbb P^{N-n-1})
\]
defined as follows. Consider a matrix $M=[v_0\ \cdots\ v_{N-1}]$ representing a configuration $x\in\Conf_N(\mathbb P^{n-1})$. Then choose an $(N-n)\times N$ matrix
$M^\vee$ whose row space is $\ker(M)$, equivalently such that $
M(M^\vee)^{\mathsf t}=0$ and $\rm{rank}(M^\vee)=N-n$. The columns $v_0^\vee,\ldots,v_{N-1}^\vee$ of $M^\vee$ determine a
configuration of $N$ points in $\mathbb P^{N-n-1}$.  Its projective
equivalence class is independent of all choices and is called the
\emph{associate} of $x$. The association
exchanges the deletion and projection maps \cite[Section 7]{Gon95b}:
\begin{equation}\label{eq:association-deletion-projection-short}
 A_{N-n,N}A_{n,N}=\id,
 \qquad \text{and} \qquad
 A_{n,N-1} a_i=b_i A_{n,N}.
\end{equation}
\begin{remark}
Below we will study the action of the association $A_{n,2n}$ on the cluster Grassmannian polylogarithm.
Goncharov observed \cite[Proposition 7.1]{Gon95b} the association $A_{n,2n}$ agrees with the duality $\vee$ in the following sense: if a  configuration is represented by two bases $(a_0,\ldots,a_{n-1},b_0,\ldots,b_{n-1})$, then its associate is represented by the corresponding dual bases
$(\alpha_0,\ldots,\alpha_{n-1},\beta_0,\ldots,\beta_{n-1})$. 
\end{remark}

The cluster Grassmannian polylogarithm  $\GLiG_n$ can be viewed as a function from $\Conf_{2n}(\mathbb{P}^{n-1})$ to $\PolyL_n(F)$ sending  $x=(v_0,\dots,v_{2n-1})$ to $\GLiG_n(v_0,\ldots,v_{2n-1}).$ 

\begin{lemma} For a generic configuration 
$(v_0,\dots,v_{2n-1})\in \Conf_{2n}(\mathbb{P}^{n-1})$ we have
\begin{equation}\label{eq:GLi-dual-block-swap-short}
 \GLiG_n\bigl(A_{n,2n}(v_0,\compactldots,v_{2n-1})\bigr)
 =\GLiG_n(v_n,\compactldots,v_{2n-1},v_0,\compactldots,v_{n-1}).
\end{equation}
\end{lemma}
\begin{proof} Put $x=(v_0,\ldots,v_{2n-1})$ and choose representatives
\[
x^\vee=A_{n,2n}(x)
=(v_0^\vee,\ldots,v_{2n-1}^\vee).
\]
For a subset $S=\{s_1,\dots,s_{n-2}\}\subset\{0,\dots,2n-1\}$ and $a,b,c,d\in S^c$ we introduce the notation
\[
[S|a,b,c,d]_{x} \coloneq [v_{s_1},\dots,v_{s_{n-2}}|v_a,v_b,v_c,v_d].
\] 
We have the following well-known identity for projective cross-ratios involving the dual vectors:
\begin{equation}\label{eq:dual-cross-ratio-short}
[S|a,b,c,d]_{x^{\vee}}
 =[S^c\setminus\{a,b,c,d\}|a,b,c,d]_{x},
\end{equation} 
where $S^c$ is the complement of $S$ in $\{0,\dots,2n-1\}$.  \cref{eq:dual-cross-ratio-short} can be either deduced from the fact that the association exchanges deletion and projection operators or proven directly by comparing the minors of the matrices $M$ and $M^{\vee}$.

\cref{eq:dual-cross-ratio-short}  can be applied to the arguments $A_k$ in \cref{sec:cluster grassm polylogs}.  It gives
\begin{align*}
 A_k(x^\vee)
 ={}&[v_0,\ldots,v_{k-1},v_{n+k+1},\ldots,v_{2n-2}\mid
       v_k,v_{n-1},v_{n+k},v_{2n-1}].
\end{align*}
This is precisely the $k$-th argument of 
for the configuration obtained from $x$ by interchanging its two blocks of
$n$ vectors. The same statement remains true after the two alternations. From here the statement follows.
\end{proof}

\begin{proposition}\label{prop:duality-goncharov-cocycle-short}
For every generic configuration $x\in\Conf_{2n}(\mathbb P^{n-1})$,
\[
 \sf{g}_n\bigl(A_{n,2n}(x)\bigr)=(-1)^n  \sf{g}_n(x).
\]
\end{proposition}

\begin{proof}
Put $x=(v_0,\ldots,v_{2n-1})$ and
$x^{\mathrm{rev}}=(v_{2n-1},\ldots,v_0)$.  The explicit formula of Section~\ref{thm:explicit formula for the cocycle on generic configurations} expresses $\sf{g}_n(x)$, up to a fixed sign, as
\[
 \GLiG_n(v_0,\ldots,v_{n-1},v_{2n-1},\ldots,v_n).
\]
Association commutes with relabelling. Applying \eqref{eq:GLi-dual-block-swap-short} to this ordered tuple therefore gives
\[
 \sf{g}_n\bigl(A_{n,2n}(x)\bigr)=\sf{g}_n(x^{\mathrm{rev}}).
\]
The definition of $\sf{g}_n$ gives the evident reversal symmetry $\sf{g}_n(x^{\mathrm{rev}})=(-1)^n \sf{g}_n(x)$. Indeed, reversal interchanges the two Steinberg factors, contributing $(-1)^{n-1}$ by the proof of \cite[Lemma 8.1]{KRS1} and reverses the ordered difference $D_{\ell_{n-1},\ell_n}$, contributing one further minus sign.  This proves the claim.
\end{proof}

We can now deduce the dual cocycle equation. 
\begin{corollary}\label{cor:gn-dual-cocycle-equation} For every generic configuration $x\in\Conf_{2n+1}(\mathbb P^n)$, the Goncharov cocycle  $\sf{g}_n$ satisfies the dual cocycle equations
\[
\sum_{i=0}^{2n}(-1)^i \sf{g}_n(b_i x)= 0. 
\]
\end{corollary}
\begin{proof}
Consider the generic configuration $y=A_{n+1,2n+1}(x) \in \Conf_{2n+1}(\mathbb P^{n}).$ By \eqref{eq:association-deletion-projection-short}, we have $b_i x=A_{n,2n}(a_i y)$ and thus, by the cocycle equation for $\sf{g}_n$ of \cref{cor:gn-cocycle-equation}, we get
\[
\sum_{i=0}^{2n}(-1)^i \sf{g}_n(b_i x)=  (-1)^n\sum_{i=0}^{2n}(-1)^i \sf{g}_n(a_i y)=0. \qedhere
\]
\end{proof}

\section{A cocycle for the Borel class and the proof of \cref{thm:polyl-iso-number-field}} Recall that for a number field $F$, the motivic realisation is a map of Lie coalgebras
\[r^\rm{MTM} \colon \scr{G}(F) \lra \scr{L}^\rm{MTM}(F)\]
whose domain and target are abstractly isomorphic: both are cofree with cogenerators given by the (rationalised) algebraic $K$-theory groups $K_*(F)_\Q$ for $*>0$: the content of \cref{thm:polyl-iso-number-field} is that the particular map $r^\rm{MTM}$ is an isomorphism. To prove this, we need to use how Borel computed $K_*(F)_\Q$, using regulators. This uses a polylogarithmic interpretation of the Borel regulators and the Borel classes. Composing the Goncharov class with the real period map $p_\bb{R} \colon \scr{G}_n(\C) \to \bb{R}$ gives a cohomology class 
\[\rm{Bo}^\scr{G}_{n,2n-1} \coloneq p_\bb{R} \circ \rm{g}_n \in H^{2n-1}(\GL_n(\C);\bb{R}).\] 
We prove below that this is a nonzero multiple of the Borel class and hence gives a construction of the Borel regulator map on a part of the rank filtration. Proving that it is a multiple amounts to verifying that the cocycle is measurable using our formula and properties of Hodge correlators, but that it is a \emph{nonzero} multiple is interwoven with the proof of \cref{thm:polyl-iso-number-field}.

\subsection{Real periods and Hodge correlators} \label{sec: real periods}
In this section we discuss the \emph{real period map} 
\begin{align*}p_{\R}\colon \PolyL_n(\C) &\lra \R \\
\CorG(x_0,\compactldots,x_n) &\longmapsto \CorH(x_0,\compactldots,x_n,)\end{align*}
where $\CorH$ is the Hodge correlator introduced by Goncharov \cite{Gon19}; see \cite{Mal20,KT24} for other expositions. (We use $\CorH$ instead of $\rm{Cor}^\rm{Hod}$ to distinguish from the Hodge realisation of $\CorG$ in $\scr{L}^\rm{Hod}$.) The real period map is a composition of the Hodge realisation map $r^{\rm{Hod}}$ constructed in \cite[Theorem D.b]{KRS1} and the canonical real period map constructed by Goncharov \cite[Section 1.11]{Gon19}. The normalisation of the Hodge correlator varies in the literature, which can lead to confusion. For the results of this paper, the particular choice of normalisation does not play a role. We prefer not to fix it, and, as a result, the real period map depends on a choice of constants $\lambda_n\in \R^\times$ for $n\in \N$. We will discuss more on normalizations in \cref{sec:normalisation}.

\subsubsection{Definition of the Hodge correlator}
Consider a disc with $n+1$ marked points on the boundary labelled by complex numbers $x_0,\dots,x_n \in \C$ oriented clockwise. Let $\cal{T}_{n+1}$ be the set of isotopy classes of plane trivalent trees embedded in the disc with leaves labelled by the marked points $x_0,\dots,x_n$.  A tree $T\in \cal{T}_{n+1}$ has $n-1$ internal vertices and $2n-1$ edges. 

An orientation of a tree is a choice of ordering of the edges up to an even permutation. A plane trivalent tree has a canonical orientation, which can be computed as follows: choose an external edge as a root and order the edges by their first appearance while traversing clockwise the boundary of a small regular neighbourhood of $T$. Changing the root cyclically permutes the $2n-1$ edges. Since $2n-1$ is odd, such a cyclic permutation has sign
$(-1)^{2n-2}=1$. Hence the resulting orientation is independent
of the chosen root, c.f.~\cite[Lemma 4.3.1]{KT24}.

Consider additional variables $x_{n+1},\dots,x_{2n-1}\in \C$ and label the internal vertices of the tree by these labels. For an edge $E$ of a tree with endpoints $x_i$ and $x_j$ consider a function 
\[
f_E=\log|x_i-x_j|.
\]

Choose for each $n\in \N$ a nonzero constant  $\lambda_n \in \R$. The Hodge correlator is defined as the following sum of absolutely convergent integrals
\begin{equation}\label{eqn:hodge correlator via log abs}
    \CorH(x_0,\dots, x_n)=\lambda_n \sum_{T\in \mathcal{T}_{n+1}}\int_{\C^{n-1}} f_{E_1} d f_{E_2}\wedge \dots \wedge d f_{E_{2n-1}}.
\end{equation} where we integrate the ``internal'' variables $x_{n+1},\ldots,x_{2n-1}$. Note that each of the integrals in the sum does not depend on the order of internal vertices and on the order of edges with the given orientation class. 

\begin{remark} Goncharov often gives a different definition of Hodge correlators, which is based on the differential form $\omega$, see \cite[11.1.1]{Gon19}. This definition can be reduced to \eqref{eqn:hodge correlator via log abs} thanks to the integral  identity in the proof of \cite[Lemma 11.2]{Gon19}.
\end{remark}

\begin{example} For $n=2$ we have an identity
\begin{equation}\label{eqn: integral bloch-wigner}
    \int_{\C} \log|1-x| d \log|x| \wedge d \log|x-a|=\pi D(a),
\end{equation} 
where $D(z) =\Im(Li_2(z))+\log|z|\arg(1-z)$ is the \emph{Bloch-Wigner dilogarithm}. Thus, we have 
\[
\CorH(1,0,a)=\lambda_2 \int_{\C} \log|1-x| d \log|x| \wedge d \log|x-a|= \lambda_2 \pi D(a).
\]
\end{example}

\begin{proposition}\label{prop:hodge-correlator-cts-bounded} For every $n\geq 2$, the Hodge correlator $\CorH(x_0,\dots,x_n)$ is continuous and bounded on
$\C^{n+1}\setminus\{x_0=\cdots=x_n\}$. 
\end{proposition}
\begin{proof}Malkin proved \cite[Theorem 11]{Mal20} that Hodge correlators are continuous away from the diagonal $x_0=\dots=x_n$. For $n\geq 2$, the Hodge correlators are homogeneous: we have
\[
\CorH(x_0,\compactldots, x_n)=\CorH(ax_0,\compactldots, ax_n) \text{ for }a\in \C^{\times}.
\]
Thus the function $\CorH(0,x_1,\dots, x_n)$ restricts to a continuous function on $\bb{P}^{n-1}(\C)$ and thus is bounded. The statement of the proposition follows from the translation invariance of Hodge correlators.
\end{proof}

\subsubsection{Real period of a classical polylogarithm} \label{sec: real period classical polylogs}
The real period of the classical polylogarithm $\LiG_n(a)$ can be computed explicitly, with the use of a result of Levin \cite[Section 4]{Levin}. To state it, we first recall a single-valued version of classical polylogarithm introduced by Zagier  \cite[\S 7]{Zag90}. Let $\pi_n(x)$ denote $\Re(x)$ if $n$ is odd and $i \Im(x)$ if $n$ is  even. Next, define the coefficients \(\beta_k\)  by the Taylor expansion
\[
\frac{2x}{e^{2x}-1}=\sum_{k\ge 0}\beta_k\,x^k.
\]
The function
\[
\mathcal{L}_n(z) \coloneq \pi_n \left(\sum_{k=0}^{n-1}
\beta_k\big(\log |z|\big)^k\,\rm{Li}_{\,n-k}(z)\right)
\]
is continuous, and extends continuously to $\bb{P}^1_{\C}$ by $\mathcal{L}_n(\infty)=0$ \cite[\S 7]{Zag90}.  Levin introduced {\it modified single-valued polylogarithm}
\[
\mathcal{L}_n^{\#}(z)=4^{-(n-1)}\sum_{\substack{0\le k\le n-2 \\ k\ \text{even}}}
\binom{2n-k-3}{n-1}\,
\frac{2^{\,k+1}}{(k+1)!}\,
\mathcal{L}_{\,n-k}(z)\,
\big(\log|z|\big)^k\,
\]
and proved the following identity \cite[Section 4.4]{Levin}:
\begin{align*}
    &\int_{\C^{n-1}}\log|1-x_1|\Lambda_{j=1}^{n-2}\Bigl( d\log|x_j|\wedge d\log|x_j-x_{j+1}| \Bigr)\wedge d \log|x_{n-1}|\wedge d \log|x_{n-1}-z|\\
    &= -(2\pi i)^{n-1}\mathcal{L}_n^{\#}(z).
\end{align*}
The reader is encouraged to check that for $n=2$ this identity is equivalent to \eqref{eqn: integral bloch-wigner}.

\begin{lemma}\label{lem: real period of classical polylogarithm}
For $n\geq 2$ and $a\in \C$ we have
\[
   p_{\R}(\LiG_n(a)) = \lambda_n (2\pi i)^{n-1}\mathcal{L}_n^{\#}(a).
\]
\end{lemma}
\begin{proof} Observe that there exists a unique isotopy class of a tree for which the functions $f_{E_1}  f_{E_2},  \dots, f_{E_{2n-1}}$ are pairwise distinct. For this tree, the integral 
\[
\int_{\C^{n-1}} f_{E_1} d f_{E_2}\wedge \dots \wedge d f_{E_{2n-1}}
\]
agrees with the integral computed by Levin. Thus we have 
\begin{align*}
   p_{\R}(\LiG_n(a)) &=p_{\R}\bigl(-\CorG(1,\underbrace{0,\compactldots,0}_{n-1},a)\bigr)= \lambda_n (2\pi i)^{n-1}\mathcal{L}_n^{\#}(a).\qedhere
\end{align*}
\end{proof}

Note that if $a=e^{i\theta}$ the modified single-valued polylogarithm simplifies and we get
\[
\mathcal{L}_n^{\#}(a)= \begin{cases} 2^{-2n+3}\binom{2n-3}{n-1}i\sum_{m\geq 1}\dfrac{\sin(m\theta)}{m^n} & \text{for even $n$,} \\
2^{-2n+3}\binom{2n-3}{n-1}\sum_{m\geq 1}\dfrac{\cos(m\theta)}{m^n} & \text{for odd $n$.}\end{cases}\] 
Thus \cref{lem: real period of classical polylogarithm} implies the following:
\begin{corollary}\label{cor:value-of-real-period-on-unit-norm}
The real period of $\LiG_n(i)$
is nonzero.
\end{corollary}

\subsection{The Borel class and the Borel regulator}

\subsubsection{Continuous and measurable cohomology} \label{sec:continuous-and-measurable-cohomology} One can compute the cohomology of a (discrete) group $G$ with $\bb{R}$-coefficients using the inhomogeneous bar resolution, that is, as the cohomology $H^*(G;\bb{R})$ of the cochain complex with $C^p(G;\bb{R}) \coloneq \rm{Fun}(G^p,\bb{R})$ with differential
\[(df)([g_0|\compactldots|g_{p}]) = f([g_1|\compactldots|g_p])+ \sum_{i=1}^p (-1)^{i+1} f([g_0|\compactldots|g_ig_{i+1}|\ldots|g_p])+(-1)^{p+1} f([g_0|\compactldots|g_{p-1}]).\] 
If a $G$ is a topological group, then we define the \emph{continuous cohomology} $H^*_\rm{cts}(G;\bb{R})$ as the cohomology of the subcomplex $C^p_\rm{cts}(G;\bb{R}) \subset C^p(G;\bb{R})$ given by the continuous functions. 

In practice, one can construct continuous cocycles by improving the regularity of measurable cocycles. Following \cite[Section 2]{AustinMoore}, measurable cohomology is defined for $G$ a locally compact second countable topological group. Fixing $\mu_G$ to be a left-invariant Haar measure, we let $C^p_\rm{meas}(G;\bb{R})$ be the set of almost everywhere equivalence classes of Borel measurable functions $G^p \to \bb{R}$ with differential given by the same formula as above; the \emph{measurable cohomology} $H^*_\rm{meas}(G;\bb{R})$ is given by the cohomology of this complex. 

Since every continuous function is Borel measurable, there is a natural map
\[H^*_\rm{cts}(G;\bb{R}) \lra H^*_\rm{meas}(G;\bb{R})\]
and this is an isomorphism \cite[Theorem A]{AustinMoore}. It is sometimes convenient to take the smaller subcomplex of almost everywhere equivalence classes of that are essentially bounded on compact subsets, as one can then directly use a regularisation argument to prove the map from continuous cohomology to its cohomology is an isomorphism \cite[\S 4]{Blanc}, cf.~\cite[Proposition 2.4]{MonodSLn} in the bounded case.

\subsubsection{A cocycle for the Borel class} The goal of this section is to compare two cohomology classes in $H^{2n-1}(\GL_n(\C);\R)$ given by (using the universal coefficients theorem)
\begin{align*}&\rm{Bo}_{n,2n-1} \colon H_{2n-1}(\GL_n(\C);\R) \xrightarrow{- \cap \ol{c}_{n,2n-1}} \R \\
&\rm{Bo}^\scr{G}_{n,2n-1} \colon H_{2n-1}(\GL_n(\C);\R) \xrightarrow{- \cap \rm{g}_n} \scr{G}_n(\C)_\bb{R} \overset{p_\R}\lra \R.\end{align*}
The latter is the composition of the Goncharov class with the real period map. The former is the \emph{Borel} class given by the cap product against the continuous cohomology class $\ol{c}_{n,2n-1}$, where we recall there are isomorphisms by the van Est isomorphism followed by \cref{lem:relative-lie-algebra-computation-gl-u}
\[H^*_\rm{cts}(\GL_n(\C);\R) \cong H^*(\mathfrak{gl}_n,\mathfrak{u}_n;\R) \cong H^*(\rm{U}(n);\R) \cong \Lambda^*_\R \langle \ol{c}_{n,1},\dots,\ol{c}_{n,2n-1} \rangle.\]

\begin{proposition} \label{prop:multiple}
    For every $n \ge 1$ there exists a constant $\mu_n \in \R$ such that $\rm{Bo}^\scr{G}_{n,2n-1} = \mu_n \cdot \rm{Bo}_{n,2n-1}$. 
\end{proposition}

This is proven by first showing that $\rm{Bo}_{n,2n-1}^\scr{G}$ lifts to a continuous cohomology class, and observing that it vanishes when restricted to $\GL_{n-1}$ but $\ol{c}_{n,2n-1}$ spans the kernel of this restriction map. We will later, in \cref{thm:polyl-iso-mun}, show the constants $\mu_n$ are nonzero. For the first step, we use that \cref{thm:goncharov-class} implies:

\begin{proposition}\label{prop:borel-class-representative}
    The cohomology class $\rm{Bo}^\scr{G}_{n,2n-1}$ is represented by the inhomogeneous cocycle
    \begin{align*} \GL_n(\C)^{2n-1} &\lra \bb{R} \\
    [g_1|\compactldots|g_{2n-1}] &\longmapsto (-1)^{n-1} p_\R \big(D_{\ell_{n-1},\ell_n}(\pi ([\ell_0,\compactldots,\ell_{n-1}] \otimes [\ell_n,\compactldots,\ell_{2n-1}])) \big),\end{align*}
    where $\ell_i \coloneq g_1 \compactcdots g_i \cdot \ell$ for a fixed line $\ell \in \bb{P}(\C^n)$.
\end{proposition}

\begin{lemma}\label{lem:borel-cocycle-continuous}
    The cohomology class $\rm{Bo}^\scr{G}_{n,2n-1}$ lifts to a continuous cohomology class, that is, along the forgetful map $H^{2n-1}_\rm{cts}(\GL_n(\C);\R) \to H^{2n-1}(\GL_n(\C);\R)$.
\end{lemma}

\begin{proof}The formula in \cref{prop:borel-class-representative} defines a Borel measurable function on $\GL_n(\C)^{2n-1}$ since it is continuous on the set of tuples $(g_1,\compactldots,g_{2n-1}) \in \GL_n(\C)^{2n-1}$ such that the lines $(\ell_0,\ldots,\ell_{2n-1})$ are in general position, whose complement has zero measure with respect to the Haar measure as a finite union of codimension one submanifolds. To see this, observe that on this subset the decomposition operator has a fixed form so our function is a finite linear sum of Hodge correlators applied to rational functions of the entries of the $g_i$, and Hodge correlators are continuous by \cref{prop:hodge-correlator-cts-bounded}. Thus its cohomology class lifts to measurable cohomology and as we saw above, the inclusion map from continuous to measurable cohomology is an isomorphism, hence giving the result. 
\end{proof}

For the second step, we start with a computation using the work of Borel--Yang, see \cref{section: borel-yang}:

\begin{lemma}\label{lem:gln-c-cts-injective} The canonical map
    $H_\rm{cts}^*(\GL_n(\C);\R) \to H^*(\GL_n(\C);\R)$ is injective.
\end{lemma} 

\begin{proof}
To see this, consider the standard embedding $\Q(i) \hookrightarrow \C$ and note it suffices to show that the composition 
        \[H_\rm{cts}^{*}(\GL_n(\C);\R) \lra H^{*}(\GL_n(\C);\R) \lra H^{*}(\GL_n(\Q(i));\R)\]
is injective. From \cite[Theorem 2.1]{borelyang} it follows that the composition 
         \[H_\rm{cts}^{*}(\SL_n(\C);\R) \lra H^{*}(\SL_n(\C);\R) \lra H^{*}(\SL_n(\Q(i));\R)\]
        is an isomorphism. Thus, $H^{*}(\SL_n(\Q(i));\R) \cong \Lambda_\bb{R} \langle \overline{c}_{n,3},\dots,\overline{c}_{n,2n-1} \rangle$ as algebras. Next, by Lemma 7.2 loc.cit., the units $\Q(i)^\times$ act trivially on the classes $\ol{c}_{n,2i-1}$ for $2 \le i \le n$.  Moreover, Theorem 7.5 loc.cit.~says that the following spectral sequence collapses
        \[E^2_{p,q}= H^p(\Q(i)^\times,H^q(\SL_n(\Q(i));\R)) \Longrightarrow H^{p+q}(\GL_n(\Q(i));\R),\]
        so gives an isomorphism $H^*(\GL_n(\Q(i));\R) \cong \rm{Hom}_\R(\Lambda_\R^*\Q(i)^\times_\R, \R) \otimes_\R \Lambda_\R^*\langle \ol{c}_{n,3},\cdots,\ol{c}_{n,2n-1}\rangle$ of algebras. Finally, the class $\overline{c}_{n,1}$ maps to the class in $\rm{Hom}_\R(\Q(i)^\times_\R, \R)$ given by $z \mapsto \log(|z|)$ using \cite{Hamida} or because this formula defines a continuous cocycle in $C^1_\rm{cts}(\GL_1(\C);\R)$ which is nonzero, so must be a (nonzero) multiple of $\ol{c}_{n,1}$. Injectivity then follows since the class $z \mapsto \log(|z|)$ is nonzero in $\rm{Hom}_\Z(\Q(i)^\times, \R)$.  This completes the proof of the claim
\end{proof}

We now prove that $\rm{Bo}_{n,2n-1}^\scr{G}$ is a multiple of $\rm{Bo}_{n,2n-1}$.

\begin{proof}[Proof of \cref{prop:multiple}]
    By construction, $\smash{\rm{Bo}_{n,2n-1}^\scr{G}}$ factors via the canonical map through $H_{n,2n-1}^{E_\infty}(\BGL(\C)_\R)= \scr{G}_n(\C)_\R$ and hence through $H_{n,2n-1}(\BGL(\C)_\R/(\sigma))$ naturally isomorphic to $H_{2n-1}(\GL_n(\C),\GL_{n-1}(\C);\R)$. It thus suffices to prove more generally that
    \[\alpha = \mu\cdot \ol{c}_{n,2n-1} \qquad \text{for some $\mu \in \R$},\]
    if the image of $\alpha \in H^{2n-1}_\rm{cts}(\GL_n(\C);\R)$ under the forgetful map to $H^{2n-1}(\GL_n(\C);\R) \cong \rm{Hom}_\R(H_{2n-1}(\GL_n(\C);\R),\R)$ factors through
    $H_{2n-1}(\GL_n(\C),\GL_{n-1}(\C);\R)$.
    
By the universal coefficients theorem, a class in $H^{2n-1}(\GL_n(\C);\R)$ factors through the group $H_{2n-1}(\GL_n(\C),\GL_{n-1}(\C);\R)$ if and only if it is in the image of $H^{2n-1}(\GL_n(\C), \GL_{n-1}(\C);\R)$, which by the long exact sequence of a pair in cohomology is equivalent to it vanishing in $H^{2n-1}(\GL_{n-1}(\C);\R)$. Thus, if we map $\alpha$ to $H^{2n-1}(\GL_n(\C);\R)$ and then restrict it to $H^{2n-1}(\GL_{n-1}(\C);\R)$ it vanishes. By \cref{lem:gln-c-cts-injective}, $\alpha$ must vanish when restricted to $H_\rm{cts}^{2n-1}(\GL_{n-1}(\C);\R)$. The result follows from the following, obtained from the description of the restriction map obtained from the dual of \cref{lem:cts-cohomology-maps}
    \[\ker\big[H^{2n-1}_\rm{cts}(\GL_n(\C);\R) \to H^{2n-1}_\rm{cts}(\GL_{n-1}(\C);\R)\big]= \R \langle \ol{c}_{n,2n-1} \rangle.\qedhere\]
\end{proof}

\subsubsection{A formula for the Borel regulator} The goal of this section is to compare two regulator maps (we continue working rationally implicitly):
\[\begin{aligned} &\rm{reg}^\rm{Bo}_n \colon K_{2n-1}(\bb{C}) \xrightarrow{\rm{inc}} H_{2n-1}(\BGL_\infty(\bb{C})) \xrightarrow{(-) \cap \ol{c}_{\infty,2n-1}} \bb{R}, \\
&\rm{reg}^\scr{G}_n \colon K_{2n-1}(\bb{C}) \xrightarrow{\rm{edge}_n} \scr{G}_n(\bb{C}) \xrightarrow{p_\bb{R}} \bb{R}.\end{aligned}\]
The top map is the classical Borel regulator \cite{BorelReg,BorelRegErrata} obtained by identifying algebraic $K$-theory with the primitives in the stable homology and taking the cap product against the stable class corresponding to the continuous cohomology classes $\ol{c}_{N,2n-1}$ for $N \geq n$. The bottom map is the edge homomorphism composed with the real period map. We shall prove that these agree up to a multiple on the $n$th step of the indecomposable rank filtration, though it is reasonable to expect they agree on all of $K_{2n-1}(\bb{C})$. We will later, in \cref{thm:polyl-iso-mun}, show that the constants $\mu_n$ are nonzero. 

\begin{proposition}\label{prop:borel-regulator-multiple}
For every $n \geq 1$ we have $\rm{reg}_n^\scr{G} = \mu_n \cdot \rm{reg}_n^\rm{Bo}$ on the subspace $\fil^\rm{ind}_n K_{2n-1}(\C) \subseteq K_{2n-1}(\C)$ for $\mu_n \in \bb{R}$ as in \cref{prop:multiple}. 
\end{proposition}

\begin{proof}
    If $\ol{\alpha} \in \fil^\rm{ind}_n K_{2n-1}(\C)_\bb{Q}$ then by \cref{defn rank filtrations}(ii), there exists $\alpha_0 \in H_{2n-1}(\BGL_n(\C);\bb{Q})$ whose stabilisation $\alpha$ maps to $\overline{\alpha}$ under the canonical map from homology to $E_\infty$-homology, identified with algebraic $K$-theory. We have that 
    \begin{align*}\rm{reg}^\scr{G}_n(\alpha) &= p_\R(\rm{edge}_n(\alpha)) = p_\R(\rm{can}_n(\alpha_0)) = p_\R(\alpha_0 \cap \rm{g}_n) \\
    &= \mu_n \cdot (\alpha_0 \cap \ol{c}_{n,2n-1}) =\mu_n\cdot (\alpha \cap \ol{c}_{\infty,2n-1}) =\mu_n \cdot \rm{reg}^\rm{Bo}_n(\alpha).\end{align*}
    The second equation follows from \cref{lem:edge-homomorphism-in-rank-n} for $\bf{R}=\BGLb(\C)_\bb{Q}$. The third equation uses that the canonical map is implemented by capping with the Goncharov class $\sf{g}_n$ by \cref{prop:canonical-is-cap-gamman}. The fourth equation uses that $\rm{Bo}^\scr{G}_{n,2n-1} = \mu_n \cdot \rm{Bo}_{n,2n-1}$ by \cref{prop:multiple}, after the definitions. The fifth equation uses that $\alpha_0$ stabilises to $\alpha$ and $\rm{Bo}_{n,2n-1}$ is obtained from $\overline{c}_{\infty,2n-1}$ by restriction to rank $n$.
\end{proof}

\begin{remark} \label{rem:goncharov-regulator-and-borel-regulator} 
    If a subfield $F$ of $\C$ satisfies $\fil^\rm{ind}_n K_{2n-1}(F) = K_{2n-1}(F)$, e.g.~if $F$ is a number field, then the Borel and Goncharov regulators agree up to the multiple $\mu_n$ on all of $K_{2n-1}(F)$.
\end{remark}

\subsection{Proof of \cref{thm:polyl-iso-number-field}} We now have gathered all tools needed to prove the main result of this paper, \cref{thm:polyl-iso-number-field}, though we state a stronger variant that also implies \cref{thm:borel-cocycles}.

\begin{theorem}\label{thm:polyl-iso-mun}
Let $F$ be a number field, then the $\mu_n \in \R$ of \cref{prop:multiple} are nonzero for every $n \ge 1$, and the canonical map 
    \[r^\rm{MTM} \colon \scr{G}(F) \lra \scr{L}^\rm{MTM}(F)\]
    is an isomorphism of graded Lie coalgebras.
\end{theorem}

\begin{proof}By induction on $n$ we prove that the maps $r^\rm{MTM} \colon \scr{G}_k(F) \to \scr{L}^\rm{MTM}_k(F)$ are isomorphisms for $k \leq n$ and that $\mu_k \neq 0$ for $k\leq n$.

For the base case $n=1$, we first use that $\rm{Bo}_{1,1}$ and $\rm{Bo}^\scr{G}_{1,1}$ are both represented, up to nonzero constants, by $[g] \in C_1(\GL_1(F)) \mapsto \log|g|$ so $\mu_1 \neq 0$; for the former this is \cite{Hamida} and for the latter this is by construction. Next we use that the isomorphism $F^\times \cong \scr{L}^\rm{MTM}_1(F)$ given by $(x_0-x_1) \mapsto \rm{Cor}^\rm{MTM}(x_0,x_1)$ of \cite[Proposition 50]{CMRR24} is compatible with the isomorphism $F^\times \cong \scr{G}_1(F)$ of \cite{KRS1}.

For the induction we suppose that we have proven the statement for $n-1$, and will prove it for $n$. We first verify that the constant $\mu_n$ is nonzero. For that, it suffices by \cref{prop:borel-regulator-multiple} to prove that the composition 
 \[K_{2n-1}(\Q(i)) \overset{j}\lra K_{2n-1}(\C) \xrightarrow{\rm{edge}_n} \scr{G}_n(\C) \xrightarrow{p_\R} \R\]
with left-most map induced by the standard inclusion $j \colon \Q(i) \hookrightarrow \C$, is nonzero. To do so, consider the element $ \LiG_n(i) = -\CorG(1,0,\compactldots,0,i) \in \scr{G}_n(\Q(i))$, which has nonzero real period by \cref{cor:value-of-real-period-on-unit-norm}. 
Its cobracket in $\scr{L}^\rm{MTM}_{n-1}(F) \otimes F^\times$ (or $\Lambda^2 F^\times$ for $n=2$) satisfies (take \cite[(25)]{CMRR24} modulo products)
\[(\Lambda^2 r^\rm{MTM})\Big(\delta\, \LiG_n(i)\Big)= \delta^\rm{MTM} \, \rm{Li}^\rm{MTM}_n(i) = \begin{cases}
   \rm{Li}^\rm{MTM}_{n-1}(i) \otimes (i) & \text{if $n \geq 3$}, \\
   (i) \wedge (1-i) & \text{if $n=2$,}
\end{cases}\]
which vanishes in both cases since $(i)=0 \in F^\times$ using that $i^4=1$ and we implicitly rationalise. By our induction hypothesis, $r^\rm{MTM}$ is an isomorphism in weights less than $n$ and hence we must have $\delta\,\LiG_n(i)=0 \in (\Lambda^2 \PolyL(F))_n$. Thus, by \cref{thm:number-fields-goncharov} we have that $\rm{Li}_n^\scr{G}(i) \in \rm{im}(\rm{edge}_n)$ and it lifts to $K_{2n-1}(\Q(i))$, giving the required result. 

Next we will prove that $r^\rm{MTM} \colon \PolyL_n(F) \to \MotCoLie_n(F)$ is an isomorphism. By the inductive hypothesis it suffices to prove that the map
\[\begin{tikzcd} \ker(\delta)_n \coloneq \ker\big[\delta \colon \scr{G}_n(F) \to (\Lambda^2 \scr{G}(F))_n\big] \dar \\[-10pt] \ker(\delta^\rm{MTM})_n \coloneq \ker\big[\delta \colon  \scr{L}^\rm{MTM}_n(F) \to (\Lambda^2 \scr{L}^\rm{MTM}(F))_n\big]\end{tikzcd}\]
is an isomorphism. There are maps
\[K_{2n-1}(F) \xrightarrow{\rm{edge}_n} \ker(\delta)_n \xrightarrow{r^\rm{MTM}} \ker(\delta^\rm{MTM})_n \lra \ker(\delta^\rm{Hod})_n \subseteq \scr{L}^\rm{Hod}_n \lra \bb{R}^{r_1+r_2}\]
where for the third map we use that by \cite[Section 5.2]{CMRR24} Hodge realisation factors over motivic realisation, and for the fourth map we apply the real period map $p_\bb{R}$ for all $r_1$ real embeddings of $F$ and $r_2$ pairs of conjugate complex embeddings of $F$. It follows from \cref{prop:borel-regulator-multiple} that this is a \emph{nonzero} multiple of the Borel regulators for all real and complex embeddings and hence is injective by \cite{BorelReg,BorelRegErrata}.

We now recall that $H^1(\scr{G}(F))_n = \ker(\delta)_n$ is isomorphic to $K_{2n-1}(F)_\Q$ because $\scr{G}(F)$ is cofree with cogenerators $K_*(F)_\Q$ for $*>0$ by \cref{thm:number-fields-goncharov}. Similarly $H^1(\scr{L}^\rm{MTM}(F))_n = \ker(\delta^\rm{MTM})_n$ is isomorphic to $K_{2n-1}(F)_\Q$ because $\scr{L}^\rm{MTM}(F)$ is cofree with cogenerators $K_*(F)_\Q$ for $*>0$ by \cite[Section 2]{DG05}. For dimension reasons it must hence be the case that $r^\rm{MTM} \colon \ker(\delta)_n \to \ker(\delta^\rm{MTM})_n$ is an isomorphism.\end{proof}

\begin{corollary}
    If $F$ is a number field then the formal realisation is an isomorphism of Lie coalgebras
    \[r^\rm{f} \colon \PolyL(F) \overset{\cong}\lra \PolyLF(F).\]
\end{corollary}

\begin{proof}
    Motivic realisation $r^\rm{MTM}$ factors through formal realisation, and hence $r^\rm{f}$ must be injective by the previous theorem. As $r^\rm{f}$ is surjective by definition, see \cite[Theorem 7.28]{KRS1}, we obtain the result. 
\end{proof}

In particular, all functional equations that are known to hold for $\PolyLF(F)$, see \cite{CMRR24}, hold in $\PolyL(F)$ for $F$ a number field. 

\subsection{Normalisations and the proof of \cref{cor:zagier}}\label{sec:normalisation}  Our definition of real period map depends on the choice of arbitrary constants $\lambda_n \in \bb{R}^\times$, see \eqref{eqn:hodge correlator via log abs}. 
By \cref{prop:multiple}, \cref{prop:borel-regulator-multiple}, and \cref{thm:polyl-iso-mun} there exist nonzero constants $\mu_n \in \bb{R}$ such that
\[\rm{reg}^\scr{G}_n = \mu_n\, \rm{reg}^\rm{Bo}_n.\]
The exact value of $\mu_n$ depends on the choice of $\lambda_n$ and only the quotient $\mu_n/\lambda_n \in \bb{R}^\times$ is canonically defined. It measures the proportionality between the two regulators and in this section we determine it up to a nonzero rational number. We finish by establishing \cref{cor:zagier}.

\subsubsection{Borel's regulator(s) and values of the zeta function} \label{sec:various-normaisations}
There are different conventions in the literature about normalisations for the Borel regulator. We will explain three conventions below, as they are more convenient in different settings:
\begin{enumerate}[\noindent 1.]
\item There is the \emph{renormalised Borel regulator} appearing in \cite[Definition 9.19]{BurgosGil}, denoted $\rm{reg}^\rm{Bo}_n$, which in terms of Beilinson's regulator $\rm{reg}^\rm{Be}_n$ satisfies  \cite[Theorem 10.9]{BurgosGil}
\[\rm{reg}^\rm{Bo}_n=2\, \rm{reg}^\rm{Be}_n.\]
\item There is an \emph{unnormalised Borel regulator}, denoted $(\rm{reg}^\rm{Bo}_n)'$, which appears in Borel's original paper \cite{BorelReg}. By \cite[Definition 9.19, Remark 9.20]{BurgosGil} this satisfies
\[\rm{reg}^\rm{Bo}_n \sim_{\Q^\times} \pi\, (\rm{reg}^\rm{Bo}_n)'.\]
It is possible to pin down the rational constant further using the now-proven Quillen--Lichtenbaum conjecture, but we will not have a use for this.
\item There is a final choice of normalisation, which appears in the literature around Zagier's conjecture \cite[Section 1]{Zag90} \cite[Remark 1.3]{Dup20} (see \cref{rem:Zaiger's normalisation} for an explanation). 
This normalisation will be called \emph{Zagier's normalisation} of Borel's regulator and will be denoted $(\rm{reg}^\rm{Bo}_n)''$. By \cite[Note 6]{Dup20}, it satisfies
\[(\rm{reg}^\rm{Bo}_n)'' \sim_{\Q^\times} \pi^{n-1} \rm{reg}^\rm{Bo}_n.\] 
\end{enumerate}

\noindent Fix a number field $F$ and let $\zeta_F(s)$ denote its Dedekind zeta function. For $n \ge 2$, we set 
    \[d_n \coloneq \dim_\Q(K_{2n-1}(F)_\Q)=\begin{cases}
        r_1+r_2 & \text{if $n$ is odd,} \\ r_2 & \text{if $n$ is even,}
    \end{cases}\]
so that the Borel regulator defines a map $\rm{reg}^\rm{Bo}_n \colon K_{2n-1}(F)/\rm{tors} \to \bb{R}^{d_n}$ whose image is a lattice with covolume denoted by $R_F(n)$. Doing the same for the un-normalised Borel regulator we obtain a lattice with covolume $R_F'(n)$ satisfying $R_F(n) \sim_{\Q^\times} \pi^{d_n} R'_F(n)$.

In \cite[Section 6]{BorelReg}, Borel proved that the un-normalised Borel regulator satisfies
\[R_F'(n) \sim_{\Q^\times} \pi^{-d_n} \zeta_F^*(1-n),\]
where $\zeta_F^*(1-n) \coloneq \lim_{s \to 1-n} (s+n-1)^{-d_n} \zeta(s) \in \bb{C}^\times$. Thus, as stated in \cite[Theorem 9.12]{BurgosGil}, the re-normalised Borel regulator satisfies 
\[R_F(n) \sim_{\Q^\times} \zeta_F^*(1-n).\]

Finally, Zagier's Borel regulator has covolume $R''_F(n)$ satisfying
\[\zeta^*_F(1-n) \sim_{\Q^\times} \pi^{(1-n)d_n} R''_F(n).\]

The functional equation for $\zeta_F(s)$ implies that for $n \ge 2$ integer, 
\[\zeta_F(n) \sim_{\Q^\times} \frac{\pi^{n [F:\Q]-d_n}}{\sqrt{|D_F|}} \zeta_F^*(1-n),\]
where $|D_F|$ is the absolute value of the discriminant of $F$,
and hence we get
\begin{equation} \label{eqn: zeta value and regulators}
\zeta_F(n) \sim_{\Q^\times} \frac{\pi^{n [F:\Q]-d_n}}{\sqrt{|D_F|}} R_F(n) \sim_{\Q^\times} \frac{\pi^{n [F:\Q]}}{\sqrt{|D_F|}} R_F'(n) \sim_{\Q^\times} \frac{\pi^{n ([F:\Q]-d_n)}}{\sqrt{|D_F|}} R_F''(n).
\end{equation}

\begin{remark} \label{rem:Zaiger's normalisation}
   Zagier's normalisation of the Borel regulator is chosen to be compatible with the single-valued classical polylogarithms $\mathcal{L}_n$ introduced in \cref{sec: real period classical polylogs} in the sense that Zagier's conjecture is equivalent to the existence of $\xi_1,\compactldots, \xi_{d_n} \in F^\times$ such that
\[R''_F(n)\sim_{\Q^\times} \det[\mathcal{L}_n(\sigma_i(\xi_j))_{1 \le i,j \le d_n}],\]
where $\sigma_1,\compactldots,\sigma_{d_n}$ are either all the conjugacy classes of complex non-real embeddings if $n$ is even, or all the conjugacy classes of complex embeddings for $n$ odd. 
\end{remark}

\subsubsection{Comparing Goncharov's and Borel's regulators} We now compare the Goncharov regulator to Borel's regulators. 

\begin{proposition}\label{prop:proportionality}
    For $n \ge 2$ we have that 
 \[\rm{reg}^{\scr{G}}_n \sim_{\Q^\times} \lambda_n \pi^{2n-2} \rm{reg}^{\rm{Bo}}_n \sim_{\Q^\times}   \lambda_n \pi^{2n-1}(\rm{reg}^{\rm{Bo}}_n)' \sim_{\Q^\times} \lambda_n {\pi^{n-1}}(\rm{reg}^{\rm{Bo}}_n)''.\]
\end{proposition}

\begin{proof}By \cref{sec:various-normaisations}, it suffices to prove $\rm{reg}^{\scr{G}}_n \sim_{\Q^\times} \lambda_n \pi^{2n-2} \rm{reg}^{\rm{Bo}}_n$. By \eqref{eqn: zeta value and regulators} we know that 
    \[\zeta_F(n) \sim_{\Q^\times} \frac{\pi^{n[F: \Q]-d_n}}{\sqrt{|D_F|}} R_F(n).\]
    We now specialize to $F=\Q(i)$ and simplify. Firstly, $\sqrt{|\smash{D_{\Q(i)}}|}=2$ is rational. Secondly, $[\Q(i):\Q]=2$ and $d_n=1$ for all $n \ge 2$. Thirdly, since $d_n=1$ for $n \ge 2$ we have 
    \[R_{\Q(i)}(n)=|\rm{reg}^{\rm{Bo}}_n(\alpha_n)|,\]
    with $\alpha_n$ a generator of $K_{2n-1}(\Q(i))/\rm{tors}$ and Borel regulator using the standard embedding $\Q(i) \hookrightarrow \bb{C}$. Fourthly, the Dedekind zeta function factors as $\zeta_{\Q(i)}(s)= \zeta(s) \cdot \beta(s)$, where $\zeta(s)$ is the Riemann zeta function and $\beta(s)$ is the Dirichlet beta function $\beta(s) \coloneq \sum_{k=0}^{\infty} \smash{\tfrac{(-1)^k}{(2k+1)^s}}$.

    By the proof of \cref{thm:polyl-iso-mun} we know that $\smash{\LiG_n(i) \in \PolyL_n(\Q(i))}$ lifts to a unique element $\beta_n \in K_{2n-1}(\Q(i))_\Q$ since its cobracket vanishes, and cannot be zero since $p_\R(\LiG(i)) \neq 0$ by \cref{cor:value-of-real-period-on-unit-norm}. Thus, $\beta_n \sim_{\Q^\times} \alpha_n$ and we obtain that for $n \ge 2$
    \[\zeta(n) \cdot \beta(n) \sim_{\Q^\times} {\pi^{2n-1}} \, \rm{reg}^\rm{Bo}_n(\beta_n) = \tfrac{\pi^{2n-1}}{\mu_n} \, \rm{reg}^\scr{G}_n(\beta_n)= \tfrac{\pi^{2n-1}}{\mu_n} \, p_\R(\LiG_n(i)).\]
    Now, by \cref{lem: real period of classical polylogarithm}, $p_\R(\LiG_n(i))= \lambda_n (2 \pi i)^{n-1} \mathcal{L}^{\#}(i)$, and
    \[\mathcal{L}^{\#}(i) \sim_{\Q^\times} \begin{cases}
        i \cdot \sum_{m \ge 1} \frac{\sin(m \pi/2)}{m^n} & \text{if $n$ even} \\
        \sum_{m \ge 1} \frac{\cos(m \pi/2)}{m^n} & \text{if $n$ odd}. 
    \end{cases}\]
    For the $n$ even case, $\sum_{m \ge 1} \frac{\sin(m \pi/2)}{m^n}= \beta(n)$ by definition. For the $n$ odd case, we have 
    \[\sum_{m \ge 1}\frac{\cos(m \pi/2)}{m^n}= - \frac{1}{2^n} \sum_{k \ge 1} \frac{(-1)^{k-1}}{k^n}= - \frac{1}{2^n} \Big(1- \frac{2}{2^n}\Big) \zeta(n) \sim_{\Q^\times} \zeta(n).\]
    Thus, we obtain that 
    \[\zeta(n) \cdot \beta(n) \sim_{\Q^\times} \pi^{3n-2} \tfrac{\lambda_n}{\mu_n} \cdot \begin{cases}
        \beta(n) & \text{if $n$ even,} \\
        \zeta(n) & \text{if $n$ odd.}
    \end{cases}\]
    We will analyse both cases separately. When $n$ is even, we see that $\smash{\zeta(n) \sim_{\Q^\times} \pi^{3n-2} \tfrac{\lambda_n}{\mu_n}}$ and the known computation $\smash{\zeta(n) \sim_{\Q^\times} \pi^n}$ yields the result. When $n$ is odd, we see that $\smash{\beta(n) \sim_{\Q^\times} \pi^{3n-2} \tfrac{\lambda_n}{\mu_n}}$, and the known computation $\smash{\beta(n) \sim_{\Q^\times} \pi^n}$ yields the result.
\end{proof}  

\subsubsection{The proof of \cref{cor:zagier}} \label{subsubsection Zagier}

Using Zagier's normalisation, \eqref{eqn: zeta value and regulators} gives that 
\[\zeta_F(n) \sim_{\Q^\times} \frac{\pi^{n ([F:\Q]-d_n)}}{\sqrt{|D_F|}} R_F''(n).\]
We now will take $\lambda_n \sim_{\Q^\times} \pi^{1-n}$ so that $\rm{reg}^\scr{G}_n \sim_{\Q^\times} (\rm{reg}^\rm{Bo}_n)''$ by \cref{prop:proportionality}. The Goncharov regulators can also be used to define a map $ K_{2n-1}(F)/\rm{tors} \to \R^{d_n}$ whose image is a lattice, with covolume denoted $R^\scr{G}_n(F)$. Our choice of normalisation then implies that $R^\scr{G}_n(F) \sim_{\Q^\times} R_n''(F)$. 

\cref{thm:number-fields-goncharov} tells us that $\rm{edge}_n$ induces an isomorphism $K_{2n-1}(F)_\Q \xrightarrow{\cong} \ker(\delta|_{\PolyL_n(F)}) \subset \PolyL_n(F)$, and by definition $\smash{\rm{reg}_n^\scr{G}= p_\R \circ \rm{edge}_n}$. Thus, if $\xi_1,\compactldots,\xi_{d_n}$ is a basis for $\ker(\delta|_{\PolyL_n(F)})$ then we can compute
\[R^\scr{G}_n(F) \sim_{\Q^\times} \det\big(p_\R(\sigma_i(\xi_j))\big)_{1 \le i,j \le d_n},\]
where $\sigma_1,\compactldots,\sigma_{d_n}$ denotes either the set of all conjugacy classes of complex nonreal embeddings of $F$ (if $n$ even) or the set of all conjugacy classes of real and complex embeddings of $F$ (if $n$ is odd). 
This proves \cref{cor:zagier}.

\bibliographystyle{amsalpha}
\bibliography{./refs}

\bigskip

\end{document}